\documentclass{article}

\usepackage[utf8]{inputenc}
\usepackage[hidelinks]{hyperref}
\usepackage{amsmath, amssymb, amsthm, bm, mathrsfs, mathtools}
\usepackage[cal=cm]{mathalpha}
\usepackage{stackengine} 
\DeclareMathAlphabet{\dutchcal}{U}{dutchcal}{m}{n} 
\usepackage{centernot}
\usepackage[top=1.5cm, bottom=1.5cm, left=2cm, right=2cm]{geometry}
\usepackage{csquotes, authblk, cases, enumitem}
\usepackage{graphicx, float, caption, subcaption, adjustbox}
\usepackage{tikz-cd}
\usepackage[numbers,sort]{natbib}
\usepackage[section]{algorithm}
\usepackage{algpseudocode}
\usepackage{array, tabularx, xcolor, colortbl, booktabs}
\usepackage{color, xcolor}
\usepackage[para]{footmisc}
\usepackage[normalem]{ulem}
\counterwithin{figure}{section}
\counterwithin{table}{section}
\numberwithin{equation}{section}
\graphicspath{{Figures/}}
\definecolor{darkred}{rgb}{0.8,0,0}
\definecolor{darkgreen}{rgb}{0,0.8,0}

\DeclareMathOperator{\trace}{\operatorname{trace}}

\DeclareMathOperator{\Span}{\operatorname{span}}
\DeclareMathOperator{\sign}{\operatorname{sign}}

\DeclareMathOperator{\diag}{\operatorname{diag}}

\DeclareMathOperator{\supp}{\operatorname{supp}}

\newcommand\green{\cellcolor{darkgreen!20}}

\newcommand\red{\cellcolor{darkred!20}}
\newcommand\gray{\cellcolor{gray!20}}
\newcommand{\rs}{\mathcal{L}}
\newcommand{\rsc}{\stackinset{c}{-1pt}{c}{0pt}{\small $\circ$}{$\mathcal{L}$}}
\newenvironment{frcseries}{\fontfamily{frc}\selectfont}{}

\usepackage{etoolbox}
\usepackage{appendix}
\usepackage{cleveref}
\AtBeginEnvironment{appendices}{\crefalias{section}{appendix}}

\theoremstyle{definition}
\newtheorem{definition}{Definition}[section]
\newtheorem{remark}[definition]{Remark}
\newtheorem{example}[definition]{Example}
\newtheorem{theorem}[definition]{Theorem}
\newtheorem{lemma}[definition]{Lemma}
\newtheorem{corollary}[definition]{Corollary}
\newtheorem{conjecture}[definition]{Conjecture}

\AtBeginEnvironment{remark}{\crefalias{definition}{remark}}         
\AtBeginEnvironment{example}{\crefalias{definition}{example}}       
\AtBeginEnvironment{theorem}{\crefalias{definition}{theorem}}       
\AtBeginEnvironment{lemma}{\crefalias{definition}{lemma}}           
\AtBeginEnvironment{corollary}{\crefalias{definition}{corollary}}   
\AtBeginEnvironment{conjecture}{\crefalias{definition}{conjecture}} 

\crefname{definition}{Definition}{Definitions}
\crefname{remark}{Remark}{Remarks}
\crefname{example}{Example}{Examples}
\crefname{theorem}{Theorem}{Theorems}
\crefname{lemma}{Lemma}{Lemmas}
\crefname{corollary}{Corollary}{Corollaries}
\crefname{conjecture}{Conjecture}{Conjectures}

\title{On the prospects of interpolatory spline bases for accurate mass lumping strategies in isogeometric analysis}
\author[1]{Yannis Voet \thanks{yannis.voet@epfl.ch}}
\author[2]{Espen Sande \thanks{sande@simula.no}}
\affil[1]{\small MNS, Institute of Mathematics, École Polytechnique Fédérale de Lausanne, Station 8, CH-1015 Lausanne, Switzerland}
\affil[2]{Department of Numerical Analysis and Scientific Computing, Simula Research Laboratory, Oslo, Norway}
\date{\today}

\begin{document}

\maketitle

\begin{abstract}
While interpolatory bases such as the Lagrange basis form the cornerstone of classical finite element methods, they have been replaced in the more general finite element setting of isogeometric analysis in favor of other desirable properties. Yet, interpolation is a key property for devising accurate mass lumping strategies that are ubiquitous in explicit dynamic analyses of structures. In this article, we explore the possibility of restoring interpolation for spline bases within isogeometric analysis for the purpose of mass lumping. Although reminiscent of the spectral element method, this technique comes with its lot of surprises and challenges, which are critically assessed.

\end{abstract}

\section{Introduction}
The finite element method (FEM) is the tool of the trade for approximating the solution of partial differential equations (PDEs) describing countless physical processes in fluid dynamics, heat transfer and wave propagation to name just a few. Nevertheless, the solution process still requires considerable computer resources and manual interventions, incurring significant costs and slowing down the design process in industrial engineering applications. One of the major bottlenecks is attributed to the poor communication between computer-aided-design (CAD) and finite element analysis (FEA) that grew from separate communities and rely on vastly different technologies. Isogeometric analysis (IGA) has promised to unite the two communities by employing smooth spline functions from CAD such as B-splines both for representing the approximate solution and describing the geometry \cite{hughes2005isogeometric, cottrell2009isogeometric}. Apart from streamlining the design process, spline spaces also have vastly superior approximation properties \cite{bazilevs2006isogeometric,bressan2019approximation,sande2019sharp,sande2020explicit}, as proved in \cite{bressan2019approximation} and observed in numerous applications, including fluid dynamics \cite{tagliabue2014isogeometric,nitti2020immersed}, structural mechanics \cite{cottrell2006isogeometric,cottrell2007studies,hughes2008duality,hughes2014finite}, and phase field modeling \cite{borden2014higher,greco2024higher}.

However, some of the early promises of IGA have still not been realized: the CAD and FEA communities are mostly evolving independently of one another, partly due to their distinct needs. Moreover, IGA has not alleviated some of the older issues of traditional FEA. Most notably, the existence of a mass matrix has always bothered structural engineers and for a good reason: explicit time integration of time-dependent PDEs in structural dynamics leads to solving a linear system with the mass matrix at each time step, a problem that was never encountered with finite differences. The repeated solution of those linear systems has long been acknowledged as one of the most expensive steps in the solution process and is in fact further exacerbated in IGA, whether those linear systems are solved directly or iteratively \citep{collier2012cost, collier2013cost}. While small to medium size applications might rely on a matrix factorization, such an approach becomes plainly infeasible for larger problems. Thus, instead of solving those linear systems ``exactly'', practitioners resort to ad hoc approximations, with mass lumping being one of its best known examples. 

Mass lumping has a long history and consists in replacing the mass matrix in the time integration scheme by some diagonal approximation. Common strategies include the row-sum technique \cite{hughes2012finite}, the Hinton-Rock-Zienkiewicz (HRZ) or diagonal scaling technique \cite{hinton1976note} and the nodal quadrature method \citep{fried1975finite, cohen1994higher}, also exploited within the spectral element method (SEM). Although some of them are sometimes equivalent \cite{duczek2019mass}, the nodal quadrature method is the only ``consistent'' lumping technique and constructs a diagonal mass matrix by choosing as finite element nodes the quadrature nodes of the Gauss-Lobatto rule. The method (nearly) preserves the convergence properties of the consistent mass and delivers positive definite lumped mass matrices if all quadrature weights are positive, a condition easily fulfilled in 1D and straightforwardly extended to multiple dimensions for tensor product elements. Specialized techniques have also been developed for more general elements \cite{cohen2001higher}. Apart from constructing diagonal matrices, mass lumping techniques are also praised for increasing the critical time step, also known as the Courant–Friedrichs–Lewy (CFL) condition. Unfortunately, not all of them are applicable to more general bases encountered for instance in IGA \citep{hughes2005isogeometric, cottrell2009isogeometric}. In particular, the nodal quadrature method does not have an immediate counterpart for non-interpolatory spline bases and alternative techniques have been investigated. 

Very soon after introducing IGA, Cottrell et al. \cite{cottrell2009isogeometric} examined the row-sum technique. Its algebraic nature allows applying it to the isogeometric mass matrix and furthermore ensures positive definite lumped mass matrices owing to the positivity of the B-spline basis. Unfortunately though, contrary to SEM, it severely deteriorates the accuracy of the smallest eigenfrequencies, which converge at a reduced second order rate independently of the spline degree. To make matters worse, increasing the spline degree actually deteriorates the constant. Those observations, also confirmed in numerous subsequent articles \cite{anitescu2019isogeometric,nguyen2023towards,voet2023mathematical}, have drawn much attention but a general proof is still lacking. The accuracy of the smallest eigenfrequencies further deteriorates on trimmed geometries \cite{coradello2021accurate,radtke2024analysis,bioli2025theoretical}, where the associated modes may even cause spurious oscillations in the solution \cite{voet2025stabilization,guarino2025stabilization}. Thus, many authors have tried to improve in one way or another the accuracy of the row-sum technique.

In \cite{voet2023mathematical,voet2025mass}, the authors proved that the eigenfrequencies for the row-sum lumped mass always underestimate those for a nonnegative consistent mass, thereby also ensuring a larger critical time step. The authors then constructed a sequence of banded (or block-banded) matrices converging to the consistent mass and monotonically improving the eigenfrequency approximation from below. While those constructions significantly reduced the eigenfrequency error and visibly improved the accuracy, they did not improve the convergence rate. Nevertheless, they are among the most versatile strategies, applicable to single-patch, multi-patch as well as trimmed geometries.

Realizing that the B-spline basis may be inadequate for mass lumping, many authors are turning to different bases for the test and/or trial spaces. Two good examples include approximate $L^2$ dual bases \cite{anitescu2019isogeometric,nguyen2023towards} and interpolatory spline bases \cite{li2022significance,li2024interpolatory}. The former applies the row-sum technique in a Petrov-Galerkin framework by choosing classical (approximate) dual functions (see \cite{chui2004nonstationary} and \cite[Chapter~4.6]{Schumaker2007}) as test functions. Although the idea, whose origins also date back to Cottrell \cite{cottrell2009isogeometric}, initially did not attract much attention, it was taken up again recently in \cite{anitescu2019isogeometric} with promising results. Since then, there has been a surge of interest in approximate dual functions \cite{nguyen2023towards, hiemstra2025higher, held2024efficient, nguyen2023higher}. This strategy produces optimally accurate lumped mass approximations but unfortunately introduces many other complications. In particular, approximate duality only holds in the parametric domain and ensuring it also holds in the physical domain complicates the assembly of the stiffness matrix \cite{nguyen2023towards}. Moreover, imposing essential boundary conditions requires ad hoc techniques \cite{hiemstra2025higher} and extending them to multi-patch or trimmed geometries might be difficult. Finally, Petrov-Galerkin methods lead to non-symmetric system matrices. While losing symmetry is not an issue per se, giving up this property to attain high order accuracy is somewhat surprising. In any case, the dual lumping method remains one of the few instances of high order mass lumping strategies in IGA.

In another line of research, some authors \cite{li2022significance, li2024interpolatory} have recently tried mimicking the nodal quadrature method by resorting to the classical interpolatory spline bases described in, e.g., \cite{Schoenberg1972,schoenberg1973cardinal}. Although the method leads to sub-optimal convergence rates, it still offers great promises, also in developing a theory that parallels SEM. Unfortunately, restoring interpolation comes with its own lot of difficulties, which were neither mentioned in \cite{li2022significance} nor in their follow-up work \cite{li2024interpolatory}. Apart from the evident sub-optimal convergence rates, we have identified three major issues:
\begin{enumerate}[noitemsep]
    \item Positive definite lumped mass matrices are not guaranteed and critically depend on the choice of interpolation points. This problem on its own may already jeopardize the method, since indefinite lumped mass matrices lead to unstable, potentially diverging, solutions.
    \item The classical interpolatory basis functions used in \cite{li2022significance, li2024interpolatory} are globally supported, which translate into dense stiffness matrices and prohibitive storage requirements.
    \item Similarly to classical Lagrange bases, the mass matrix for interpolatory spline bases may have negative off-diagonal entries and thus the impact of mass lumping on the CFL condition is not immediately clear.
\end{enumerate}

To our knowledge, this method is not related to the Lagrange extraction technique proposed in \cite{schillinger2016lagrange,nguyen2017collocated}, where the authors perform an elementwise change of basis to a polynomial Lagrange basis defined on the same Gauss-Lobatto points as in SEM.

To this date, despite significant improvements, the lack of effective mass lumping techniques in IGA remains an open problem. In view of the issues raised above, we further investigate whether interpolatory spline bases are a realistic option. As we will see, while some problems are easily resolved, others are much more serious.

The rest of the article is structured as follows: after introducing our model problem and recalling some basic concepts in \Cref{se: problem_statement}, we review in \Cref{se: mass_lumping} some of the best known mass lumping strategies for classical FEM. As far as we know, some of the results in that section are new and help identify suitable basis properties. \Cref{se: high_order_mass_lumping} is the core of the article and explores how to extend those properties to spline spaces within IGA. After presenting the construction of interpolatory spline bases, we carefully examine the aforementioned issues, one at a time. Theoretical as well as computational issues are discussed and a complete algorithm is presented at the end of this section. Although most of our results are limited to single-patch geometries, some also apply to multi-patch ones. Numerical experiments follow in \Cref{se: numerical_experiments} to validate the theoretical results and provide additional insights. Finally, we state our conclusions in \Cref{se: conclusion} and outline several directions for future research.

\section{Problem statement}
\label{se: problem_statement}
In this article, we approximate the solution of time-dependent PDEs from structural dynamics. The simplest and best known example is the wave equation, which serves as model problem for acoustic, elastic and electromagnetic wave propagation. Let $\Omega \subset \mathbb{R}^d$ be an open connected domain with Lipschitz boundary $\partial \Omega$ and let $[0,T]$ be the time domain with $T>0$ denoting the final time. We look for $u \colon \Omega \times [0,T] \to \mathbb{R}$ such that 
\begin{align}
 \rho(\bm{x})\partial_{tt} u(\bm{x},t)-\nabla \cdot (\kappa(\bm{x})\nabla u(\bm{x},t)) &=f(\bm{x},t) & &\text{ in } \Omega \times (0,T], \label{eq: wave_equation} \\
 u(\bm{x},t)&=0 & &\text{ on } \partial \Omega \times (0,T], \nonumber\\
 u(\bm{x},0)&=u_0(\bm{x}) & &\text{ in } \Omega,  \nonumber\\
 \partial_t u(\bm{x},0)&=v_0(\bm{x}) & &\text{ in } \Omega, \nonumber
\end{align}
where $u_0$ and $v_0$ are initial conditions on the solution and its first time derivative, respectively, and $\rho$ and $\kappa$ are positive-valued coefficient functions, often material-dependent. To simplify the presentation, we only prescribe homogeneous Dirichlet boundary conditions. The Galerkin method seeks an approximate solution $u_h(.,t)$ of $u(.,t)$ in a finite dimensional subspace $V_h$. In the classical FEM, this subspace is a space of continuous polynomials. In more recent developments, Hughes et al. \cite{hughes2005isogeometric} proposed using a spline space; i.e., a smooth space of piecewise polynomials. This choice is at the heart of IGA. In either case, once a basis $\Phi=\{\varphi_1,\dots,\varphi_n\}$ is chosen for $V_h$, discretizing in space the weak form of the PDE leads to solving a system of ordinary differential equations (see for instance \citep{hughes2012finite, quarteroni2009numerical})
\begin{align}
\label{eq: semi_discrete_pb}
\begin{split}
M\ddot{\bm{u}}(t) + K\bm{u}(t) &= \bm{f}(t) \qquad \text{for } t \in [0,T], \\
\bm{u}(0) &= \bm{u}_0,\\
\dot{\bm{u}}(0) &= \bm{v}_0,
\end{split}
\end{align}
where $\bm{u}(t)$ is the coefficient vector of $u_h(.,t)$ in the basis $\Phi$ and the so-called stiffness and mass matrices are given by 
\begin{equation*}
    K_{ij} = a(\varphi_i,\varphi_j) \quad \text{and} \quad M_{ij} = b(\varphi_i,\varphi_j)
\end{equation*}
for the bilinear forms $a,b \colon V_h \times V_h \to \mathbb{R}$ defined as
\begin{equation}
\label{eq: bilinear_forms}
    a(u,v)=\int_{\Omega} \kappa(\bm{x}) \nabla u(\bm{x}) \cdot \nabla v(\bm{x}) \,d\bm{x} \quad \text{and} \quad b(u,v) = \int_{\Omega} \rho(\bm{x}) u(\bm{x}) v(\bm{x}) \,d\bm{x},
\end{equation}
where $u,v \in V_h$. Similarly, the right-hand side $\bm{f}(t)$ is defined as
\begin{equation*}
    f_i(t) = F(\varphi_i)
\end{equation*}
for the linear functional $F \colon V_h \to \mathbb{R}$
\begin{equation*}
    F(v) = \int_{\Omega} f(\bm{x},t) v(\bm{x}) \,d\bm{x},
\end{equation*}
where $v \in V_h$. Regardless of the basis, the stiffness and mass matrices $K$ and $M$ are both symmetric and while $M$ is always positive definite, $K$ is generally only positive semidefinite (unless Dirichlet boundary conditions are prescribed on some portion of the boundary). However, to ensure sparsity, compactly supported basis functions are sought. Well-known examples include the Lagrange basis for classical FEM and the B-spline basis for IGA. Although a host of other choices are possible, also among the spline ``zoo'', we will restrict our discussion to these two choices that form a partition of unity. The Lagrange basis functions are typically constructed over a reference element $\hat{\Omega}$ before being defined on the physical elements and ``glued together'' across element boundaries. More specifically, in 1D, given $p+1$ distinct interpolation points $\{\hat{x}_i\}_{i=0}^p$ in $\hat{\Omega} = [-1,1]$, the Lagrange basis functions over the reference element are defined as
\begin{equation*}
    \hat{\varphi}_j(\hat{x}) = \prod_{\substack{i=0 \\ i \neq j}}^q \frac{(\hat{x}-\hat{x}_i)}{(\hat{x}_j-\hat{x}_i)}.
\end{equation*}
To ensure coupling between elements and the imposition of boundary conditions, we require that $\hat{x}_0=-1$ and $\hat{x}_q=1$. Denoting $F_e \colon \hat{\Omega} \to \Omega_e$ the bijective (affine) map from the reference to the physical element, the basis functions on $\Omega_e$ are defined as $\varphi_i = \hat{\varphi}_i \circ F_e^{-1}$ and are then coupled together across neighboring elements. Although specific 2D and 3D finite elements exist, we will mostly focus on tensor product elements $\hat{\Omega} = [-1,1]^d$, obtained by repeating this construction along separate parametric directions, which leads to quadrilateral elements in 2D and hexahedral elements in 3D. In dimension $d$, tensor product basis functions are defined as
\begin{equation*}
\hat{\varphi}_{\bm{i}}=\hat{\varphi}_{1,i_1}\hat{\varphi}_{2,i_2}\dots \hat{\varphi}_{d,i_d}
\end{equation*}
where $\hat{\varphi}_{j,i}$ denotes the $i$th function in the $j$th direction and $\bm{i}=(i_1,i_2,\dots,i_d)$ is a multi-index. For convenience, multi-indices are often identified with ``linear'' indices in the global numbering and with a slight abuse of notation we write $\hat{\varphi}_i=\hat{\varphi}_{\bm{i}}$. Naturally, separate directions may have different polynomial degrees but for simplicity we will not consider such cases here. Finally, the basis functions are mapped to the physical elements and coupled together just as they are in 1D.

The B-spline basis construction for spline spaces has both similarities and important differences. Similarly to the Lagrange basis functions, the B-spline basis functions follow a standardized construction in a so-called parametric domain $\hat{\Omega}=[0,1]^d$ before being defined in the physical domain $\Omega$. In dimension $d=1$, the B-spline basis $\{\hat{B}_i\}_{i=1}^n$ is constructed recursively from a \emph{knot vector} $\Xi:=(\xi_1,\dots,\xi_{n+p+1})$ forming a sequence of non-decreasing real numbers. The integers $p$ and $n$ denote the spline degree and spline space dimension, respectively. A knot vector is called \emph{open} if
\begin{equation*}
    \xi_1=\dots=\xi_{p+1} < \xi_{p+2} \leq \dots \leq \xi_n < \xi_{n+1} = \dots = \xi_{n+p+1}.
\end{equation*}
Internal knots of multiplicity $1 \leq m \leq p$ lead to $C^{p-m}$ continuous spline spaces. Greater smoothness has many beneficial consequences, including better approximation properties \cite{bazilevs2006isogeometric,bressan2019approximation,sande2020explicit}. In dimension $d \geq 2$, the spline space is again defined as a tensor product of univariate spaces, which all follow a similar construction. In the isogeometric paradigm, the geometry is described by a spline map $F \colon \hat{\Omega} \to \Omega$ from the parametric domain to the physical domain. Geometries described by such a map are called \emph{single-patch} and the basis functions over the physical domain are defined as $B_i = \hat{B}_i \circ F^{-1}$. For complex geometries, dividing the physical domain into $N$ subdomains (or patches) is often inevitable such that
\begin{equation*}
    \Omega = \bigcup_{e=1}^{N} \Omega_e.
\end{equation*}
Each subdomain (or patch) $\Omega_e$ is described by its own map $F_e \colon \hat{\Omega} \to \Omega_e$ and a \emph{multi-patch} geometry is just a collection of patches. The construction of spline spaces over multi-patch geometries is similar to the construction of standard finite element spaces over multiple elements. In other words, the parametric domain in IGA plays the role of a reference element in FEM. However, the spline map is rarely affine and multi-patch IGA is closer to isoparametric FEM. Apart from that, basis functions from different patches in IGA are coupled together just as they are for different elements in FEM. Multi-patch geometries are typically only $C^0$ and ensuring greater inter-patch smoothness proves at least as difficult as ensuring greater inter-element smoothness for classical finite elements; see e.g. \cite{collin2016analysis,dornisch2017dual,kapl2019isogeometric,dornisch2021isogeometric,kapl2021family} for various $C^1$ multi-patch constructions in IGA. Be aware that in the IGA literature, elements are typically defined as knot spans $[\xi_{i}, \, \xi_{i+1}]$ with $\xi_{i} \neq \xi_{i+1}$, instead of patches. However, the aforementioned analogies suggest that patches can equally be viewed as an immediate counterpart of elements. This point of view will have important implications later in the article.

Now that we have explained some of the key differences between FEM and IGA, we turn to the solution of the semi-discrete problem \eqref{eq: semi_discrete_pb}. Without knowing the origin of the system, the inventory of all the methods available for solving it would be a long one.
Fortunately, for applications in structural dynamics, the choice narrows significantly, not only because of the properties of the methods but also for efficiency reasons. In particular, for fast transient (nonlinear) processes such as car-crash simulations \cite{leidinger2019explicit} and metal stamping \cite{hartmann2015mass}, the community has overwhelmingly adopted explicit integrators. Although it means giving up on unconditional stability, wave propagation problems already require relatively small step sizes. Moreover, explicit time integration permits colossal savings, both in terms of memory and floating point operations. The main reason is that explicit methods applied to undamped systems \emph{only} require solving linear systems with the mass matrix (also for nonlinear PDEs) and the latter is commonly substituted with an ad hoc diagonal approximation, a device widely known as \emph{mass lumping}. In addition to avoiding costly matrix factorizations or iterative solution procedures, it also often increases the critical time step \cite{voet2023mathematical}.

However, lumping the mass matrix generally comes with a loss of accuracy, the extent of which depends on the method. In classical FEA, exceedingly good (near-optimal) mass lumping techniques are known. The most successful instance is undoubtedly the nodal quadrature method within SEM, which sometimes also connects to more algebraic techniques such as the row-sum \cite{duczek2019mass}. Unfortunately, some of those techniques do not have an immediate counterpart for IGA and even if they do, applying them often causes a staggering loss of accuracy. This is even more surprising given the edge IGA initially took over classical FEM in structural vibrations \cite{cottrell2006isogeometric,cottrell2007studies,hughes2008duality}. Although some engineering applications might not require stringent accuracy, others such as structural acoustics are more sensitive to it \cite{radtke2024analysis}. Despite intensive research over the last couple of decades, a solution combining the simplicity and efficiency of SEM is still desperately sought. We believe that the shortcomings of the row-sum in IGA are not specific to the B-spline basis per se but to nonnegative bases more generally. In the next section, we review classical mass lumping techniques in our quest of finding desirable basis properties.

\section{Review of mass lumping}
\label{se: mass_lumping}
Before defining mass lumping techniques for IGA, we must understand why they thrive for classical FEM. Thus, we consider in this section a standard finite element discretization of $\Omega$ with $C^0$ finite element spaces and interpolatory Lagrange basis functions, as explained in \Cref{se: problem_statement}. Let $N$ denote the number of elements. For simplicity, we assume that the mesh consists of a single type of element with $m$ nodes. This assumption is merely for notational convenience and relaxing it does not cause any difficulties. Classical FEM follows an elementwise assembly procedure by first computing local element matrices and later assembling them into global matrices \cite{hughes2012finite,voet2023fast}. Mass lumping is often defined locally by altering the element mass matrices. Three popular mass lumping techniques are reviewed in this section in an attempt to identify desirable properties for spline functions to later mimic. Hereafter, we use the Loewner order on symmetric matrices and write $A \succeq B$ (resp. $A \succ B$) to indicate that $A-B$ is positive semidefinite (resp. positive definite).

\subsection{Row-sum technique}
\label{se: row_sum}
The row-sum technique is undoubtedly the simplest to implement. Given a matrix $M \in \mathbb{R}^{n \times n}$, the lumping operator $\rsc \colon \mathbb{R}^{n \times n} \to \mathbb{R}^{n \times n}$ is defined algebraically as
\begin{equation*}
\rsc(M)=\diag(d_1,\dots,d_n)
\end{equation*}
where $d_i=\sum_{j=1}^n m_{ij}$ for $i=1,\dots,n$.

One may easily show that lumping the global mass matrix is equivalent to lumping all element mass matrices before assembling them into a global (diagonal) matrix. Moreover, from the definition of the consistent mass and the partition on unity property of the basis, we immediately deduce that
\begin{equation*}
    d_i = \int_{\Omega} \rho(\bm{x})\varphi_i(\bm{x}) \,d\bm{x},
\end{equation*}
showing that the row-sum lumped mass is clearly basis-dependent.

\begin{remark}
In \cite{voet2023mathematical}, the lumping operator $\rs$ was defined as the \emph{absolute} row-sum. This definition ensures that $\rs(M)$ remains positive definite for a consistent mass $M$ and $\rs(M) \succeq M$, thereby guaranteeing an improvement of the CFL condition within explicit time integration schemes \cite[Corollary 3.10]{voet2023mathematical}. In contrast, none of those properties are guaranteed for the standard row-sum $\rsc(M)$ and stability must then be studied on a case-by-case basis. However, when it comes to accuracy, this definition might yield a smaller consistency error and is a natural choice for high order techniques. Obviously, for nonnegative matrices, the two definitions coincide.
\end{remark}

Owing to its simplicity, the row-sum technique is a popular choice, unless $\rsc(M)$ is indefinite. This shortcoming was the main reason for introducing the diagonal scaling method, which we describe next.

\subsection{Diagonal scaling}
\label{se: diagonal_scaling}
The diagonal scaling method \cite{hinton1976note}, also known as HRZ (from its creators Hinton, Rock and Zienkiewicz) and referred to as the ``special lumping technique'' in \cite{hughes2012finite}, is an ad hoc technique guaranteeing positive definite lumped mass matrices. As with many mass lumping techniques for classical FEM, it follows an elementwise construction, where, as the name suggests, the element lumped mass matrix is simply a rescaling of the diagonal of the element consistent mass. Denoting $D_e = \diag(M_e)$ the diagonal matrix formed from the diagonal of $M_e$, the element lumped mass matrix is defined as $\overline{M}_e = \beta_e D_e$ where
\begin{equation*}
    \beta_e = \frac{\int_{\Omega_e} \rho(\bm{x}) \,d\bm{x}}{\trace(D_e)} > 0.
\end{equation*}
By construction, $\overline{M}_e$ is positive definite and since $\trace(\overline{M}_e)=\int_{\Omega_e} \rho(\bm{x}) \,d\bm{x}$, it also ``preserves the mass''. 
Many authors have concluded that the diagonal scaling method was well-suited for low-order finite elements but quickly became quite inaccurate for higher orders \cite{malkus1988reversed,duczek2019mass,duczek2019critical}. Although extending it to IGA is rather straightforward, its poor performance for classical high order FEM presages the same fate for IGA. We must therefore turn to the last and most promising technique.

\subsection{Nodal quadrature}
\label{se: nodal_quadrature}
The nodal quadrature method is often described (and rightly so) as a form of consistent mass lumping. The method simply consists in choosing as element nodes the quadrature points of an accurate quadrature rule. Integrating the mass matrix with that same quadrature rule then naturally leads to a diagonal matrix. Since inter-element compatibility constraints and the imposition of boundary conditions require that nodes be placed on the element boundaries, the Gauss-Lobatto rule is the optimal choice. Denoting $\{\hat{\bm{x}}_k, \hat{w}_k\}_{k=1}^m$ the pairs of quadrature nodes/weights for the Gauss-Lobatto rule on the reference tensor product element $\hat{\Omega}=[-1,1]^d$, the bilinear forms $b_e,\widehat{b}_e \colon \mathbb{P}_d \times \mathbb{P}_d \to \mathbb{R}$ corresponding to the element consistent and lumped mass matrices, respectively, are defined as
\begin{equation*}
    b_e(u,v) = \int_{\Omega_e} \rho(\bm{x}) u(\bm{x})v(\bm{x}) \,d\bm{x} \quad \text{and} \quad \widehat{b}_e(u,v) = \sum_{k=1}^m w_k u(\bm{x}_k)v(\bm{x}_k),
\end{equation*}
where $\bm{x}_k = F_e(\hat{\bm{x}}_k)$, $w_k = \hat{w}_k \rho(F_e(\hat{\bm{x}}_k)) |\det(J_e(\hat{\bm{x}}_k))|$, $F_e \colon \hat{\Omega} \to \Omega_e$ is the mapping from the reference to the physical element and $J_e$ denotes its Jacobian matrix. Denoting $\Phi_e = \{\varphi_1,\dots,\varphi_m\}$ the Lagrange basis functions interpolating at the quadrature nodes $\bm{x}_k$, we immediately deduce that
\begin{equation*}
    (M_e)_{ij} = b_e(\varphi_i,\varphi_j) = \int_{\Omega_e} \rho(\bm{x}) \varphi_i(\bm{x})\varphi_j(\bm{x}) \,d\bm{x} \quad \text{and} \quad (\widehat{M}_e)_{ij} = \widehat{b}_e(\varphi_i,\varphi_j) = w_i \delta_{ij},
\end{equation*}
where $\delta_{ij}$ is the Kronecker delta defined as $\delta_{ij}=1$ if $i=j$ and zero otherwise. As shown in \cite{duczek2019mass}, if the row-sum lumped mass matrix $\rsc(M_e)$ is integrated with the same quadrature rule, $\rsc(M_e)=\widehat{M}_e$ and the two lumping strategies coincide. Assembling the element matrices $\{\widehat{M}_e\}_{e=1}^{N}$ in the usual manner then leads to a diagonal matrix $\widehat{M}$ by construction.

The nodal quadrature method is undoubtedly one of the most successful mass lumping strategies, often yielding (near-)optimal accuracy \cite{fried1975finite,cohen1994higher} and supported by rigorous convergence proofs using the Strang lemma \cite{ciarlet2002finite}. Another outstanding property, widely reported in the literature, is that the nodal quadrature method also increases the critical time step \cite{duczek2019mass,radtke2024analysis}. However, since $M_e$ features negative entries, this property does not immediately follow from results like those in \cite{voet2023mathematical} and despite being well-known to the engineering community, we could not find a general proof in existing literature. The next lemma provides the coveted result under some assumptions. Its non-trivial proof is deferred to \Cref{se: proof_CFL_SEM}.

\begin{lemma}
\label{lem: spectral_fem}
For elementwise constant density and affine tensor product spectral elements, $\widehat{M} \succeq M$ and
\begin{equation*}
    \lambda_k(K,\widehat{M}) \leq \lambda_k(K,M).
\end{equation*}
\end{lemma}

Unfortunately, the nodal quadrature method remains indeed mostly limited to tensor product spectral elements, for which positive quadrature weights are guaranteed. Indeed, for SEM, the row-sum is equivalent to approximating the consistent mass with the Gauss-Lobatto rule while for uniformly spaced interpolation nodes, it becomes equivalent to the Newton-Cotes rule, yielding negative weights for high degrees. Negative weights have disastrous consequences and, as we will see, the same issue resurfaces for spline functions. Yet, the method still possesses many desirable properties. In the next section, we will try to extend those properties to spline spaces.

\section{High order mass lumping}
\label{se: high_order_mass_lumping}
\subsection{Interpolatory spline bases}
\label{se: interpolatory_bases}
Let $\mathbb{S}$ be an $n$-dimensional spline space with associated B-spline basis $\dutchcal{B}=\{B_1,\dots,B_n\}$. Similarly to Lagrange polynomials, we would like to construct an interpolatory spline basis $\dutchcal{L}=\{L_1,\dots,L_n\}$ for $\mathbb{S}$ from a set $\{\bm{x}_i\}_{i=1}^n \subset \mathbb{R}^d$ of distinct interpolation points, as in \cite{Schoenberg1972,schoenberg1973cardinal}. Similarly to classical finite elements, those points are defined as $\bm{x}_i = F(\hat{\bm{x}}_i)$, where $\{\hat{\bm{x}}_i\}_{i=1}^n \subset \hat{\Omega} = [0, 1]^d$ are the interpolation points in the parametric domain $\hat{\Omega}$ and $F \colon \hat{\Omega} \to \Omega$ is the spline map from the parametric to the physical domain, described as a single patch. Since $\mathbb{S}=\Span(\dutchcal{B})=\Span(\dutchcal{L})$, we seek coefficients $c_{kj}$ such that
\begin{equation}
\label{eq: B_to_L}
    L_j(\bm{x})=\sum_{k=1}^n c_{kj}B_k(\bm{x}) \qquad j=1,\dots,n.
\end{equation}
The interpolation conditions $L_j(\bm{x}_i)=\delta_{ij}$ for $i,j=1,\dots,n$ lead to the matrix equation $AC=I$, where
\begin{equation}
\label{eq: def_matrices}
    A=
    \begin{pmatrix}
        B_1(\bm{x}_1) & \hdots & B_n(\bm{x}_1) \\
        \vdots & \ddots & \vdots \\
        B_1(\bm{x}_n) & \hdots & B_n(\bm{x}_n)
    \end{pmatrix},
    \quad C=
    \begin{pmatrix}
        c_{11} & \hdots & c_{1n} \\
        \vdots & \ddots & \vdots \\
        c_{n1} & \hdots & c_{nn}
    \end{pmatrix}.
\end{equation}
The coefficient matrix $C$ is uniquely defined provided the collocation matrix $A$ is invertible, which is the case if the distinct interpolation points $\{\bm{x}_i\}_{i=1}^n$ satisfy a generalized form of the Schoenberg-Whitney theorem.

\begin{theorem}[Schoenberg-Whitney theorem]
\label{th: schoenberg_whitney}
For distinct interpolation points $\{\bm{x}_i\}_{i=1}^n$, the collocation matrix $A$ in \eqref{eq: def_matrices} is invertible if and only if
\begin{equation*}
    B_i(\bm{x}_i)>0 \qquad i=1,\dots,n.
\end{equation*}
\end{theorem}
\begin{proof}
The result for $d=1$ is the classical statement of the Schoenberg-Whitney theorem and is well-known (see e.g. \citep[Theorem 10.6]{lyche2018spline}). The statement for arbitrary dimension $d$ is less common but is a natural extension to multivariate spline interpolation (see e.g. \cite{floater2023introduction} for a slightly different but equivalent form). Using the identification of linear and multi-indices,
\begin{equation*}
    A_{ij}=B_j(\bm{x}_i)=B_j(F(\hat{\bm{x}}_i))=\hat{B}_j(\hat{\bm{x}}_i) = \hat{B}_{\bm{j}}(\hat{\bm{x}}_{\bm{i}})= \prod_{k=1}^d \hat{B}_{k,j_k}(\hat{x}_{i_k}) = \prod_{k=1}^d (A_k)_{i_k j_k} = (\bigotimes_{k=1}^d A_k)_{ij}
\end{equation*}
where $(A_k)_{ij}=\hat{B}_{k,j}(\hat{x}_{i})$. Hence, $A=\bigotimes_{k=1}^d A_k$ is a Kronecker product. Thus, $A$ is invertible if and only if all factor matrices $A_k$ are invertible, which, from the Schoenberg-Whitney theorem for $d=1$, is equivalent to
\begin{equation*}
    B_i(\bm{x}_i)=\prod_{k=1}^d \hat{B}_{k,i_k}(\hat{x}_{i_k})>0 \qquad i=1,\dots,n.
\end{equation*}
\end{proof}

\begin{remark}
For a multi-dimensional problem, solving linear systems with the collocation matrix (or its transpose) or computing its inverse naturally leverages its Kronecker structure and merely requires performing the same operations on the factor matrices \cite{golub2013matrix}. Those operations are even cheaper given the banded nature of the 1D collocation matrices \cite{lyche2018spline} and the absence of pivoting \cite{morken1996total}. Thus, solving linear systems with $A$ or computing $C=A^{-1}$ is perfectly affordable.
\end{remark}

\Cref{th: schoenberg_whitney} simply states that the diagonal of $A$ must be positive, or, equivalently, $\bm{x}_i \in \supp(B_i)$ for $i=1,\dots,n$. In analogy to classical interpolation, a set of distinct interpolation points that satisfies this condition is called \emph{unisolvent}. Moreover, for imposing boundary conditions and coupling patches, interpolation points are also placed at the boundaries of patches, just as they were for finite elements. In the sequel, we will always assume those conditions are satisfied without necessarily specifying any points. Concrete examples will come later. Clearly, those conditions are not very restrictive and still leave considerable freedom for choosing the interpolation points. Since our goal is to approximate the mass matrix, we want to choose them as quadrature points. This theory parallels the one developed for SEM, where the points $\{\bm{x}_i\}_{i=1}^n$ both serve as interpolation and quadrature points (see \Cref{se: nodal_quadrature}). The concepts straightforwardly extend to spline functions. Given $n$ interpolation (or quadrature) points $\{\bm{x}_i\}_{i=1}^n$, the corresponding weights are found by requiring that the quadrature rule exactly integrates all (weighted) splines in $\mathbb{S}$:
\begin{equation}
\label{eq: integration}
\int_\Omega \rho(\bm{x}) s(\bm{x}) \,d\bm{x} = \sum_{k=1}^n w_k s(\bm{x}_k) \qquad \forall s \in \mathbb{S}.
\end{equation}
Imposing \eqref{eq: integration} on all basis functions in $\dutchcal{B}$ leads to the so-called moment-fitting equations
\begin{equation}
\label{eq: moment_fitting}
    A^T \bm{w} = \bm{b}
\end{equation}
where
\begin{equation*}
    A^T =
    \begin{pmatrix}
        B_1(\bm{x}_1) & \hdots & B_1(\bm{x}_n) \\
        \vdots & \ddots & \vdots \\
        B_n(\bm{x}_1) & \hdots & B_n(\bm{x}_n)
    \end{pmatrix},
    \quad
    \bm{w} =
    \begin{pmatrix}
        w_1 \\
        \vdots \\
        w_n
    \end{pmatrix}
    \quad \text{and} \quad
    \bm{b} =
    \begin{pmatrix}
        \int_\Omega \rho(\bm{x}) B_1(\bm{x}) \,d\bm{x} \\
        \vdots \\
        \int_\Omega \rho(\bm{x}) B_n(\bm{x}) \,d\bm{x}
    \end{pmatrix}.
\end{equation*}
Since the collocation matrix becomes the identity for the Lagrange spline basis $\dutchcal{L}$, \eqref{eq: integration} also leads to 
\begin{equation}
\label{eq: quadrature_weights}
    w_i = \int_\Omega \rho(\bm{x}) L_i(\bm{x}) \,d\bm{x} \qquad i=1,\dots,n.
\end{equation}
Once the quadrature weights have been computed, the resulting quadrature rule allows approximating integrals. Given a continuous function $f \in C^0(\Omega)$, the integral and quadrature operators are defined as
\begin{equation}
\label{eq: operators}
I(f) = \int_\Omega \rho(\bm{x}) f(\bm{x}) \,d\bm{x} \quad \text{and} \quad Q(f) = \sum_{k=1}^n w_k f(\bm{x}_k),
\end{equation}
respectively, where $\{\bm{x}_i\}_{i=1}^n$ are the quadrature nodes and $\{w_i\}_{i=1}^n$ are the quadrature weights obtained through \eqref{eq: moment_fitting}. Although inexact, the quadrature rule introduced in \eqref{eq: operators} may also approximate the bilinear form $b$ in \eqref{eq: bilinear_forms} and we define $\widehat{b} \colon \mathbb{S} \times \mathbb{S} \to \mathbb{R}$ such that
\begin{equation*}
    \widehat{b}(u,v) = \sum_{k=1}^n w_k u(\bm{x}_k)v(\bm{x}_k).
\end{equation*}
The bilinear form $\widehat{b}$ will generally differ from $b$. As a matter of fact, for $\widehat{b}$ to exactly reproduce $b$, the quadrature rule would have to exactly integrate the (weighted) square of splines; i.e. $Q(s^2)=I(s^2)$ for all $s \in \mathbb{S}$. Unfortunately, it only integrates the (weighted) splines themselves; i.e. $Q(s)=I(s)$ for all $s \in \mathbb{S}$. In the sequel, we set $\rho(\bm{x})=1$ to simplify the expressions but all the results carry over to the weighted case with very minor adjustments. We next define the mass matrices for the exact and approximate bilinear forms $b$ and $\widehat{b}$.

\begin{definition}[Mass matrices]
\label{def: mass_matrices}
Let $\Phi=\{\varphi_1, \dots, \varphi_n\}$ be an arbitrary spline basis for the $n$-dimensional spline space $\mathbb{S}$. Then, we denote
\begin{equation*}
    (M_\Phi)_{ij} = b(\varphi_i,\varphi_j) \quad \text{and} \quad (\widehat{M}_\Phi)_{ij} = \widehat{b}(\varphi_i,\varphi_j)
\end{equation*}
the consistent and approximate mass matrices with respect to the basis $\Phi$.
\end{definition}

\begin{definition}[Lumped mass matrix]
\label{def: lumped_mass}
For an arbitrary basis $\Phi$ of $\mathbb{S}$ we denote
\begin{equation*}
    \widetilde{M}_\Phi = \rsc(M_\Phi)
\end{equation*}
the lumped mass matrix with respect to $\Phi$.
\end{definition}
Since the mass matrices introduced in \Cref{def: mass_matrices} are induced by bilinear forms, their expression in different bases are related through congruence transformations. More specifically, if $id \colon \mathbb{S} \to \mathbb{S}$ is the identity endomorphism and $P=(id)_\Phi^\Psi$ denotes the change of basis matrix between two bases $\Phi$ and $\Psi$, then
\begin{equation}
\label{eq: basis_change}
    M_\Phi = P^T M_\Psi P \quad \text{and} \quad \widehat{M}_\Phi = P^T \widehat{M}_\Psi P.
\end{equation}

The next lemma is reminiscent of SEM.

\begin{lemma}
\label{lem: LML_QML}
For the Lagrange spline basis $\dutchcal{L}$
\begin{equation*}
    \widetilde{M}_\dutchcal{L}=\widehat{M}_\dutchcal{L}=\diag(w_1,\dots,w_n).
\end{equation*}
\end{lemma}
\begin{proof}
First note that the Lagrange spline basis forms a partition of unity; i.e., $\sum_{j=1}^n L_j(x)=1$. Indeed, since the B-spline basis is a partition of unity, $A\bm{e}=\bm{e}$, where $\bm{e}$ is the vector of all ones. Therefore, $(1,\bm{e})$ is an eigenpair of $A$ and is thus also an eigenpair of $C=A^{-1}$, meaning that $\sum_{j=1}^n c_{kj}=1$ for all $k=1,\dots,n$. Consequently,
\begin{equation*}
    \sum_{j=1}^n L_j(x) = \sum_{j=1}^n \sum_{k=1}^n c_{kj}B_k(x) = \sum_{k=1}^n B_k(x)=1.
\end{equation*}
With this result in mind, the statement of the lemma easily follows, since, on the one hand,
\begin{equation*}
    (\widetilde{M}_\dutchcal{L})_{ii} = \sum_{j=1}^n \int_{\Omega} L_i(x)L_j(x) \,dx = \int_{\Omega} L_i(x) \sum_{j=1}^n L_j(x) \,dx = \int_{\Omega} L_i(x) \,dx = w_i,
\end{equation*}
and on the other hand,
\begin{equation*}
    (\widehat{M}_\dutchcal{L})_{ij} = \sum_{k=1}^n w_k L_i(x_k)L_j(x_k) = w_i \delta_{ij}.
\end{equation*}
Consequently, $\widetilde{M}_\dutchcal{L}=\widehat{M}_\dutchcal{L}=\diag(w_1,\dots,w_n)$.
\end{proof}

Hence, in analogy to SEM, for interpolatory spline bases, lumping the mass matrix with the row-sum technique is equivalent to applying a quadrature rule built on the same interpolation points. However, for more general bases such as the B-spline basis, $\widetilde{M}_\dutchcal{B} \neq \widehat{M}_\dutchcal{B}$. The different approximations are nevertheless connected, as shown in the next corollary.

\begin{corollary}
\label{cor: LMB_QMB}
If $\mathcal{Q}$, $\rsc$ and $\mathcal{C}$ identify the quadrature, lumping and change of basis, respectively, then the following diagram commutes:
\begin{equation*}
\begin{tikzcd}
M_\dutchcal{B} \arrow{r}{\mathcal{Q}} \arrow[swap]{d}{\mathcal{C}} & \widehat{M}_\dutchcal{B} \arrow{d}{\mathcal{C}} \\%
 M_\dutchcal{L} \arrow{r}{\stackinset{c}{-0.75pt}{c}{0pt}{\scriptsize $\circ$}{$\mathcal{L}$}} & \widetilde{M}_\dutchcal{L}
\end{tikzcd}  
\end{equation*}
\end{corollary}
\begin{proof}
From \eqref{eq: B_to_L} we deduce that $C=(id)_\dutchcal{L}^\dutchcal{B}$ and consequently \eqref{eq: basis_change} yields
\begin{equation*}
    M_\dutchcal{L} = C^T M_\dutchcal{B} C \quad \text{and} \quad \widehat{M}_\dutchcal{L} = C^T \widehat{M}_\dutchcal{B} C.
\end{equation*}
After combining those relations with \Cref{lem: LML_QML}, we obtain
\begin{equation*}
    \rsc(C^T M_\dutchcal{B} C) = \rsc(M_\dutchcal{L}) = \widetilde{M}_\dutchcal{L} = \widehat{M}_\dutchcal{L} = C^T \widehat{M}_\dutchcal{B} C,
\end{equation*}
which proves the statement of the corollary.
\end{proof}

In other words, applying the quadrature rule to the consistent mass in the B-spline basis and changing to the interpolatory basis is identical to first changing the basis and then approximating the consistent mass in the interpolatory basis with the row-sum technique. This result holds regardless of problem dimension and mesh distortion for two main reasons. Firstly, \Cref{lem: LML_QML} assumes a partition of unity basis, which also implies that the mass matrix is lumped before handling the Dirichlet boundary conditions. Secondly, the lumped mass matrix $\widetilde{M}_\dutchcal{L}$ exactly contains the moments $w_i = \int_\Omega L_i(\bm{x}) \,d\bm{x}$ along its diagonal. Previous authors have assumed those quantities were only approximated \cite{duczek2019mass}. However, the integrals can be approximated to arbitrary precision with potentially different quadrature rules and should not be a concern.

In summary, \Cref{cor: LMB_QMB} effectively connects interpolation, mass lumping and quadrature and provides a pathway towards high order mass lumping strategies. The accuracy of mass lumping schemes revolves around the spectrum of the matrix pair formed by the stiffness and mass matrices, denoted $\Lambda(K_\Phi, M_\Phi)$. Here, similarly to the mass matrix, $K_\Phi$ denotes the stiffness matrix expressed in a basis $\Phi$. Thanks to \Cref{cor: LMB_QMB}, we immediately deduce the following result.

\begin{corollary}
\label{cor: spectrum}
\begin{equation*}
    \Lambda(K_\dutchcal{B},\widehat{M}_\dutchcal{B})=\Lambda(K_\dutchcal{L}, \widetilde{M}_\dutchcal{L}).
\end{equation*}
\end{corollary}
\begin{proof}
The proof is an immediate consequence of \Cref{cor: LMB_QMB} and \cite[Theorem VI.1.8]{stewart1990matrix}:
\begin{equation*}
\Lambda(K_\dutchcal{B},\widehat{M}_\dutchcal{B})=\Lambda(C^TK_\dutchcal{B}C,C^T\widehat{M}_\dutchcal{B}C)=\Lambda(K_\dutchcal{L}, \widetilde{M}_\dutchcal{L}).
\end{equation*}
\end{proof}
Thus, the accuracy of mass lumping in the interpolatory basis is tied to the accuracy of the quadrature rule. This better explains the experimental results in \cite{li2022significance,li2024interpolatory} and from now on, we will instead focus on analyzing the properties of the quadrature rule. For interpolatory quadrature rules, we assume that $n$ quadrature (or interpolation) points $\{\bm{x}_k\}_{k=1}^n$ are given and then compute the quadrature weights $\{w_k\}_{k=1}^n$ by solving the moment-fitting equation \eqref{eq: moment_fitting}. There are two classical choices of interpolation points:
\begin{itemize}
    \item The Greville abscissa \cite{gordon1974b}, also known as knot averages, are defined as
    \begin{equation*}
    \label{eq: greville_pts}
        \hat{x}_i = \frac{1}{p}\sum_{j=1}^p \xi_{i+j} \qquad i=1,\dots,n
    \end{equation*}
    where $\Xi = \{\xi_1,\dots,\xi_{n+p+1}\}$ is the knot vector for an $n$-dimensional (univariate) spline space of degree $p$. 
    \item The Demko points (also sometimes called Chebyshev-Demko points) are implicitly defined as the extrema abscissa of an equioscillating (Chebyshev) spline \cite{demko1985existence,morken1984two}. More specifically, there exists a spline $s \in \mathbb{S}$ and unique points $\{\hat{x}_i\}_{i=1}^n$ \cite{smith1988knots} such that
    \begin{align}
        s(\hat{x}_i)&=(-1)^{i+1}, \label{eq: equioscillation} \\
        \|s\|_\infty &= 1. \label{eq: max_norm}
    \end{align}
    The second condition ensures that the Chebyshev spline neither undershoots nor overshoots between the Demko points, contrary to other choices of interpolation points that might as well satisfy the first condition. The Demko points and their associated Chebyshev spline are exemplarily shown in \Cref{fig: Chebyshev_spline} for a non-uniform knot vector. In practice, the Demko points are computed iteratively, starting from the Greville abscissa and successively choosing as next iterate the extrema abscissa of the oscillating spline satisfying \eqref{eq: equioscillation} until \eqref{eq: max_norm} is finally satisfied (up to a tolerance) \cite{smith1988knots}. For Matlab users, the algorithm is part of the Curve Fitting Toolbox.

    \begin{figure}[H]
    \centering
    \includegraphics[width=0.5\linewidth]{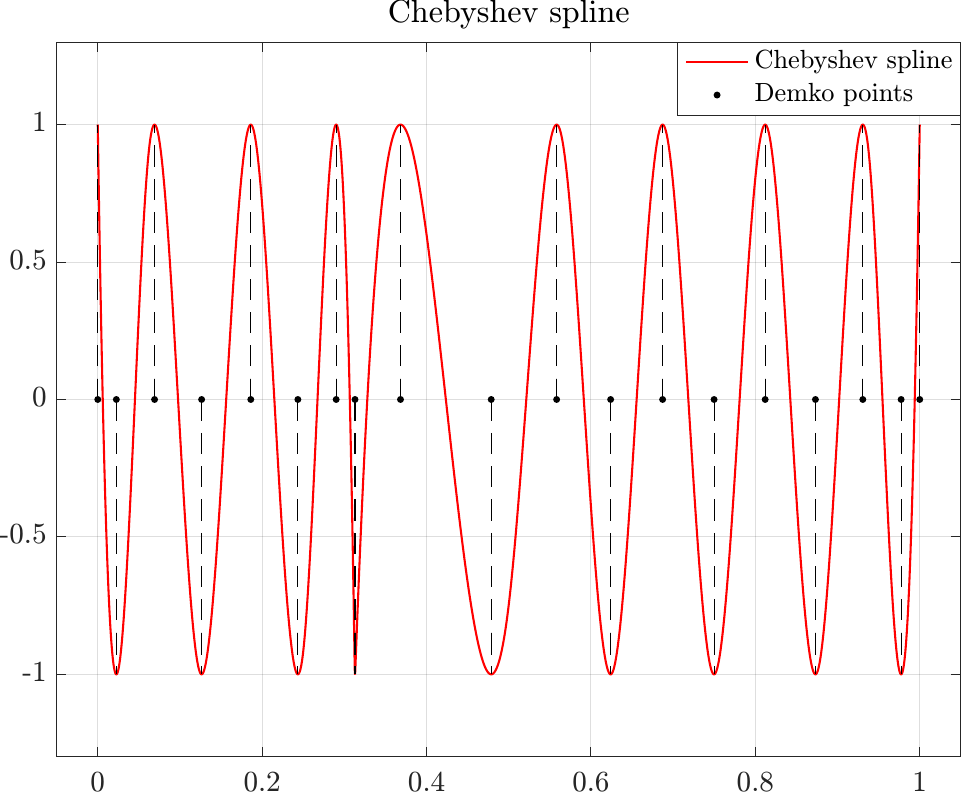}
    \caption{Cubic $C^2$ Chebyshev spline for a non-uniform knot vector}
    \label{fig: Chebyshev_spline}
    \end{figure}
\end{itemize}
In addition to the above choices there are several works on finding improved quadrature rules for spline spaces, both classical \cite{micchelli1977moment,nikolov1996certain} and in the context of IGA \cite{hughes2010efficient,fahrendorf2018reduced,calabro2019quadrature}.

As usual, interpolation points for multivariate spaces are defined as the tensor product of univariate ones, before being mapped to the physical domain $\Omega$. Finally, the quadrature weights are obtained by solving \eqref{eq: moment_fitting} and are not necessarily a tensor product of univariate weights, unless the geometry mapping and density function are separable. The Greville and Demko points are extensively used for isogeometric collocation, where they often perform equally well \cite{auricchio2010isogeometric,evans2018explicit}. The search for better collocation points has also led to using superconvergent points for the considered differential equation; see e.g.
\cite{wahlbin2006superconvergence,anitescu2015isogeometric,gomez2016variational,montardini2017optimal}. While the Greville and Demko points perform equally well for collocation, major differences appear when it comes to interpolation and quadrature. In particular, the weights associated to prescribed quadrature points are not necessarily positive. Negative weights are already a source of anxiety for quadrature and, as we will see in the next section, their consequences for mass lumping are even more dramatic.

\subsection{Negative weights}
\label{se: negative_weights}
Unfortunately, although the system matrix and right-hand side of \eqref{eq: moment_fitting} are nonnegative, the solution may have negative entries. As a matter of fact, for the Greville abscissa, negative weights already emerge for quartic $C^1$ spline spaces on uniform meshes \cite{zou2021galerkin} and abound as the spline degree increases, as shown in \Cref{tab: test_conditions}. Interestingly, the sign pattern appearing for $C^0$ B-splines is exactly the same as for the Newton-Cotes quadrature rule. This is not surprising given that the Greville abscissa for $C^0$ spline discretizations on uniform meshes are uniformly spaced. Although negative weights apparently do not appear for maximally smooth discretizations on uniform meshes, they do eventually appear on non-uniform ones, as we found out in a series of counter-examples reported in \Cref{se: non_uniform_meshes}.

\begin{table}[H]
\centering
\begin{tabular}{lllllllllllllllllllll}
$k$\textbackslash $p$ & 1 & 2 & 3 & 4 & 5 & 6 & 7 & 8 & 9 & 10 & 11 & 12 & 13 & 14 & 15 & 16 & 17 & 18 & 19 & 20 \\ 
\hline 
0 & \green & \green & \green & \green & \green & \green & \green & \red & \green & \red & \red & \red & \red & \red & \red & \red & \red & \red & \red & \red \\ 
1 & \gray & \green & \green & \red & \green & \red & \red & \red & \red & \red & \red & \red & \red & \red & \red & \red & \red & \red & \red & \red \\ 
2 & \gray & \gray & \green & \green & \green & \red & \red & \red & \red & \red & \red & \red & \red & \red & \red & \red & \red & \red & \red & \red \\ 
3 & \gray & \gray & \gray & \green & \green & \green & \green & \red & \red & \red & \red & \red & \red & \red & \red & \red & \red & \red & \red & \red \\ 
4 & \gray & \gray & \gray & \gray & \green & \green & \green & \green & \red & \red & \red & \red & \red & \red & \red & \red & \red & \red & \red & \red \\ 
5 & \gray & \gray & \gray & \gray & \gray & \green & \green & \green & \green & \red & \red & \red & \red & \red & \red & \red & \red & \red & \red & \red \\ 
6 & \gray & \gray & \gray & \gray & \gray & \gray & \green & \green & \green & \green & \red & \red & \red & \red & \red & \red & \red & \red & \red & \red \\ 
7 & \gray & \gray & \gray & \gray & \gray & \gray & \gray & \green & \green & \green & \green & \red & \red & \red & \red & \red & \red & \red & \red & \red \\ 
8 & \gray & \gray & \gray & \gray & \gray & \gray & \gray & \gray & \green & \green & \green & \green & \red & \red & \red & \red & \red & \red & \red & \red \\ 
9 & \gray & \gray & \gray & \gray & \gray & \gray & \gray & \gray & \gray & \green & \green & \green & \green & \red & \red & \red & \red & \red & \red & \red \\ 
10 & \gray & \gray & \gray & \gray & \gray & \gray & \gray & \gray & \gray & \gray & \green & \green & \green & \red & \red & \red & \red & \red & \red & \red \\ 
11 & \gray & \gray & \gray & \gray & \gray & \gray & \gray & \gray & \gray & \gray & \gray & \green & \green & \red & \red & \red & \red & \red & \red & \red \\ 
12 & \gray & \gray & \gray & \gray & \gray & \gray & \gray & \gray & \gray & \gray & \gray & \gray & \green & \green & \red & \red & \red & \red & \red & \red \\ 
13 & \gray & \gray & \gray & \gray & \gray & \gray & \gray & \gray & \gray & \gray & \gray & \gray & \gray & \green & \green & \red & \red & \red & \red & \red \\ 
14 & \gray & \gray & \gray & \gray & \gray & \gray & \gray & \gray & \gray & \gray & \gray & \gray & \gray & \gray & \green & \green & \red & \red & \red & \red \\ 
15 & \gray & \gray & \gray & \gray & \gray & \gray & \gray & \gray & \gray & \gray & \gray & \gray & \gray & \gray & \gray & \green & \green & \red & \red & \red \\ 
16 & \gray & \gray & \gray & \gray & \gray & \gray & \gray & \gray & \gray & \gray & \gray & \gray & \gray & \gray & \gray & \gray & \green & \green & \red & \red \\ 
17 & \gray & \gray & \gray & \gray & \gray & \gray & \gray & \gray & \gray & \gray & \gray & \gray & \gray & \gray & \gray & \gray & \gray & \green & \green & \red \\ 
18 & \gray & \gray & \gray & \gray & \gray & \gray & \gray & \gray & \gray & \gray & \gray & \gray & \gray & \gray & \gray & \gray & \gray & \gray & \green & \green \\ 
19 & \gray & \gray & \gray & \gray & \gray & \gray & \gray & \gray & \gray & \gray & \gray & \gray & \gray & \gray & \gray & \gray & \gray & \gray & \gray & \green \\ 
\hline 
\end{tabular}
\caption{Negative and positive Greville quadrature weights (identified by red and green cells, respectively) for spline spaces of degree $p=1,\dots,20$ and smoothness $k=0,\dots,p-1$ on the unit line discretized with uniform meshes with $32$ subdivisions. Gray cells are infeasible combinations of degree and smoothness.}
\label{tab: test_conditions}
\end{table}

From \eqref{eq: quadrature_weights} and the partition of unity property,
\begin{equation*}
    \sum_{k=1}^n w_k = \int_\Omega \sum_{k=1}^n L_k(\bm{x}) \,d\bm{x} = I(1) = |\Omega| > 0.
\end{equation*}
Thus, negative weights must necessarily coexist with positive ones, which leads to indefinite lumped mass matrices. The resulting ``negative masses'' have always upset the engineering community for purely physical reasons. However, there are far more dramatic mathematical reasons for avoiding them. Firstly, if the lumped mass matrix is indefinite, $\widehat{b}$ is not an inner product and lacks a key property one would naturally expect it possesses. Secondly, indefinite lumped mass matrices may produce negative or infinite generalized eigenvalues. Their number and the conditions under which they arise were investigated in \cite{malkus1986zero,malkus1988divisor} but the results are scattered across multiple lemmas. For clarity, the conditions are summarized below in a single theorem with a short proof. We begin by recalling some well-known concepts and preliminary results before formulating the conditions. They are particularized to symmetric matrix pairs (i.e. matrix pairs $(A,B) \in \mathbb{R}^{n \times n} \times \mathbb{R}^{n \times n}$ where $A$ and $B$ are both symmetric) since those are the only cases of interest here.

\begin{definition}[Definite matrix pair]
\label{def: definite_pairs}
A symmetric matrix pair $(A,B)$ is called \emph{definite} if
\begin{equation*}
    \inf_{\substack{\bm{x} \in \mathbb{C}^n \\ \bm{x} \neq \bm{0}}} \sqrt{(\bm{x}^*A\bm{x})^2 + (\bm{x}^*B\bm{x})^2} > 0.
\end{equation*}
\end{definition}

\begin{theorem}[{\cite[Corollary VI.1.19]{stewart1990matrix}}]
\label{th: real_eigenvalues}
If $(A,B)$ is a definite pair, then there exists an invertible matrix $U \in \mathbb{R}^{n \times n}$ such that
\begin{equation*}
    U^TAU = D_\alpha = \diag(\alpha_1,\dots,\alpha_n), \quad U^TBU = D_\beta = \diag(\beta_1,\dots,\beta_n),
\end{equation*}
where $D_\alpha$ and $D_\beta$ are real diagonal matrices.
\end{theorem}

In others words, \Cref{th: real_eigenvalues} shows that definite pairs have real (possibly infinite) eigenvalues. For characterizing the eigenvalues of $(K,\widehat{M})$, we still need to define the \emph{inertia} of a symmetric matrix.

\begin{definition}
\label{def: inertia}
For a symmetric matrix $A$, the inertia of $A$ is the ordered triple
\begin{equation*}
    i(A) = (i_+(A),i_-(A),i_0(A)),
\end{equation*}
where $i_+(A)$, $i_-(A)$ and $i_0(A)$ denote the number of positive, negative and zero eigenvalues, respectively.
\end{definition}

We are now ready to state the result.

\begin{theorem}
\label{th: spectrum}
Let $(A,B)$ be a symmetric matrix pair with inertia
\begin{equation*}
    i(A) = (i_+(A),0,i_0(A)), \quad i(B) = (i_+(B),i_-(B),i_0(B)).
\end{equation*}
If $\bm{x}^*B\bm{x} \neq 0$ for all $\bm{x} \in \ker(A)$, then $(A,B)$ has real eigenvalues, including
\begin{itemize}[noitemsep]
    \item $i_0(A)$ zero eigenvalues,
    \item $i_+(B)$ nonnegative finite eigenvalues,
    \item $i_-(B)$ nonpositive finite eigenvalues,
    \item $i_0(B)$ infinite (positive) eigenvalues.
\end{itemize}
\end{theorem}
\begin{proof}
We first show that $(A,B)$ is a definite pair. Since $A$ is positive semidefinite, $\bm{x}^*A\bm{x} \geq 0$ for all $\bm{x} \in \mathbb{C}^n$ and $\bm{x}^*A\bm{x} = 0$ if and only if $\bm{x} \in \ker(A)$. Thus, if $\bm{x}^*B\bm{x} \neq 0$ for all $\bm{x} \in \ker(A)$, then
\begin{equation*}
    \inf_{\substack{\bm{x} \in \mathbb{C}^n \\ \bm{x} \neq \bm{0}}} \sqrt{(\bm{x}^*A\bm{x})^2 + (\bm{x}^*B\bm{x})^2} > 0
\end{equation*}
and $(A,B)$ is a definite pair. Thanks to \Cref{th: real_eigenvalues}, its eigenvalues are real and there exists an invertible matrix $U \in \mathbb{R}^{n \times n}$ such that
\begin{equation*}
    U^TAU = D_\alpha = \diag(\alpha_1,\dots,\alpha_n), \quad U^TBU = D_\beta = \diag(\beta_1,\dots,\beta_n),
\end{equation*}
where $D_\alpha$ and $D_\beta$ are real diagonal matrices. Since $U^TAU$ and $U^TBU$ are congruence transformations, it follows from Sylvester's law of inertia \cite[Theorem 4.5.8]{horn2012matrix} that
\begin{equation*}
    i(D_\alpha) = (i_+(A),0,i_0(A)), \quad i(D_\beta) = (i_+(B),i_-(B),i_0(B)).
\end{equation*}
Since $(A,B)$ is definite, $(\alpha_i,\beta_i) \neq (0,0)$ and the eigenvalues of $(A,B)$ are given by the ratios $\alpha_i/\beta_i$. The result then follows from the fact that $\alpha_i \geq 0$.
\end{proof}

If $A$ is positive definite, ``nonnegative'' and ''nonpositive'' in the statement of \Cref{th: spectrum} may be replaced by ``positive'' and ``negative'', respectively. Note that the assumption of \Cref{th: spectrum} is equivalent to stating that $B$ is definite on the kernel of $A$. In terms of bilinear forms, it amounts to saying that
\begin{equation*}
    |\widehat{b}(v,v)| > 0 \quad \forall v \in \mathbb{S}^{\perp_a} = \{u \in \mathbb{S} \colon a(u,v)=0 \; \forall v \in \mathbb{S}\}.
\end{equation*}
If this assumption does not hold, $(A,B)$ may have complex eigenvalues. The situation is even worse if $A$ and $B$ share a common null space since $(A,B)$ is not even regular \cite{stewart1990matrix}. Fortunately, for the standard Laplace operator, the assumption of \Cref{th: spectrum} is satisfied since $\mathbb{S}^{\perp_a} = \Span \{1\}$ and
\begin{equation*}
    \widehat{b}(1,1) = \sum_{k=1}^n w_k = |\Omega| > 0.
\end{equation*}
Moreover, since $u=1$ is also an eigenfunction, the zero eigenvalue is counted among the nonnegative eigenvalues and \Cref{th: spectrum} shows that there are as many negative eigenvalues as there are negative weights and as many infinite (positive) eigenvalues as there are zero weights. In numerical computations, the latter is very unlikely but the former is a serious concern. As a matter of fact, the exact solution of the semi-discrete problem \eqref{eq: semi_discrete_pb} might diverge and so does the fully discrete solution, if the time step is sufficiently small \cite{cohen2002higher,malkus1988reversed,duczek2019critical}. Clearly, negative weights must be avoided at all cost, and so the strategy recently proposed in \cite{li2022significance} based on the Greville abscissa should only be used when the quadrature weights are positive.
However, since negative weights may emerge for certain quadrature points such as the Greville abscissa, we are naturally led to the following questions:
\begin{enumerate}[noitemsep]
    \item For a given set of interpolation points, are there sufficient conditions guaranteeing positive weights?
    \item Is there a set of interpolation points that always produces positive weights?
\end{enumerate}

We will answer the first question and later give strong evidence for positively answering the second one. Let us recall that the quadrature weights are the solution of $A^T\bm{w}=\bm{b}$, where $A$ is nonnegative with positive diagonal and $\bm{b}$ is positive. Such systems are often called \emph{positive} and their solution has already been studied by several authors \cite{kaykobad1985positive,voet2022internodes}. 

\begin{definition}[Positive system]
\label{def: positive_system}
A linear system $A\bm{x} = \bm{b}$ with $A \in \mathbb{R}^{n \times n}$ and $\bm{b} \in \mathbb{R}^n$ is called positive if
\begin{equation*}
\begin{cases}
 a_{ij} \geq 0, \ a_{ii}>0 & i,j=1,\dots,n, \\
 b_i > 0 & i=1,\dots,n.
\end{cases}
\end{equation*}
\end{definition}
Positive systems may not have a positive solution, unless the coefficient matrix and right-hand side verify additional conditions. Such conditions are provided in the next lemma.

\begin{lemma}
\label{lem: positive_solution}
For a positive system $A\bm{x} = \bm{b}$, if
\begin{equation*}
    b_i > \sum_{\substack{j=1 \\ j \neq i}}^n \frac{a_{ij}}{a_{jj}}b_j \qquad i=1,\dots,n,
\end{equation*}
then $A$ is invertible and the components of the solution satisfy
\begin{equation*}
    0 < x_i \leq \frac{b_i}{a_{ii}} \qquad i=1,\dots,n.
\end{equation*}
\end{lemma}
\begin{proof}
The result simply follows from applying \cite[Theorem 2.2]{voet2022internodes} to the rescaled system $D_1^{-1}AD_2^{-1}\tilde{\bm{x}}=D_1^{-1}\bm{b}$, where $D_1=\diag(\bm{b})$, $D_2=\diag(D_1^{-1}A)$, $\tilde{\bm{x}}=D_2\bm{x}$ and then back-transforming to the original variables.    
\end{proof}

\begin{remark}
The conditions in \Cref{lem: positive_solution} were independently shown in \cite{kaykobad1985positive}, but the author does not provide an upper bound on the solution.
\end{remark}

Applied to the moment-fitting equations \eqref{eq: moment_fitting}, \Cref{lem: positive_solution} leads to the following conditions.

\begin{corollary}
\label{cor: sufficient_conditions}
For the moment-fitting equations \eqref{eq: moment_fitting}, if
\begin{equation*}
    I(B_i) > \sum_{\substack{j=1 \\ j \neq i}}^n \frac{B_i(\bm{x}_j)}{B_j(\bm{x}_j)}I(B_j) \qquad i=1,\dots,n,
\end{equation*}
then the quadrature weights satisfy
\begin{equation*}
    0 < w_i \leq \frac{I(B_i)}{B_i(\bm{x}_i)} \qquad i=1,\dots,n.
\end{equation*}
\end{corollary}
\begin{proof}
The result immediately follows from applying \Cref{lem: positive_solution} to the moment-fitting equations \eqref{eq: moment_fitting}.     
\end{proof}

For the Greville abscissa, the conditions stated in \Cref{cor: sufficient_conditions} are typically satisfied for low order splines. In particular, they partly recover the intriguing pattern observed in \Cref{tab: test_conditions} for quartic splines, being satisfied for smoothness $k=0$ and $k=2$, but not for $k=1$. Unfortunately though, they are never satisfied from degree $5$ onward. Differences in the values of $I(B_i)$ are the main cause of breakdown. For open knot vectors and uniform meshes, $I(B_i)$ is much smaller for basis functions near the boundaries than for those in the interior and increasing the spline order only exacerbates this discrepancy. Thus, the conditions of \Cref{cor: sufficient_conditions} are quickly violated for basis functions near the boundaries. Similar restrictive conditions held for the Demko points. However, in this case, the quadrature weights were always positive, even for the wildest knot vectors we could imagine (see \Cref{se: non_uniform_meshes}). This is a clear limitation of purely algebraic techniques. Unfortunately, since \Cref{lem: positive_solution} is derived from \cite[Theorem 2.2]{voet2022internodes} and none of the assumptions therein can be relaxed, finding more lenient algebraic conditions is very unlikely.

Nevertheless, this finding suggests the existence of a set of quadrature points (i.e. the Demko points) guaranteeing positive weights, a property that does not seem entirely surprising given the optimal nature of Demko points for interpolation \cite{demko1985existence,morken1996total,smith1988knots}. Indeed, interpolatory quadrature rules inherit some of the properties of the underlying interpolation operator. For example, the negative weights appearing for high order Newton-Cotes quadrature rules are an immediate consequence of the Runge phenomenon developing on uniformly spaced interpolation nodes \cite{trefethen2019approximation}. We believe similar issues might also arise for interpolatory spline quadrature rules, which are built on similar concepts. For spline spaces, the interpolation operator $P \colon C^0(\Omega) \to \mathbb{S}$ is constructed as follows: given a continuous function $f \in C^0(\Omega)$ and a unisolvent set of interpolation points $\{\bm{x}_i\}_{i=1}^n$, $Pf$, the interpolant of $f$, is the unique spline satisfying
\begin{equation*}
    Pf(\bm{x}_i)=f(\bm{x}_i) \quad i=1,\dots,n.
\end{equation*}
The coefficients of $Pf(\bm{x}) = \sum_{i=1}^n \alpha_i B_i(\bm{x})$ are obtained by solving the linear system
\begin{equation*}
    A \bm{\alpha} = \bm{f},
\end{equation*}
where $A$ is the collocation matrix defined in \eqref{eq: def_matrices} and $f_i=f(\bm{x}_i)$. The interpolation operator is linear and since it also satisfies $Ps=s \; \forall s \in \mathbb{S}$, it is sometimes called an interpolating projection. Furthermore, since the quadrature rule is built on the same interpolation points and $Q(s)=I(s) \; \forall s \in \mathbb{S}$, we immediately deduce that
\begin{equation*}
    Q(f) = Q(Pf) = I(Pf).
\end{equation*}
These relations suggest a close connection between quadrature and interpolation whereby instabilities for the latter could cause instabilities for the former. The stability of spline interpolants was extensively studied in the 1980s. The problem consists in finding a set of interpolation points such that $\|Pf\| \leq C \|f\|$ for some constant $C$ independent of the knot distribution. In the following derivations, the norm employed is the infinity norm, unless stated otherwise. From the stability of the B-spline basis \cite{lyche2018spline}, we immediately deduce that
\begin{equation*}
    \|Pf\| \leq \|\bm{\alpha}\| = \|A^{-1}\bm{f}\| \leq \|A^{-1}\|\|f\|.
\end{equation*}
Therefore, we deduce the following upper bound on the norm of the interpolation operator \cite[Lemma 1.1]{de1975bounding},
\begin{equation*}
    \|P\| = \sup_{f \in C^0(\Omega)} \frac{\|Pf\|}{\|f\|} \leq \|A^{-1}\|.
\end{equation*}

The Greville abscissa are provably stable (i.e. $\|A^{-1}\|$ is bounded independently of the knot distribution) only for degree $p \leq 3$ \cite{de1975bounding,marsden1974quadratic}. For splines of very high degree $p \geq 20$, counter-examples on graded meshes were constructed, demonstrating the instability of spline interpolation for the Greville abscissa \cite{jia1988spline}. In contrast, Demko \cite{demko1985existence} proved that there exists a set of interpolation points for which spline interpolation is stable. Nowadays, those points are widely known as the Demko points, although Mørken proved an identical result in his Master thesis \cite{morken1984two} at about the same time.

\begin{theorem}[{\cite{demko1985existence}}]
\label{th: stable_interpolation}
For the Demko points,
\begin{equation*}
    \|A^{-1}\| \leq \kappa(\dutchcal{B}),
\end{equation*}
where $\kappa(\dutchcal{B})$ is the condition number of the B-spline basis.
\end{theorem}

Since $\kappa(\dutchcal{B})$ is independent of the knot sequence (see e.g. \cite{lyche2018spline}), \Cref{th: stable_interpolation} shows that spline interpolation is stable for the Demko points. In fact, Mørken not only proved stability, but also optimality in the sense that $\|A^{-1}\|$ is minimized over all sets of unisolvent points. This result, which is also found in \cite{smith1988knots}, is summarized in the next theorem. For clarity, a subscript is now appended to the collocation matrix for marking the dependency on the interpolation points.

\begin{theorem}
\label{th: optimality}
Let $\{\bm{\tau}_i\}_{i=1}^n$ denote the Demko points. Then, for any other unisolvent set $\{\bm{x}_i\}_{i=1}^n$,
\begin{equation*}
    \|A^{-1}_{\bm{\tau}}\| \leq \|A^{-1}_{\bm{x}}\|.
\end{equation*}
\end{theorem}
Stability is an important prerequisite for the convergence of an interpolant. Indeed, for a function $f \in C^0(\Omega)$ and spline $s \in \mathbb{S}$,
\begin{equation}
\label{eq: interpolation_bound}
    \|f-Pf\| \leq \|f-s\| + \|P(f-s)\| \leq (1+\|P\|)\|f-s\|.
\end{equation}
Since the previous argument holds for any spline $s \in \mathbb{S}$, having a stable interpolant ensures than $\|f-Pf\|$ behaves (up to a constant) as the best approximation error $\inf_{s \in \mathbb{S}} \|f-s\|$. This result is summarized in the next lemma.

\begin{lemma}
\label{lem: convergence_interpolation}
For a function $f \in C^0(\Omega)$ and a stable spline interpolation operator $P$,
\begin{equation*}
    \|f-Pf\| \leq C_p \inf_{s \in \mathbb{S}} \|f-s\|,
\end{equation*}
where $C_p$ is a constant independent of the knot distribution.    
\end{lemma}

We will now attempt to connect some of the properties of interpolation to quadrature. Firstly, the definition of the quadrature operator in \eqref{eq: operators} straightforwardly implies that
\begin{equation*}
    |Q(f)| \leq \|\bm{w}\|_1 \|f\|.
\end{equation*}
Moreover, the upper bound is attained for a continuous function $f$ such that $\|f\|=1$ and $f(x_i)=\sign(w_i)$. Hence, $\|Q\| = \|\bm{w}\|_1$ and the latter is a widely accepted definition for the condition number of a quadrature rule \cite{huybrechs2009stable,van2021geometrical}. Obviously, if all weights are positive, $\|Q\|=\sum_{k=1}^n w_k = I(1)$ is uniformly bounded. But if there exists a negative weight, $\|Q\| > I(1)$. Furthermore, from the partition of unity property of the B-spline basis, we deduce that
\begin{equation}
\label{eq: quadrature_condition}
    I(1) = \sum_{i=1}^n w_i \leq \sum_{i=1}^n |w_i| = \|\bm{w}\|_1 = \|A^{-T}\bm{b}\|_1 \leq \|A^{-1}\|I(1),
\end{equation}
thereby yielding lower and upper bounds on the condition number. The stability constant is exactly the same as for the interpolant. Thus, if the interpolant is stable (in the sense that $\|A^{-1}\|$ is uniformly bounded), then so is the quadrature. The following error bound shows why bounded quadrature operators matter. The linearity of integration and quadrature implies that for a function $f \in C^0(\Omega)$ and any spline $s \in \mathbb{S}$,
\begin{equation*}
    |I(f)-Q(f)| \leq |I(f-s)| + |Q(f-s)| \leq (I(1)+\|Q\|)\|f-s\|.
\end{equation*}
This bound is the immediate counterpart of \eqref{eq: interpolation_bound} and implies the following lemma.

\begin{lemma}
\label{lem: convergence_quadrature}
For a function $f \in C^0(\Omega)$ and a stable spline interpolation operator $P$,
\begin{equation*}
    |I(f)-Q(f)| \leq C_p \inf_{s \in \mathbb{S}} \|f-s\|,
\end{equation*}
where $C_p$ is a constant independent of the knot distribution.      
\end{lemma}

We deduce that stable interpolation implies stable quadrature and insures convergence of the associated quadrature rule. However, stable quadrature a priori does not imply positive weights. The fact that negative weights are also encountered for low order Greville interpolation (which is provably stable) confirms that negative weights do not necessarily imply unstable quadrature rules. The inequality in \eqref{eq: quadrature_condition} already hints at that possibility. However, \eqref{eq: quadrature_condition} might not tell the whole story. In particular, even when negative weights occur, the inequality appears excessively pessimistic (see \Cref{se: non_uniform_meshes}). Numerical experimentation and intuition has led to the following conjecture for an improved bound.

\begin{conjecture}
\label{conj: quadrature_condition}
Let $c(\bm{x})$ denote the oscillating spline on a set of unisolvent interpolation points. Then, the associated quadrature weights satisfy
\begin{equation*}
    I(1) \leq \|\bm{w}\|_1 \leq \|c\|I(1).
\end{equation*}
\end{conjecture}

\begin{remark}
\Cref{conj: quadrature_condition} would improve the upper bound of \eqref{eq: quadrature_condition} and this is all that is needed for proving that the Demko weights are nonnegative. Indeed, for the Demko points, the oscillating (Chebyshev) spline satisfies $\|c\|=1$ (see \eqref{eq: max_norm}). Hence, if \Cref{conj: quadrature_condition} holds, $\sum_{i=1}^n |w_i| = \|\bm{w}\|_1 = I(1) = \sum_{i=1}^n w_i$, and all weights are necessarily nonnegative.   
\end{remark}

Unfortunately, despite multiple attempts, we have neither found a counter-example nor a complete proof for \Cref{conj: quadrature_condition}. Thus, it is left as an open problem.

\subsection{Truncated interpolatory spline basis}
\label{se: truncation}
Another important issue is related to the support of interpolatory basis functions. Their coefficients in the B-spline basis are stored along the columns of $C$ and since the latter is completely dense, the interpolatory basis functions are globally supported. This fact alone precludes explicitly forming $K_\dutchcal{L}$, although $C$ itself may be implicitly computed and stored by exploiting the Kronecker structure of $A$. Even if $K_\dutchcal{L}$ could be formed, we would have merely transferred the burden of solving linear systems with $M_\dutchcal{B}$ to computing matrix-vector products with $K_\dutchcal{L}$, which might end up being just as expensive \cite{radtke2024analysis}. This issue seems to have been overlooked in \cite{li2022significance,li2024interpolatory}. Fortunately, there is a simple workaround. Although $C$ is completely dense, its off-diagonal entries generally decay exponentially fast away from the diagonal. In the case of interior Greville points on a uniform mesh this directly follows from \cite[Theorem~2]{Schoenberg1972}. More generally, this property has been extensively studied in the linear algebra and spline communities alike \cite{demko1977inverses,demko1984decay,benzi1999bounds,canuto2014decay,benzi2015decay} and naturally suggests approximating $C$ by a sparse matrix $\bar{C}$, thereby truncating the interpolatory functions. Constructing $\bar{C}$ simply consists in discarding the entries of $C$ whose magnitude drops below a certain threshold $\epsilon$ (e.g. $10^{-14}$). In higher dimensions, truncating univariate bases ensures that $\bar{C}$ retains a Kronecker product structure, which is crucial for computational efficiency, as we will later discuss. \Cref{fig: 1D_Lagrange_Demko_spline_global_compact_p3_n25_s14} compares one of the original interpolatory basis functions to its truncated version for a cubic spline discretization on a uniform mesh of $25$ elements and a truncation tolerance of $10^{-4}$. The globally and compactly supported functions are visually undistinguishable. Since the truncation error might impede on the convergence, the truncation tolerance must naturally match the desired accuracy. Provided $\bar{C}$ is invertible, the truncated set of functions $\bar{\dutchcal{L}} = \{\bar{L}_1,\dots,\bar{L}_n\}$ still forms a basis for $\mathbb{S}$ and the stability of the B-spline basis immediately shows that
\begin{equation*}
    \|L_i - \bar{L}_i\| \leq \epsilon
\end{equation*}
for univariate spaces. In the multivariate case, for a multi-index $\bm{i}=(i_1,\dots,i_d)$, the estimate becomes
\begin{equation*}
    \|L_{\bm{i}} - \bar{L}_{\bm{i}}\| \lesssim \epsilon \sum_{k=1}^d \|L_{i_k}\|
\end{equation*}
up to high order terms in $\epsilon$. Although the truncation error might get amplified in the multivariate case, the amplification factor is only slightly larger than $1$ in all practically relevant cases. A truncation level of $\epsilon=10^{-14}$ typically does not impede on the convergence and restores some sparsity in the stiffness matrix. It is still inevitably denser than for the B-spline basis but a similar issue arises for alternative techniques based on approximate dual functions \cite{nguyen2023towards,held2024efficient}. We will later explain how to leverage the Kronecker structure of $\bar{C}$ to further mitigate the issue.

\begin{figure}[H]
    \centering
    \includegraphics[scale=0.5]{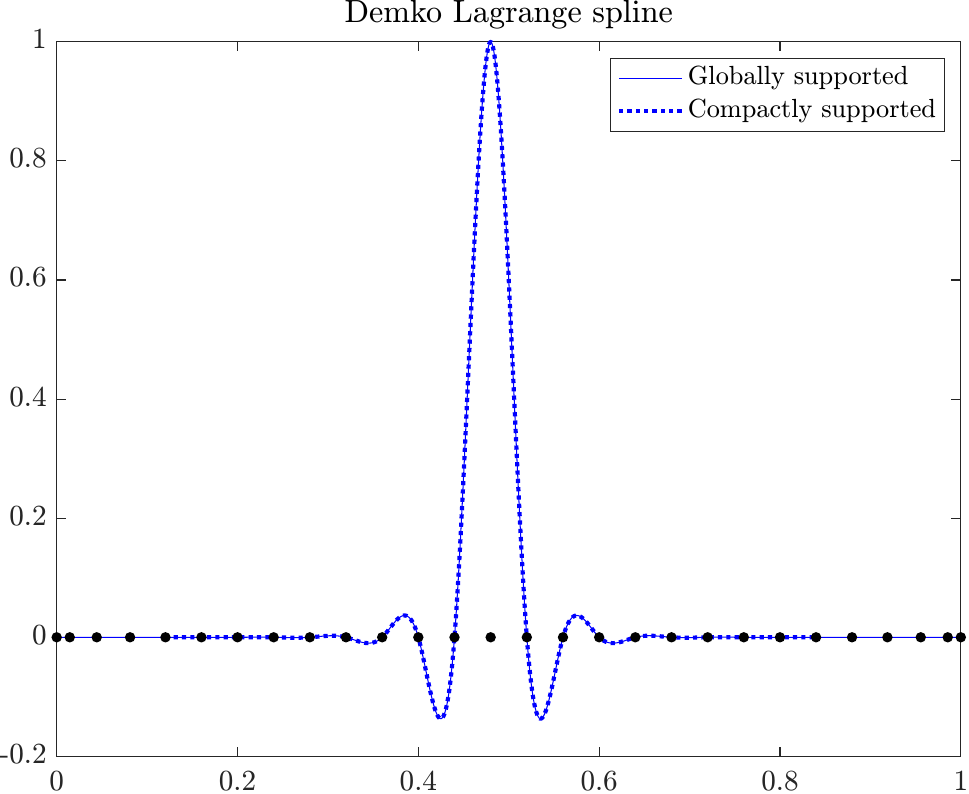}
    \caption{Original and truncated Lagrange cubic spline function for the Demko points on a uniform mesh of $25$ elements and a truncation tolerance of $10^{-4}$}
    \label{fig: 1D_Lagrange_Demko_spline_global_compact_p3_n25_s14}
\end{figure}

\subsection{Critical time step}
\label{se: CFL_condition}
Finally, lumping the mass matrix for an interpolatory basis has yet another issue in addition to those already mentioned: it may deteriorate the CFL condition. As mentioned in the introduction, mass lumping not only seeks a diagonal approximation of the consistent mass but also an improvement of the critical time step. As shown in \cite[Corollary 3.10]{voet2023mathematical}, this property is guaranteed for the row-sum technique and nonnegative mass matrices. Unfortunately, for interpolatory bases, the mass matrix usually features negative entries and the previous result no longer holds. Indeed, as we have seen, lumping positive definite matrices may result in indefinite ones. Strict diagonal dominance is sufficient for guaranteeing positive definite lumped mass matrices but the mass matrix rarely satisfies this condition for high order discretizations. Moreover, the critical time step may still decrease even if the lumped mass is positive definite. Nevertheless, if 
\begin{equation}
\label{eq: sufficient_cond}
    Q(s^2) \geq I(s^2) \quad \forall s \in \mathbb{S},
\end{equation}
then the lumped mass matrix is positive definite and the critical time step cannot decrease. Indeed, if \eqref{eq: sufficient_cond} holds, then
\begin{equation*}
    \widehat{b}(s,s) = Q(s^2) \geq I(s^2) = b(s,s) > 0 \quad \forall s \in \mathbb{S},
\end{equation*}
which is equivalent to $\widehat{M}_{\Phi} \succeq M_{\Phi} \succ 0$ for any basis $\Phi$ of $\mathbb{S}$ and in particular for the Lagrange spline basis $\dutchcal{L}$. Combining this result with \cite[Corollary 3.6]{voet2023mathematical} shows that
\begin{equation*}
\lambda_k(K_{\Phi},\widehat{M}_{\Phi}) \leq \lambda_k(K_{\Phi}, M_{\Phi}) \quad k=1,\dots,n.
\end{equation*}
Unfortunately, neither the Greville nor the Demko quadrature rules satisfy \eqref{eq: sufficient_cond} and although the condition is only sufficient, our experiments confirmed that the critical time step may decrease. In fact, the existence of a spline quadrature rule satisfying \eqref{eq: sufficient_cond} is uncertain at this stage, although the condition emerges quite naturally from the proof of \Cref{lem: spectral_fem} (see \Cref{se: proof_CFL_SEM}). Provided all weights are positive, the next section explains how to practically solve the semi-discrete problem for a generic set of interpolation points.

\subsection{Algorithm}
\label{se: algorithm}
First recall that the semi-discrete problem in \eqref{eq: semi_discrete_pb} stems from a specific choice of basis. Once a basis $\Phi$ is chosen, we find the coefficient vector of the solution relative to the basis $\Phi$, denoted $\bm{u}_\Phi(t)$, by (approximately) solving the system of ordinary differential equations
\begin{equation}
\label{eq: semi_discrete_pb_phi}
M_\Phi\ddot{\bm{u}}_\Phi(t) + K_\Phi\bm{u}_\Phi(t) = \bm{f}_\Phi(t) \qquad \text{for } t \in [0,T],
\end{equation}
with appropriate initial conditions. Instead of commonly choosing the B-spline basis $\Phi=\dutchcal{B}$ in \eqref{eq: semi_discrete_pb_phi}, we choose the Lagrange spline basis $\Phi=\dutchcal{L}$ and lump the mass matrix in the interpolatory basis. Note that forming $M_{\dutchcal{L}}$ is unnecessary before lumping it. Indeed, \Cref{lem: LML_QML} shows that $\widetilde{M}_\dutchcal{L}=\widehat{M}_{\dutchcal{L}}=\diag(\bm{w})$. Thus, we directly compute the weights by solving the moment-fitting equation \eqref{eq: moment_fitting} and solve linear systems with $\widehat{M}_{\dutchcal{L}}$ by simply rescaling the components of the right-hand side vector. Also note that the lumped mass matrix is always formed before taking care of any Dirichlet boundary conditions in \eqref{eq: semi_discrete_pb_phi}. Once \eqref{eq: semi_discrete_pb_phi} has been solved, the solution in the B-spline basis can be recovered in a single back-transformation step. A complete pseudo-code is provided in \Cref{algo: time_stepping} for pure Neumann boundary conditions. Adapting the algorithm to Dirichlet or mixed boundary conditions merely requires extracting submatrices from the global lumped mass and stiffness matrices.

\begin{algorithm}[H]
\begin{algorithmic}[1]
\caption{Time stepping with the Lagrange spline basis}
\label{algo: time_stepping}
\Statex \textbf{Input}: Set of interpolation points $\{\bm{x}_k\}_{k=1}^n$, initial conditions $\bm{u}_0$ and $\bm{v}_0$ in the B-spline basis.
\State Compute the collocation matrix $A$ in \eqref{eq: def_matrices} and its inverse $C$.
\State Compute the quadrature weights $\bm{w}=A^{-T}\bm{b}$ from \eqref{eq: moment_fitting}.
\State Form the lumped mass $\widehat{M}_{\dutchcal{L}}=\diag(\bm{w})$.
\State Solve (approximately) the system $\widehat{M}_{\dutchcal{L}}\ddot{\bm{u}}_{\dutchcal{L}}(t) + K_{\dutchcal{L}}\bm{u}_{\dutchcal{L}}(t) = \bm{f}_{\dutchcal{L}}(t)$ 
\Statex with initial conditions $\bm{u}_{\dutchcal{L}}(0) = A \bm{u}_0$ and $\dot{\bm{u}}_{\dutchcal{L}}(0)=A\bm{v}_0$.
\State Recover the approximate solution in the B-spline basis $\bm{u}_{\dutchcal{B}}(t) = C\bm{u}_{\dutchcal{L}}(t)$.
\end{algorithmic}
\end{algorithm}

At this stage, we would like to mention some important differences of \Cref{algo: time_stepping} with dual mass lumping approaches \cite{nguyen2023towards,held2024efficient,hiemstra2025higher}. Although the lumped mass matrix is diagonal in both cases, the cost of computing matrix-vector multiplications with the stiffness for the evaluation of the right-hand side may significantly differ. Dual mass lumping achieves approximate duality in the physical domain by rescaling the approximate dual functions by $c(\hat{\bm{x}}) = \rho(F(\hat{\bm{x}})) |\det(J(\hat{\bm{x}}))|$ where $F \colon \hat{\Omega} \to \Omega$ is the mapping from the parametric to the physical domain and $J$ denotes its Jacobian matrix. Unfortunately, due to the derivatives appearing in the stiffness, the rescaling doubles or even quadruples the cost of computing matrix-vector multiplications with the stiffness \cite{nguyen2023towards,hiemstra2025higher} and the authors admit that the overall saving is still unclear. Also bear in mind that the stiffness itself is non-symmetric. Although this might not have any computational drawbacks, it remains an unexpected outcome for what originally was a standard Galerkin method. 

In contrast, the stiffness matrix $K_{\dutchcal{L}} = C^T K_\dutchcal{B}C$ for interpolatory spline bases remains symmetric positive semidefinite and the change of basis is just a separate add-on. Thus, one can reuse the same algorithm for computing matrix-vector multiplications with the stiffness in the B-spline basis and simply compose it with the change of basis. We still have to discuss how to perform this composition efficiently. The first idea, explained in \Cref{se: truncation}, is to truncate the interpolatory basis, which amounts to substituting $C$ with a sparse approximation $\bar{C}$. Truncating univariate bases ensures that $\bar{C}$ retains a Kronecker product structure; i.e.
\begin{equation*}
    \bar{C} = \bigotimes_{i=1}^d \bar{C}_i.
\end{equation*}
Although forming the stiffness in the truncated basis $K_{\bar{\dutchcal{L}}}$ is possible, we would advise against it due to considerable fill-in. Indeed, assuming that $K_\dutchcal{B}$ and $\bar{C}$ are block-banded of bandwidths $\bm{b}=(b_1,\dots,b_d)$ and $\bm{q}=(q_1,\dots,q_d)$, respectively, then $K_{\bar{\dutchcal{L}}} = \bar{C}^T K_\dutchcal{B}\bar{C}$ is also block-banded of bandwidth $\bm{b}+2\bm{q}$ (see e.g. \cite{voet2025mass}). Thus, if $K_{\bar{\dutchcal{L}}}$ is formed explicitly, a single matrix-vector multiplication will scale as $O(n \prod_{i=1}^d (b_i + 2q_i))$, where $n=\dim(\mathbb{S})$. If matrix-vector multiplications with $K_{\bar{\dutchcal{L}}}$ are instead performed sequentially from right to left by computing one matrix-vector multiplication at a time, the cost essentially reduces to 
\begin{equation*}
    O\left(n( \prod_{i=1}^d b_i + 2\sum_{i=1}^dq_i)\right).
\end{equation*}
This strategy not only avoids the fill-in arising when explicitly forming $K_{\bar{\dutchcal{L}}}$ but also exploits the Kronecker product structure of $\bar{C}$ when computing matrix-vector multiplications with it (or its transpose). The Kronecker product is a well-known device for breaking the ``curse of dimensionality'' and further savings occur if $K_\dutchcal{B}$ itself is assembled in low (Kronecker) rank format (see e.g. \cite{mantzaflaris2017low,scholz2018partial,hofreither2018black}) or not assembled at all. In such cases, the cost of $O(n\prod_{i=1}^d b_i)$ is substituted with the cost of a matrix-vector multiplication with $K_\dutchcal{B}$.

Another possibility consists in directly exploiting the Kronecker structure of $A = \bigotimes_{i=1}^d A_i$ and might produce similar savings. Since $A_i$ has a band structure with at most $p_i+1$ consecutive nonzero entries in each row \cite{lyche2018spline}, its bandwidth never exceeds $p_i$ and the cost of computing matrix-vector multiplications with $A^{-T}K_\dutchcal{B}A^{-1}$ sequentially amounts to
\begin{equation*}
    O\left(n(\prod_{i=1}^d b_i + 2 \sum_{i=1}^d p_i^2)\right).
\end{equation*}
The best strategy depends on how $q_i$ compares to $p_i^2$. However, the comparison is unclear since $q_i$ (the bandwidth of $\bar{C}_i$) depends on the choice of interpolation points and truncation tolerance. Even if $q_i$ is a linear function of $p_i$, the hidden constants in the big $O$ notation will play a role for small to moderate values of $p_i$. Moreover, systems with small bandwidths can be solved more efficiently at a cost that is linear in $p_i$ \cite{quarteroni2010numerical}. Thus, from a cost perspective, the two strategies seem equally viable for common values of $p_i$. However, the second one does not require truncating the basis. Those considerations are truly independent of the way the stiffness is treated (i.e. explicitly or implicitly through matrix-vector products).

Finally, thanks to the interpolatory nature of the basis functions, imposing essential boundary conditions is just as simple as standard finite elements. This is in contrast to the (approximate) dual bases approaches, which require some extra work to impose such conditions \cite{hiemstra2025higher, held2024efficient}.

\section{Numerical experiments}
\label{se: numerical_experiments}
This section collects a few experiments to confirm our results but especially to draw additional insights on the intriguing behavior of spline quadrature and mass lumping. The experiments are repeated for two common choices of interpolation points, namely the Greville and Demko points, although negative weights are occasionally experienced for the former. The Lagrange spline bases for those two sets are exemplarily shown in \Cref{fig: spline_bases} alongside the standard B-spline basis for a maximally smooth cubic spline space with $7$ subdivisions. The Lagrange spline basis functions are indeed interpolatory: $L_j(x_i)=\delta_{ij}$, where $\{x_i\}_{i=1}^n$ are the Greville or Demko points, identified as black dots in \Cref{fig: 1D_Lagrange_Greville_basis_p3_ne7,fig: 1D_Lagrange_Demko_basis_p3_ne7}, respectively. However, similarly to the polynomial Lagrange basis, nothing prevents $|L_j(x)|$ from exceeding $1$ between the interpolation points. This phenomenon is also somewhat connected to the stability properties of the basis and, as expected, was always more pronounced for the Greville points. Moreover, although the basis functions are globally supported, for sufficiently fine meshes, they are often very well approximated by compactly supported functions, as explained in \Cref{se: truncation}. 

\begin{figure}[H]
     \centering
     \begin{subfigure}[t]{0.31\textwidth}
    \centering
    \includegraphics[width=\textwidth]{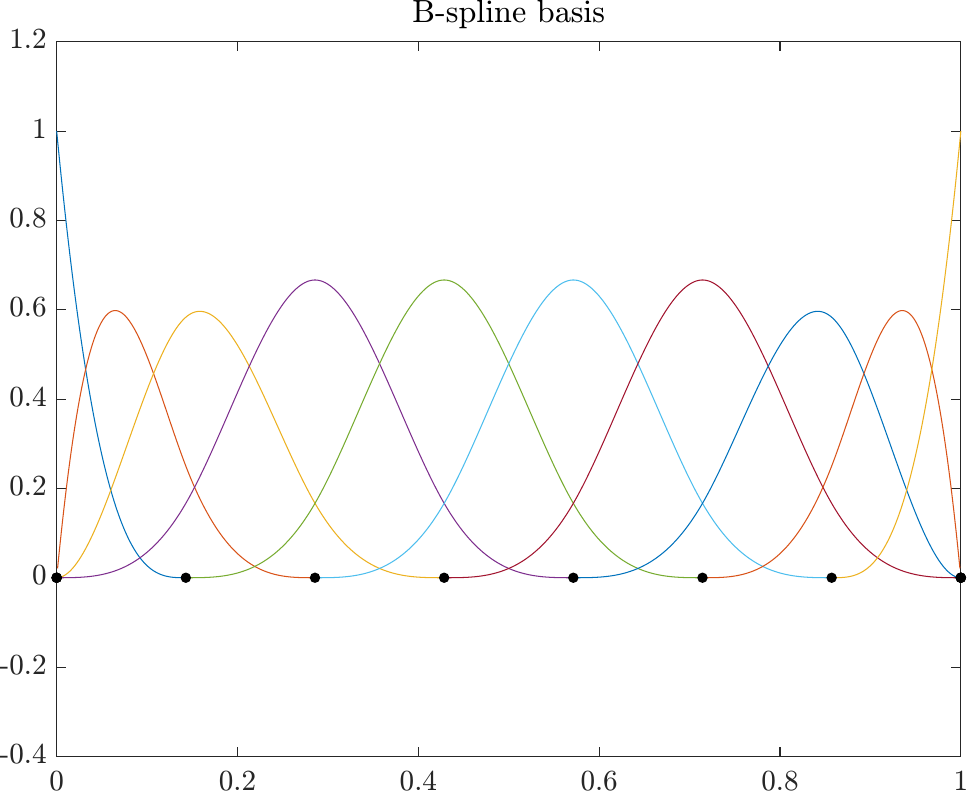}
    \caption{B-spline basis}
    \label{fig: 1D_Bspline_basis_p3_ne7}
     \end{subfigure}
     \hfill
     \begin{subfigure}[t]{0.31\textwidth}
    \centering
    \includegraphics[width=\textwidth]{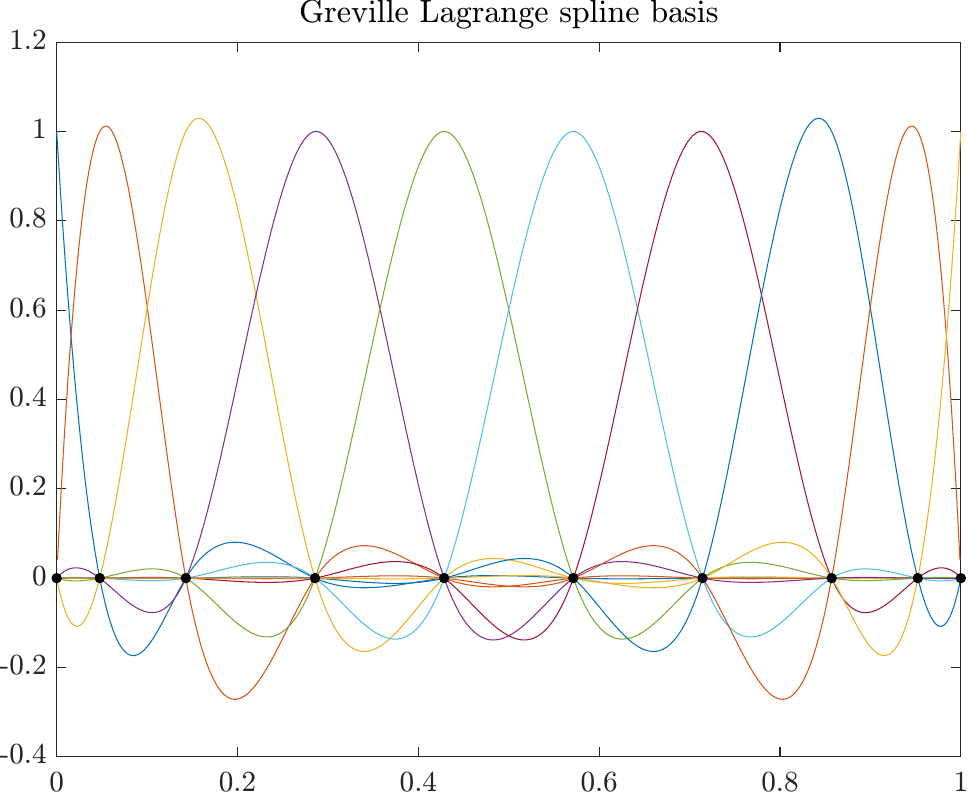}
    \caption{Greville Lagrange spline basis}
    \label{fig: 1D_Lagrange_Greville_basis_p3_ne7}
     \end{subfigure}
     \hfill
    \begin{subfigure}[t]{0.31\textwidth}
    \centering
    \includegraphics[width=\textwidth]{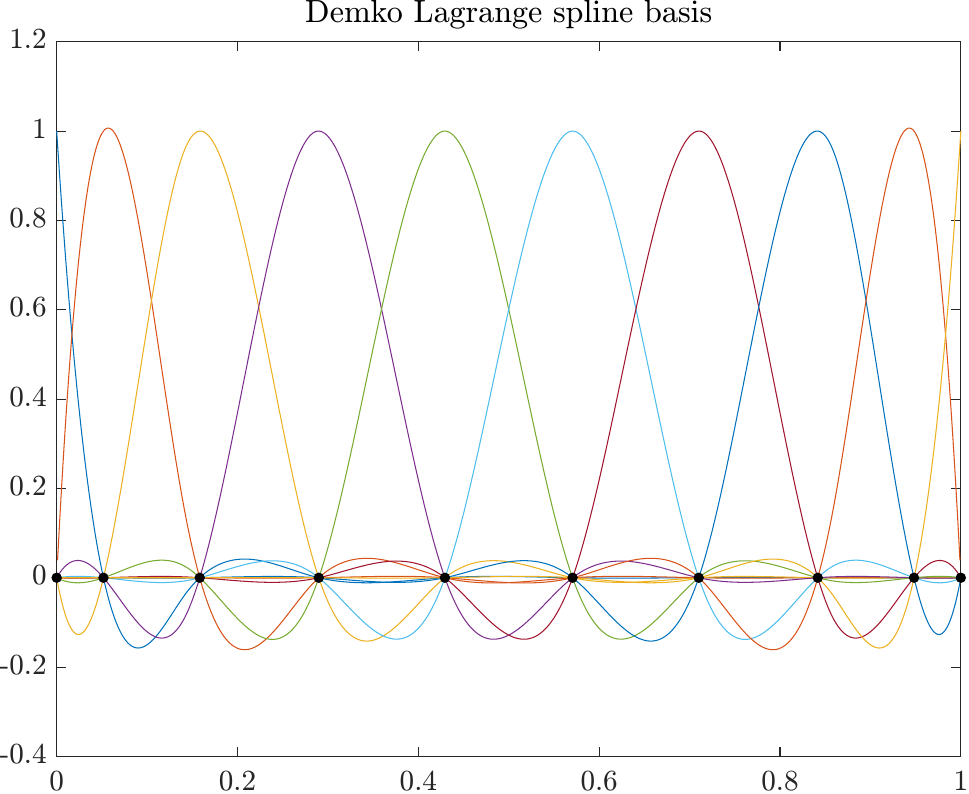}
    \caption{Demko Lagrange spline basis}
    \label{fig: 1D_Lagrange_Demko_basis_p3_ne7}
     \end{subfigure}
     \hfill
    \caption{Different spline bases. Black dots mark knot locations and interpolation points for the B-spline and Lagrange spline bases, respectively.}
    \label{fig: spline_bases}
\end{figure}

The properties of the Greville and Demko points are now closely examined on a series of examples.

\begin{example}[1D - Interpolation and quadrature error]
Our first example studies the convergence of the interpolation and quadrature operators for the smooth Runge function $f \colon [0,1] \to \mathbb{R}$ defined as
\begin{equation*}
    f(x) = \frac{1}{1+25x^2}
\end{equation*}
whose integral is $I(f)=\frac{1}{5}\arctan(5)$. The convergence of the interpolation error $\|f-Pf\|$ is shown in \Cref{fig: 1D_regular_mesh_interp_error} for the Greville and Demko interpolants of spline degrees $p$ ranging from $1$ to $6$ on uniformly refined meshes of mesh size $h$. For moderate spline degrees on uniform meshes, Greville interpolation is expected to remain stable and indeed, the error decays as $O(h^{p+1})$ in both cases, perfectly aligning with \Cref{lem: convergence_interpolation}. While this is expected to always hold for the Demko points, it might eventually break down for the Greville points on graded meshes with high degrees \cite{jia1988spline}. We indeed found numerical evidence of instabilities in the interpolant for degree $p \geq 45$. Furthermore, neither interpolants converge for a non-smooth function, as expected from \Cref{lem: convergence_interpolation}.

\begin{figure}[H]
     \centering
     \begin{subfigure}[t]{0.48\textwidth}
    \centering
    \includegraphics[width=\textwidth]{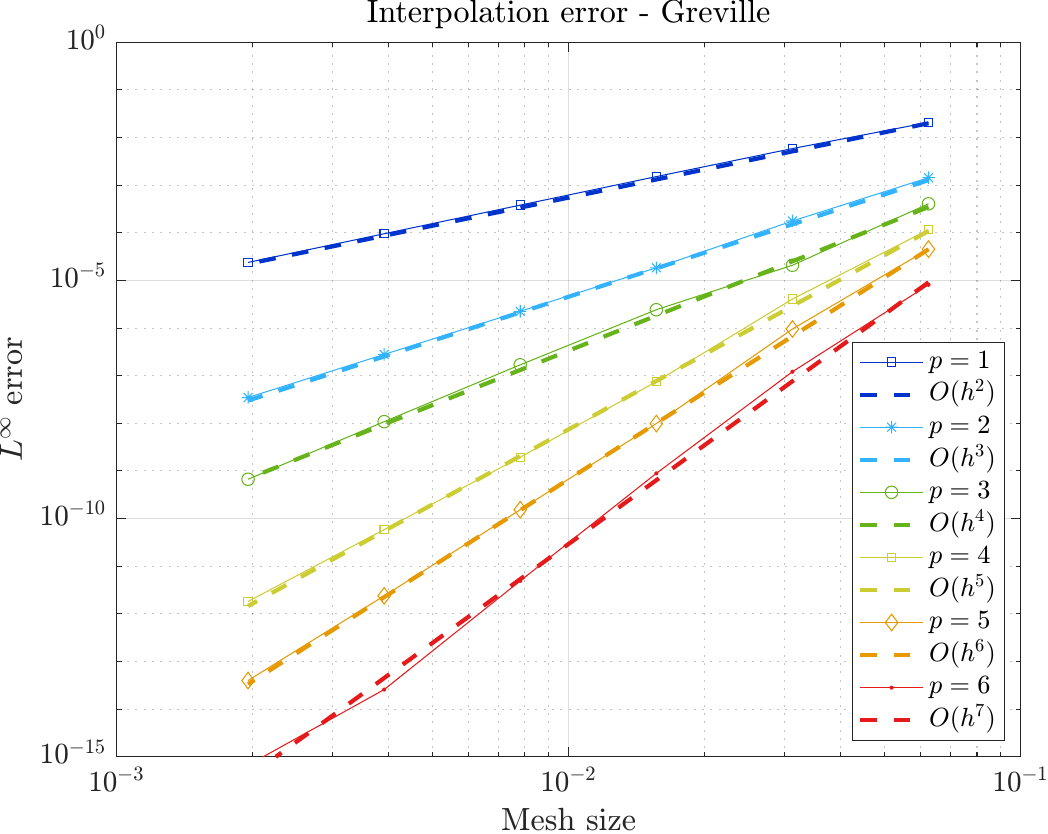}
    \caption{Greville points}
    \label{fig: 1D_regular_mesh_interp_error_Greville}
     \end{subfigure}
     \hfill
     \begin{subfigure}[t]{0.48\textwidth}
    \centering
    \includegraphics[width=\textwidth]{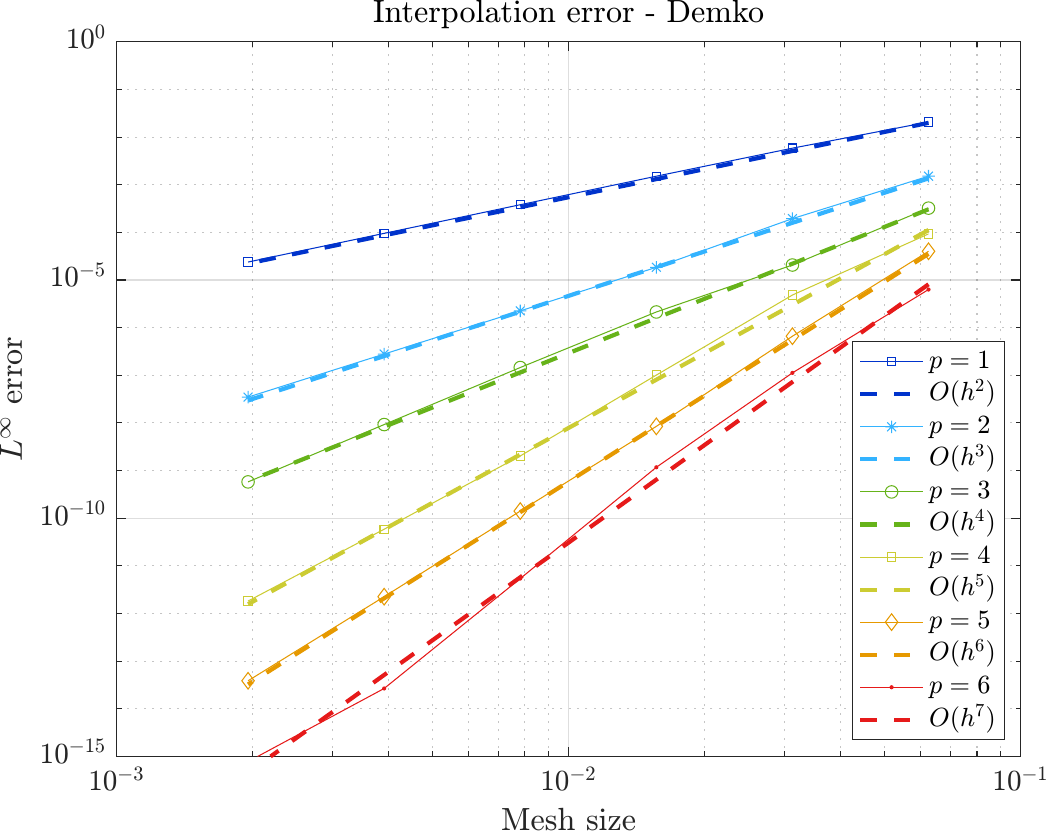}
    \caption{Demko points}
    \label{fig: 1D_regular_mesh_interp_error_Demko}
     \end{subfigure}
     \hfill
    \caption{Interpolation error $\|f-Pf\|$}
    \label{fig: 1D_regular_mesh_interp_error}
\end{figure}

What is more surprising is the quadrature error $|I(f)-Q(f)|$ shown in \Cref{fig: 1D_regular_mesh_quad_error}, which actually decays much faster than one would expect from \Cref{lem: convergence_quadrature}. In this figure and all following ones, the trendline tries to capture the most sensitive convergence pattern. In this case, the convergence rate follows a pattern reminiscent of collocation methods \cite{auricchio2010isogeometric,evans2018explicit}. We observe super-convergence for certain degrees and the expected additional order of convergence for splines of odd degree.

\begin{figure}[H]
     \centering
     \begin{subfigure}[t]{0.48\textwidth}
    \centering
    \includegraphics[width=\textwidth]{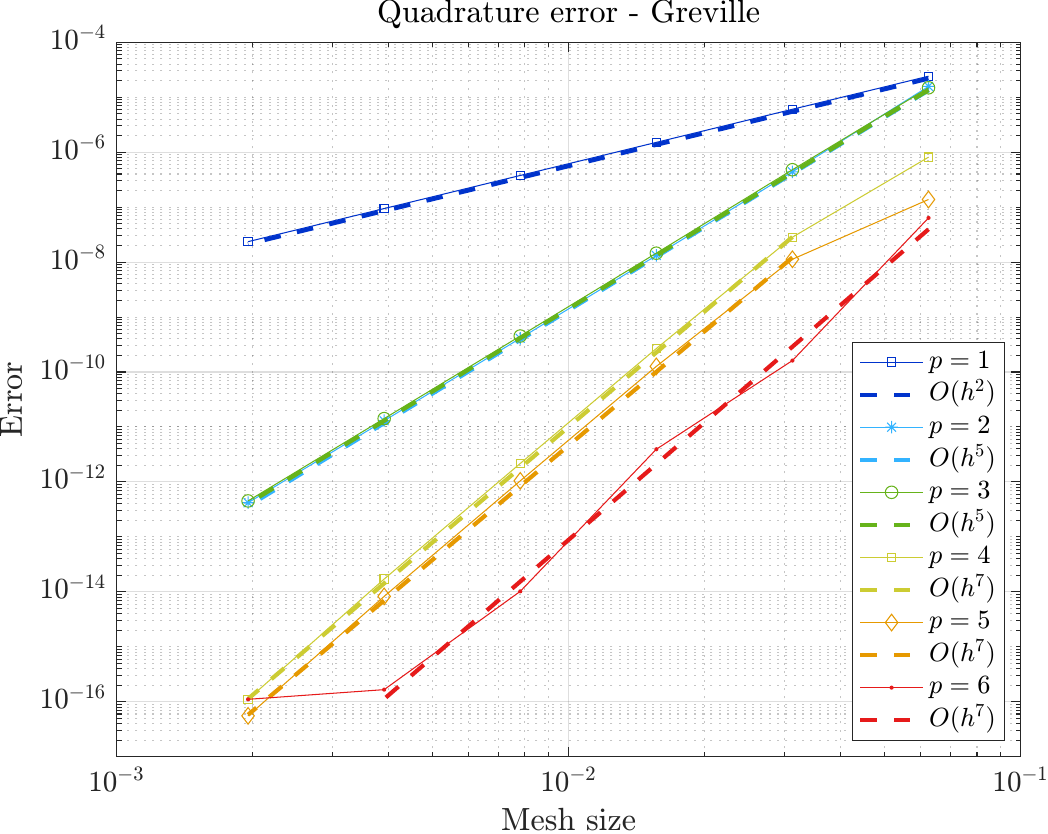}
    \caption{Greville points}
    \label{fig: 1D_regular_mesh_quad_error_Greville}
     \end{subfigure}
     \hfill
     \begin{subfigure}[t]{0.48\textwidth}
    \centering
    \includegraphics[width=\textwidth]{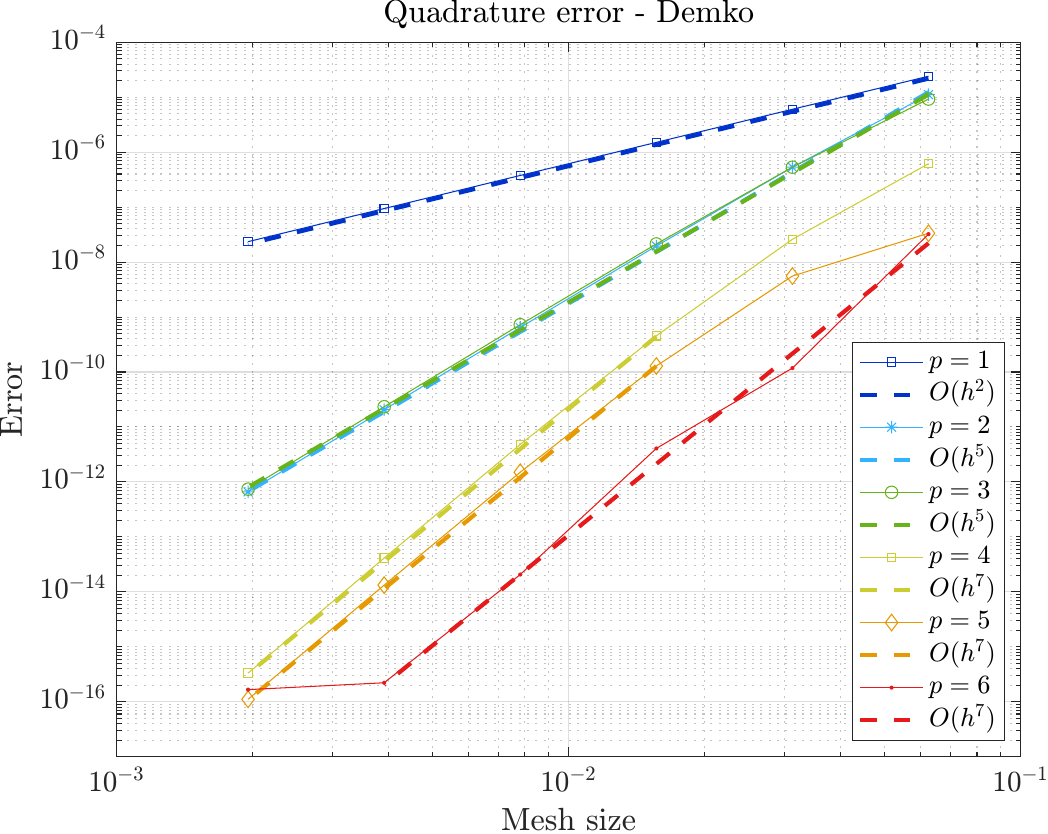}
    \caption{Demko points}
    \label{fig: 1D_regular_mesh_quad_error_Demko}
     \end{subfigure}
     \hfill
    \caption{Quadrature error $|I(f)-Q(f)|$}
    \label{fig: 1D_regular_mesh_quad_error}
\end{figure}
\end{example}

\begin{example}[1D - Eigenvalue error]
\label{ex: 1D_eig_convergence}
We now turn to the eigenvalues and eigenvectors of $(K_\dutchcal{B},\widehat{M}_\dutchcal{B})$, which are critical for the accuracy of dynamical simulations with lumped mass matrices. The convergence of the eigenvalues has already been extensively studied in \cite{li2022significance,li2024interpolatory} for the Greville points on different geometries with Dirichlet, Neumann or mixed boundary conditions and spline degrees ranging from $1$ to $5$. The conclusions of those studies were that the convergence rate of the smallest eigenfrequencies depended not only on the degree but also on the boundary conditions. For Neumann boundary conditions, the convergence rate was $2$ for degree $1$, $4$ for degrees $2$ and $3$ and $6$ for degrees $4$ and $5$. However, for Dirichlet or mixed boundary conditions, the convergence rate dropped to $5$ for degrees $4$ and $5$. Those patterns were systematically observed for 1D and 2D problems in \cite{li2022significance,li2024interpolatory}. The purpose of our experiment is firstly to confirm this trend, secondly to determine whether it also holds for the Demko points and thirdly to examine the convergence of the corresponding eigenfunctions. \Cref{fig: 1D_Laplace_eig4} shows the convergence rate for the relative error 
\begin{equation*}
    \left|\frac{\omega_4-\omega_{h,4}}{\omega_4}\right|
\end{equation*}
in the $4$th eigenfrequency for the Laplace eigenvalue problem on the unit line with homogeneous Dirichlet boundary conditions and spline degrees ranging from $1$ to $5$. In this simple academic setting, the eigenvalues and eigenfunctions are known in closed form (see e.g. \cite{manni2022application}). The $4$th frequency converges more slowly than the fundamental frequency, which allows analyzing the convergence over a broader range of mesh sizes. For a consistent mass approximation (\Cref{fig: 1D_Laplace_consistent_dirichlet_eig4}), the error converges at the expected rate of $2p$. Lumping the mass matrix in the B-spline basis reduces the convergence rate to $2$, independently of the spline degree (\Cref{fig: 1D_Laplace_ML_Bspline_dirichlet_eig4}). This trend has been consistently observed in several independent studies \cite{anitescu2019isogeometric,nguyen2023towards}, also for generalizations of the row-sum technique \cite{voet2023mathematical,voet2025mass}. In contrast, lumping the mass matrix for the Lagrange spline basis improves the convergence rate. For the Greville points, \Cref{fig: 1D_Laplace_ML_Greville_dirichlet_eig4} perfectly matches the findings in \cite{li2022significance,li2024interpolatory}. In \Cref{fig: 1D_Laplace_ML_Demko_dirichlet_eig4}, the convergence rate for the Demko points is almost the same, except that it increases to $6$ for degrees $4$ and $5$, even for Dirichlet boundary conditions. Similar results were obtained on the unit square. Thus, the convergence rate for the Demko points is apparently insensitive to the boundary conditions, providing an edge over the Greville points. However, the convergence of the corresponding eigenfunctions was not investigated in earlier work. The convergence of the error
\begin{equation*}
    \|u_4-u_{h,4}\|_{L^2}
\end{equation*}
in the $4$th (normalized) eigenfunction is shown in \Cref{fig: 1D_Laplace_eigf4}. As expected, the rate of convergence is $p+1$ for the consistent mass (\Cref{fig: 1D_Laplace_consistent_dirichlet_eigf4}) and drops to $2$ for the lumped mass in the B-spline basis (\Cref{fig: 1D_Laplace_ML_Bspline_dirichlet_eigf4}). Surprisingly, although it initially improves for the Greville points (\Cref{fig: 1D_Laplace_ML_Greville_dirichlet_eigf4}) and the Demko points (\Cref{fig: 1D_Laplace_ML_Demko_dirichlet_eigf4}), it apparently stalls at about $3.5$ in both cases. For pure Neumann boundary conditions, it does not increase beyond $4.5$ (for the range of degrees tested). This finding is a serious concern, as it may impede on the convergence of an initial boundary value problem. Furthermore, although the two sets seem nearly as good in this experiment, we have found numerous occurrences of negative weights for the Greville points on non-uniform meshes (see \Cref{se: non_uniform_meshes}) and the lumping strategy then leads to an unstable scheme. In contrast, we have never experienced negative weights for the Demko points, even for the most absurd knot vectors we could imagine. Some of our attempts are listed in \Cref{se: non_uniform_meshes}. The reason might be concealed in \Cref{conj: quadrature_condition}. That being said, the convergence rate for the Demko points remains sub-optimal with respect to the consistent mass (\Cref{fig: 1D_Laplace_consistent_dirichlet_eig4}).

\begin{figure}[H]
     \centering
     \begin{subfigure}[t]{0.48\textwidth}
    \centering
    \includegraphics[width=\textwidth]{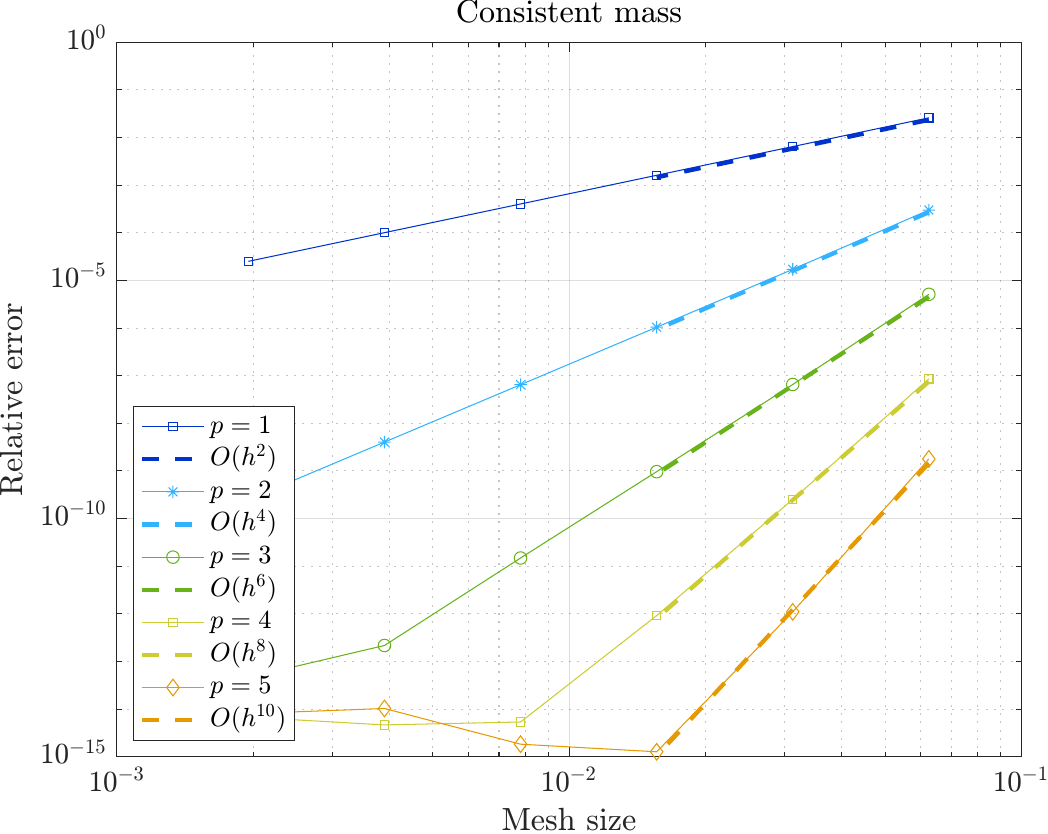}
    \caption{Consistent mass}
    \label{fig: 1D_Laplace_consistent_dirichlet_eig4}
     \end{subfigure}
     \hfill
     \begin{subfigure}[t]{0.48\textwidth}
    \centering
    \includegraphics[width=\textwidth]{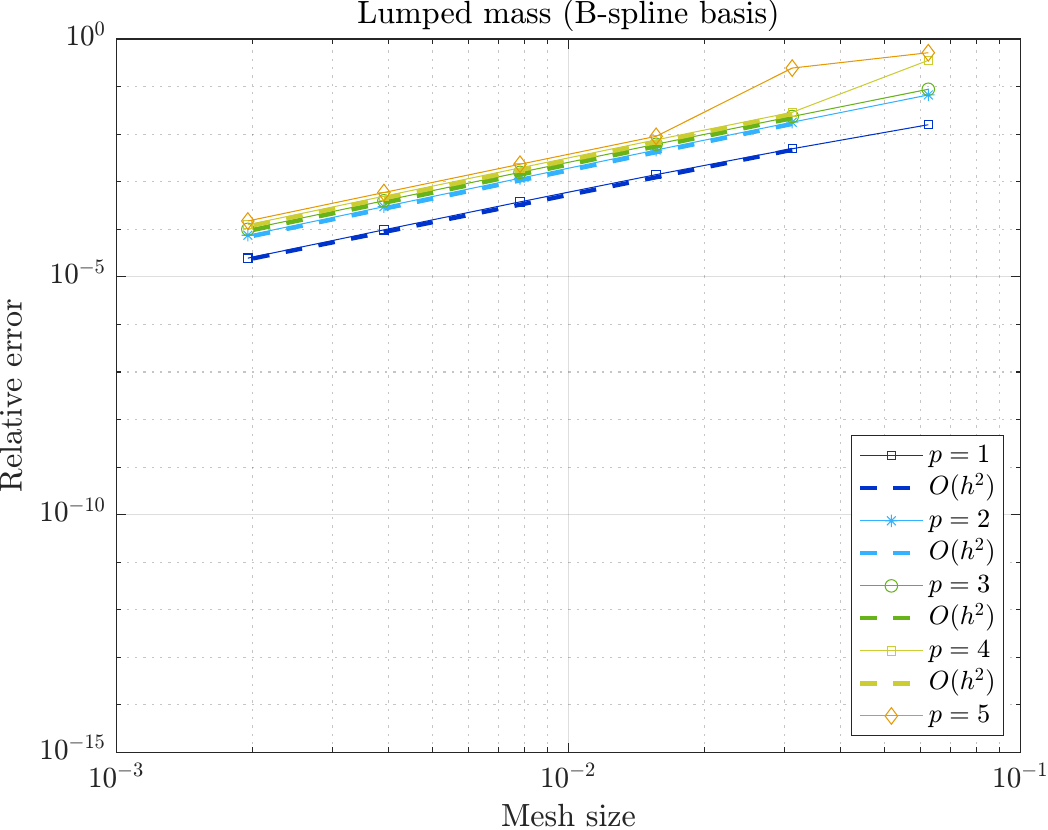}
    \caption{Lumped mass (B-spline basis)}
    \label{fig: 1D_Laplace_ML_Bspline_dirichlet_eig4}
     \end{subfigure}
     \hfill
     \vspace{5pt}
    \begin{subfigure}[t]{0.48\textwidth}
    \centering
    \includegraphics[width=\textwidth]{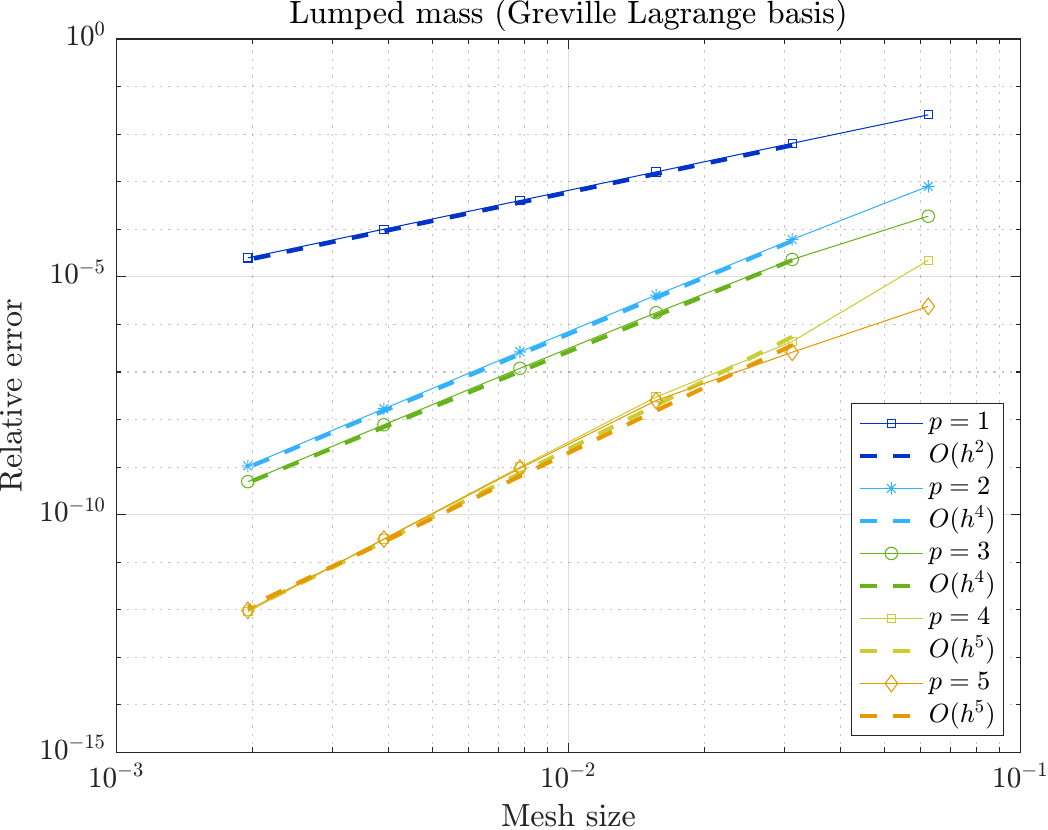}
    \caption{Lumped mass (Greville Lagrange spline basis)}
    \label{fig: 1D_Laplace_ML_Greville_dirichlet_eig4}
     \end{subfigure}
     \hfill
    \begin{subfigure}[t]{0.48\textwidth}
    \centering
    \includegraphics[width=\textwidth]{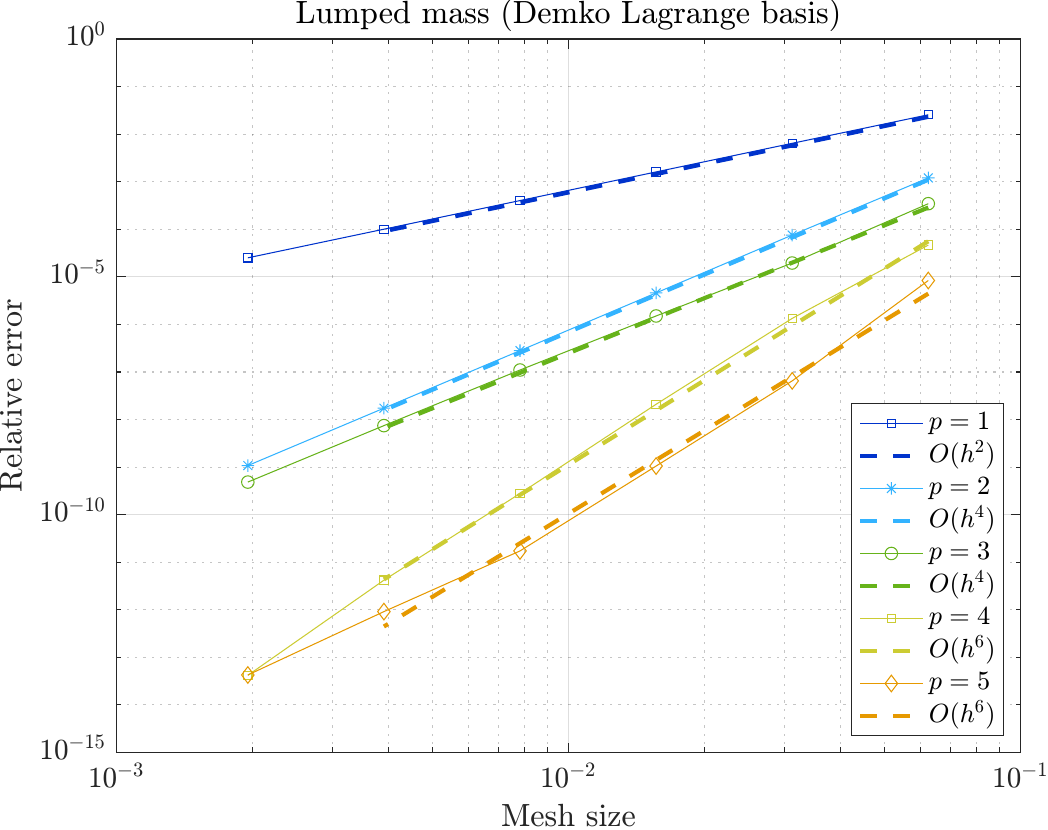}
    \caption{Lumped mass (Demko Lagrange spline basis)}
    \label{fig: 1D_Laplace_ML_Demko_dirichlet_eig4}
     \end{subfigure}
     \hfill
    \caption{Relative eigenfrequency error for the Laplace eigenvalue problem on the unit line with homogeneous Dirichlet boundary conditions}
    \label{fig: 1D_Laplace_eig4}
\end{figure}

\begin{figure}[H]
     \centering
     \begin{subfigure}[t]{0.48\textwidth}
    \centering
    \includegraphics[width=\textwidth]{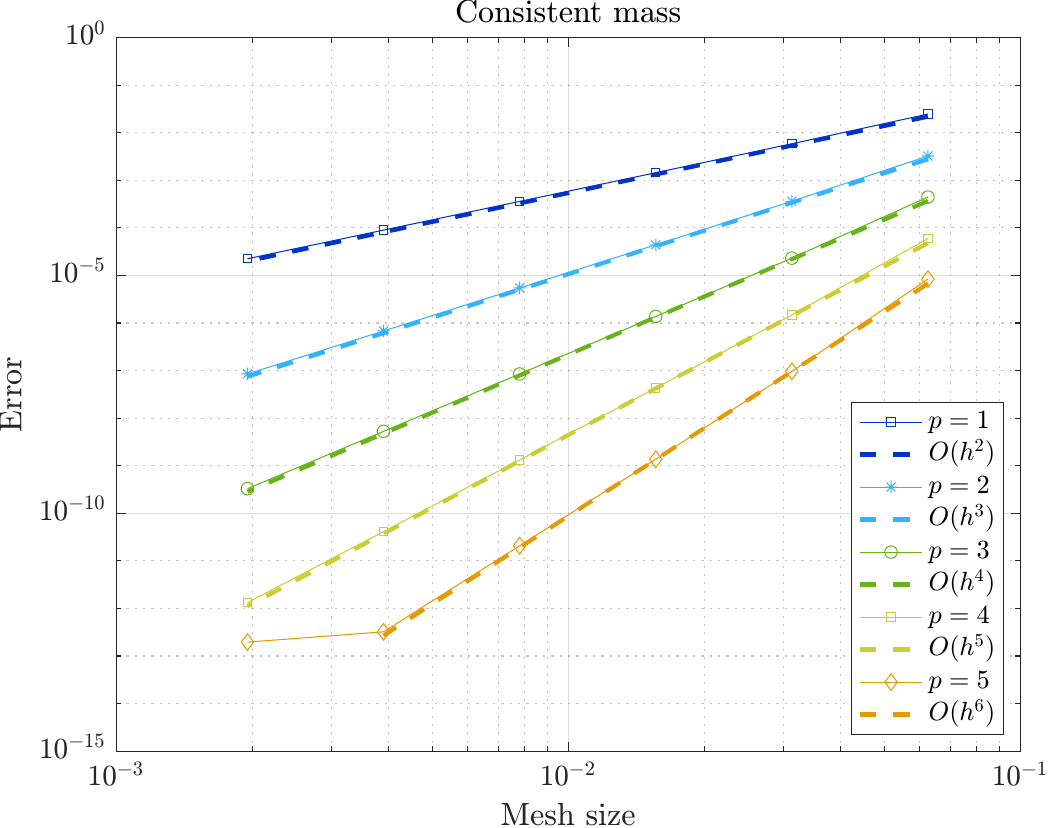}
    \caption{Consistent mass}
    \label{fig: 1D_Laplace_consistent_dirichlet_eigf4}
     \end{subfigure}
     \hfill
     \begin{subfigure}[t]{0.48\textwidth}
    \centering
    \includegraphics[width=\textwidth]{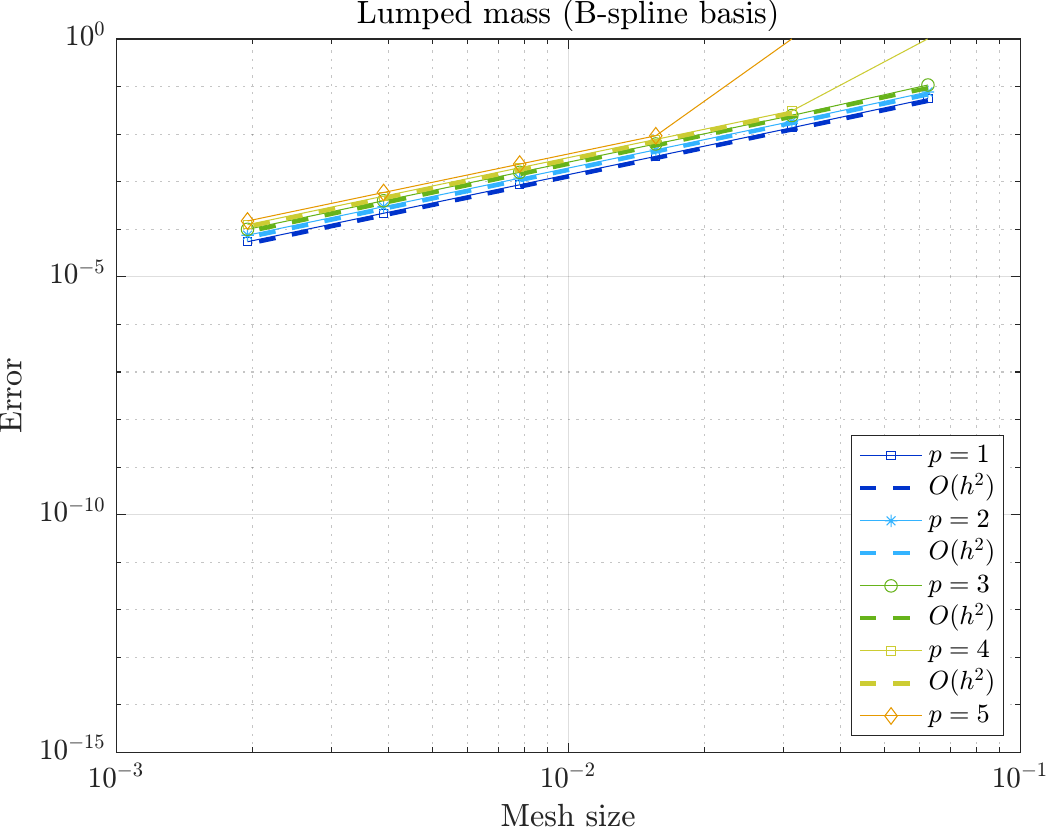}
    \caption{Lumped mass (B-spline basis)}
    \label{fig: 1D_Laplace_ML_Bspline_dirichlet_eigf4}
     \end{subfigure}
     \hfill
     \vspace{5pt}
    \begin{subfigure}[t]{0.48\textwidth}
    \centering
    \includegraphics[width=\textwidth]{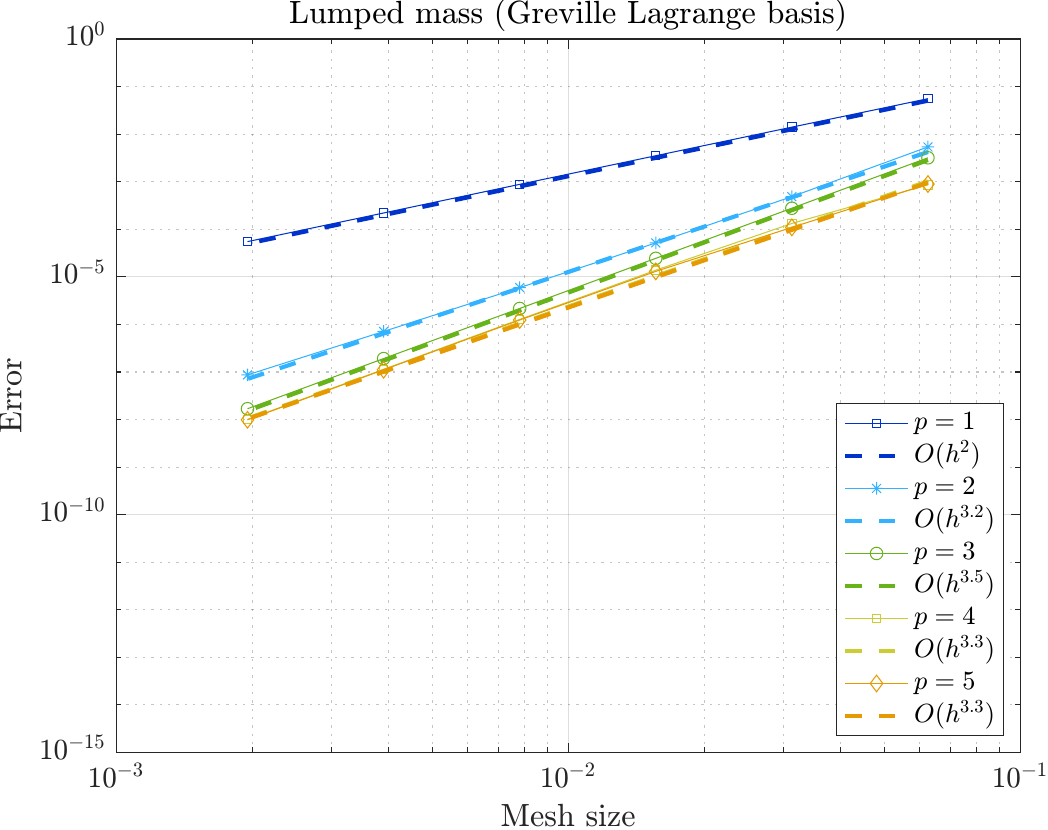}
    \caption{Lumped mass (Greville Lagrange spline basis)}
    \label{fig: 1D_Laplace_ML_Greville_dirichlet_eigf4}
     \end{subfigure}
     \hfill
    \begin{subfigure}[t]{0.48\textwidth}
    \centering
    \includegraphics[width=\textwidth]{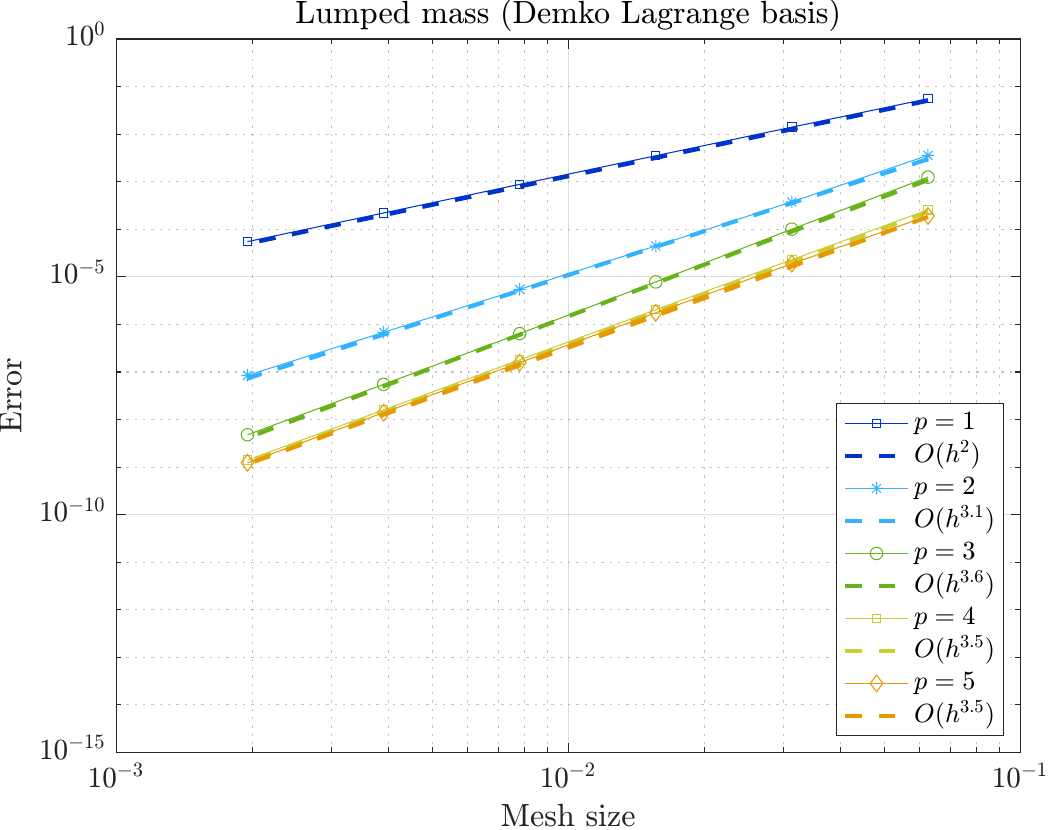}
    \caption{Lumped mass (Demko Lagrange spline basis)}
    \label{fig: 1D_Laplace_ML_Demko_dirichlet_eigf4}
     \end{subfigure}
     \hfill
    \caption{Eigenfunction error for the Laplace eigenvalue problem on the unit line with homogeneous Dirichlet boundary conditions}
    \label{fig: 1D_Laplace_eigf4}
\end{figure}
\end{example}

\begin{example}[1D - Dynamics]
\label{ex: 1D_dynamics}
The improved accuracy of mass lumping for interpolatory spline bases should logically translate into improved accuracy for an initial boundary value problem. For exploring the properties of the method, we consider two distinct manufactured solutions $u_i(x,t)=w_i(x)\sin(\omega t)$ for $i=1,2$, which only differ in their spatial part given by
\begin{equation*}
    w_1(x)= \sin(\omega x) \quad \text{and} \quad w_2(x)=x(1-x)\mathrm{e}^{-\left(\frac{x-x_c}{\sigma}\right)^2}
\end{equation*}
with parameter values $\sigma=0.1$, $x_c=0.5$ and $\omega = 3\pi$. The functions $w_1(x)$ and $w_2(x)$ are shown in \Cref{fig: 1D_Laplace_dynamics_func}.

\begin{figure}[H]
     \centering
     \begin{subfigure}[t]{0.48\textwidth}
    \centering
    \includegraphics[width=\textwidth]{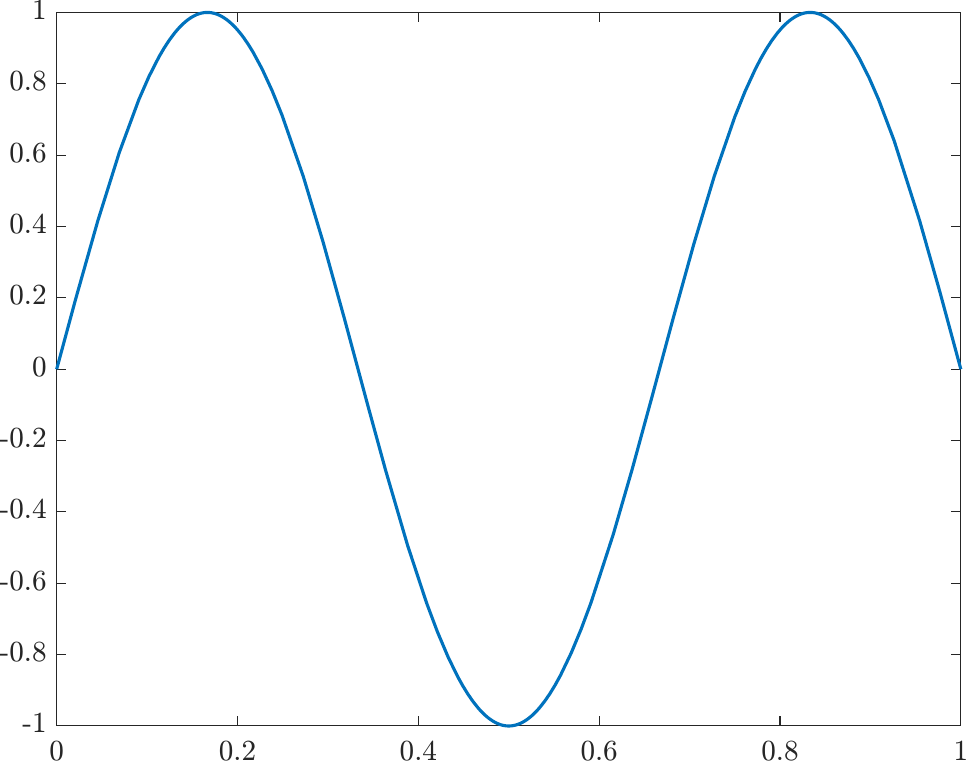}
    \caption{Function $w_1(x)$}
    \label{fig: 1D_Laplace_dynamics_easy_func}
     \end{subfigure}
     \hfill
     \begin{subfigure}[t]{0.48\textwidth}
    \centering
    \includegraphics[width=\textwidth]{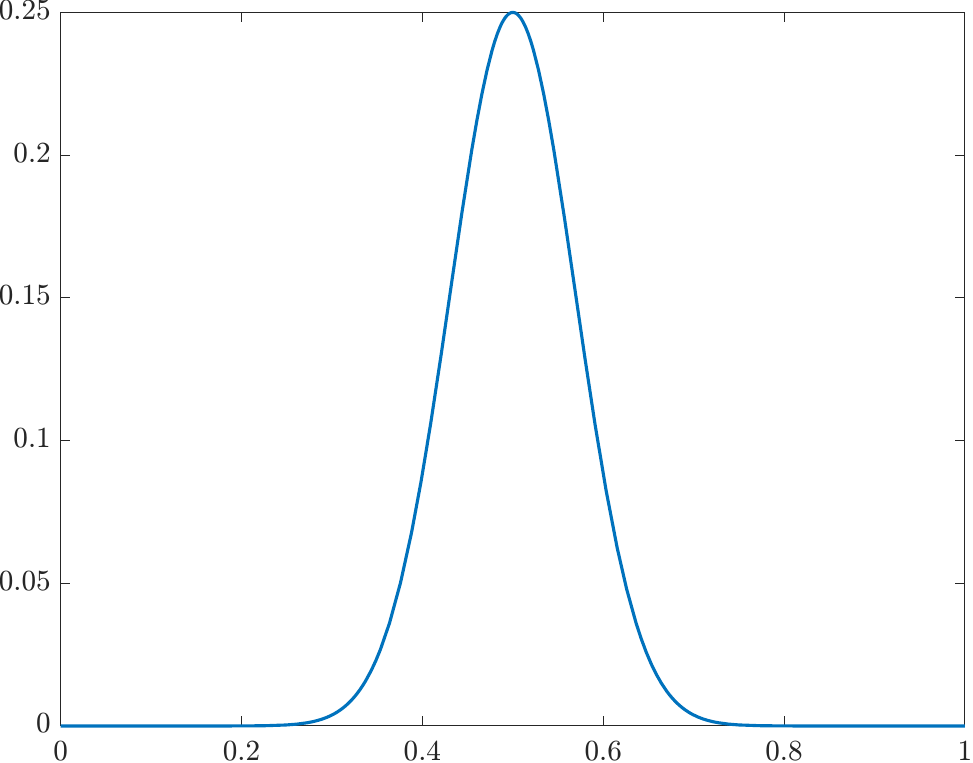}
    \caption{Function $w_2(x)$}
    \label{fig: 1D_Laplace_dynamics_diff_func}
     \end{subfigure}
     \hfill
    \caption{Spatial part of the manufactured solutions for \Cref{ex: 1D_dynamics}}
    \label{fig: 1D_Laplace_dynamics_func}
\end{figure}

We now solve the standard wave equation on the unit line with unit material coefficients, homogeneous Dirichlet boundary conditions and where the right-hand side and initial conditions are obtained from the manufactured solutions. For the spatial discretization, spline spaces of degree $1$ to $5$ are built on increasingly fine meshes. In general, in order to balance the spatial and temporal discretization errors, high order spatial discretizations must be coupled with high order temporal ones, otherwise the temporal discretization error will drive the convergence rate. In principle, one could also use low order methods by simply adapting the step size depending on the spatial accuracy. However, this technique soon becomes computationally intractable due to the exceedingly small step sizes and high order explicit methods offer a better alternative. Such techniques will be investigated in future examples but for the time being, it is safer to analyze the spatial discretization error without any interference from the time discretization. Fortunately, for the manufactured solutions we have chosen, the exact solution of the semi-discrete problem \eqref{eq: semi_discrete_pb} has a closed form solution (see e.g. \cite{voet2025stabilization}). Hence, we analyze the convergence of the exact semi-discrete solution $u_h(x,t)$ by computing the relative $L^2$ error at discrete times
\begin{equation*}
    \frac{\|u(t_j)-u_h(t_j)\|_{L^2}}{\|u(t_j)\|_{L^2}},
\end{equation*}
where $t_j = j \Delta t$. The convergence of the relative error at the final time $T=1.5$ is examined for both manufactured solutions for a consistent mass (\Cref{fig: 1D_Laplace_dynamics_consistent}) and a lumped mass approximation, either directly in the B-spline basis (\Cref{fig: 1D_Laplace_dynamics_ML_Bspline}) or in the interpolatory spline basis based on the Greville points (\Cref{fig: 1D_Laplace_dynamics_ML_Greville}) or the Demko points (\Cref{fig: 1D_Laplace_dynamics_ML_Demko}). For the consistent mass, the error converges at the optimal rate, independently of the solution. For the lumped mass in the B-spline basis, the convergence rate drops to $2$ in both cases. However, when lumping the mass matrix in the interpolatory spline basis, the convergence rate is apparently solution-dependent. Similarly to the eigenfunctions, the convergence rate for $u_1$ initially improves but eventually stalls at about $3.5$. Choosing instead $w_1(x)=\cos(\omega x)$ and imposing Neumann boundary conditions increases the convergence rate to about $4.5$. However, the convergence rate for $u_2$ is surprisingly optimal and matches the consistent mass. This difference is certainly not related to a problem in the treatment of the boundary conditions: the numerical solutions are exactly zero at the Dirichlet boundary as they should be. Instead, an intuitive explanation lies in the behavior of the solution near the boundaries and, more specifically, in the number of vanishing derivatives. By construction, both solutions vanish at the boundary. However, while the first derivative of $w_1$ is nonzero, $w_2$ has infinitely many (near) zero derivatives at the boundaries and the discretization error in those regions is damped out. As a matter of fact, improved rates of convergence are also witnessed for $u_1$ if the $L^2$ error is only computed over an ``interior subdomain'' away from the boundaries. For instance, \Cref{fig: 1D_Laplace_dynamics_ML_easy_func_sub} shows the relative $L^2$ error over the $[0.1, 0.9]$ sub-interval. The convergence rates for the consistent and lumped mass in the B-spline basis remain unchanged and the corresponding figures are omitted. We do not have a clear theoretical explanation for this phenomenon but it is certainly related to the Greville and Demko points near the boundary. A boundary correction scheme similar to what was done in \cite[Section~5]{manni2022application} for different boundary conditions might recover optimal convergence rates for \emph{all} smooth solutions, but this is a topic for future work. We note that a similar increase of the convergence rate with the number of zero derivatives of the true solution at the boundary was proven to occur in the numerical quadrature solution to the integral equation studied in \cite{Bressan2020}.

Furthermore, although this example did not require any time discretization scheme, we still compared the CFL conditions. Unfortunately, the time integration of the lumped mass solution for the Greville Lagrange spline basis would occasionally require more time steps than for the consistent mass. We later encountered the same issue for the Demko Lagrange spline basis on a different example. Thus, lumping the mass matrix for the Lagrange spline basis may deteriorate the CFL condition, even for very simple academic examples.

\begin{figure}[H]
     \centering
     \begin{subfigure}[t]{0.48\textwidth}
    \centering
    \includegraphics[width=\textwidth]{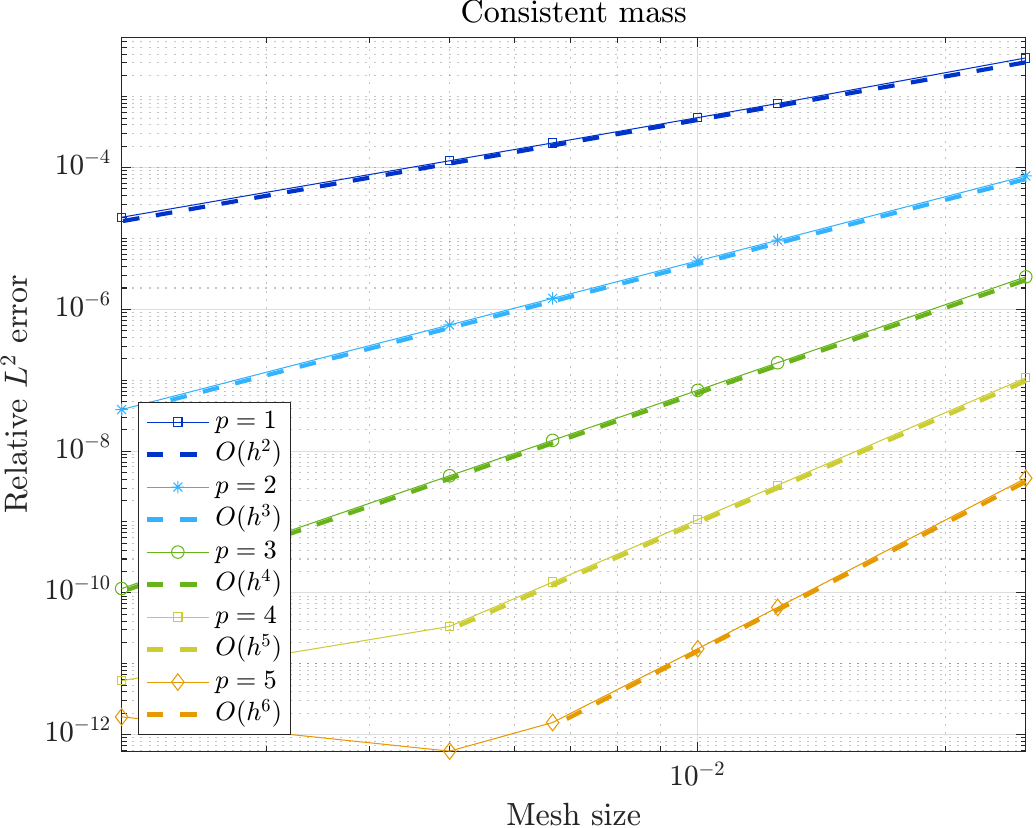}
    \caption{Error for $u_1(x,t)$}
    \label{fig: 1D_Laplace_dynamics_consistent_easy_func}
     \end{subfigure}
     \hfill
     \begin{subfigure}[t]{0.48\textwidth}
    \centering
    \includegraphics[width=\textwidth]{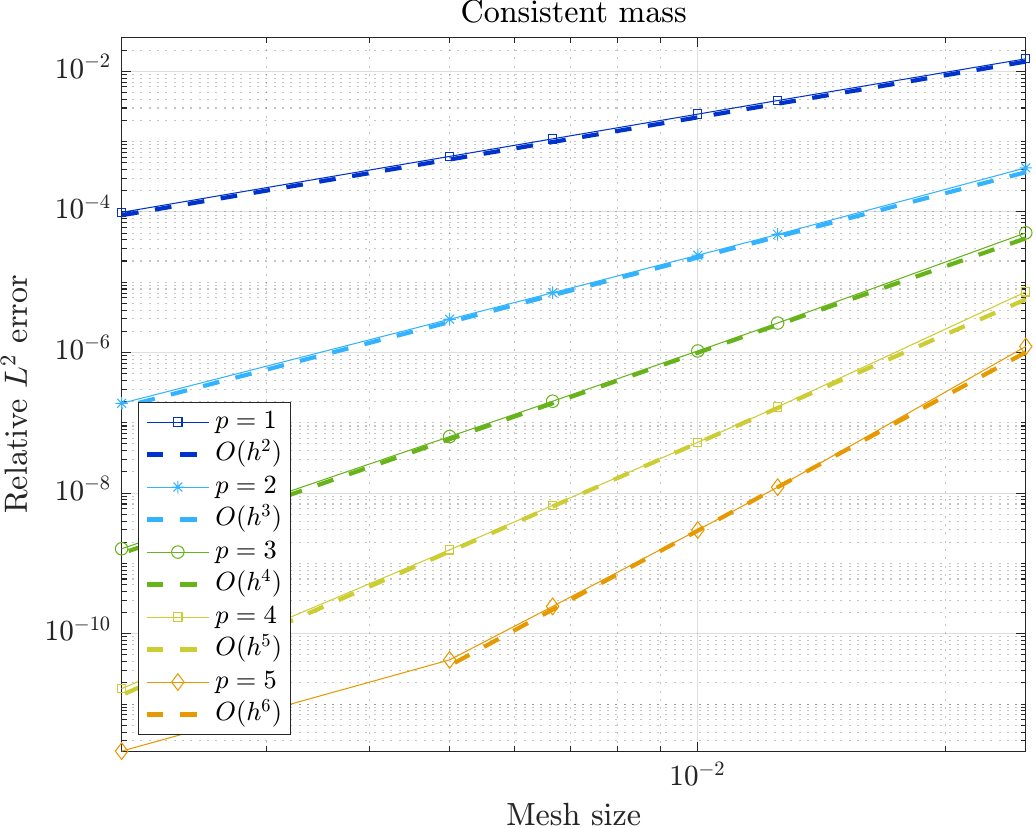}
    \caption{Error for $u_2(x,t)$}
    \label{fig: 1D_Laplace_dynamics_consistent_diff_func}
     \end{subfigure}
     \hfill
    \caption{Consistent mass}
    \label{fig: 1D_Laplace_dynamics_consistent}
\end{figure}

\begin{figure}[H]
     \centering
     \begin{subfigure}[t]{0.48\textwidth}
    \centering
    \includegraphics[width=\textwidth]{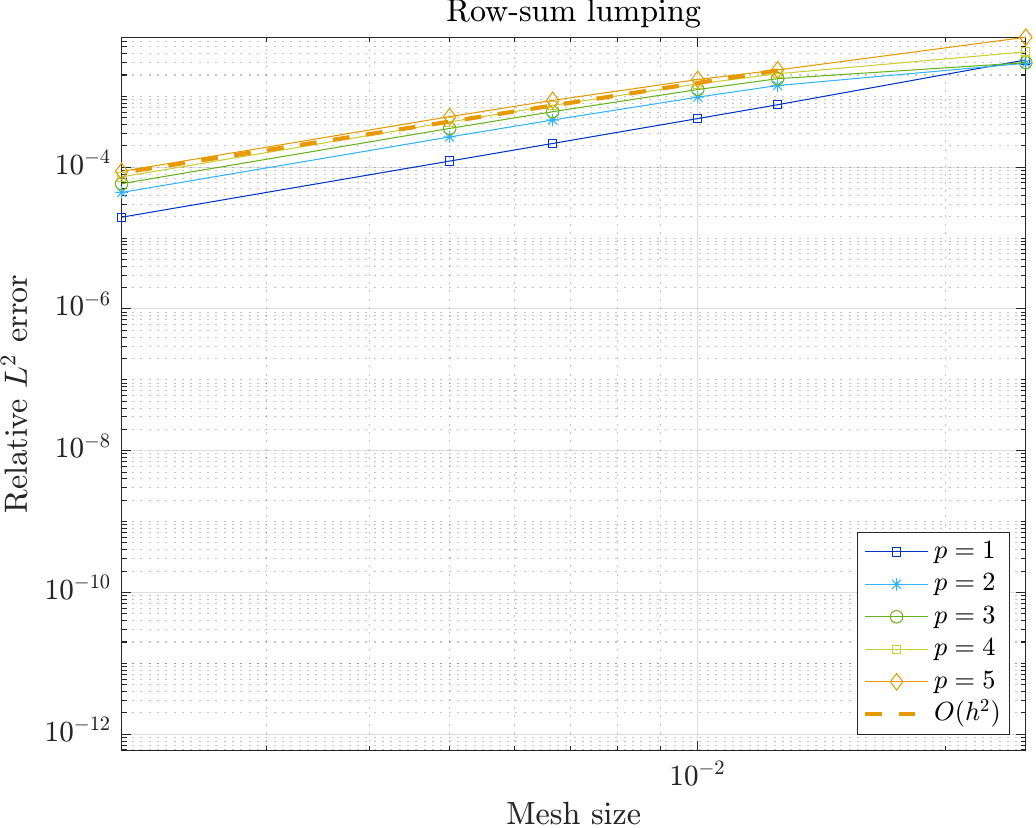}
    \caption{Error for $u_1(x,t)$}
    \label{fig: 1D_Laplace_dynamics_ML_Bspline_easy_func}
     \end{subfigure}
     \hfill
     \begin{subfigure}[t]{0.48\textwidth}
    \centering
    \includegraphics[width=\textwidth]{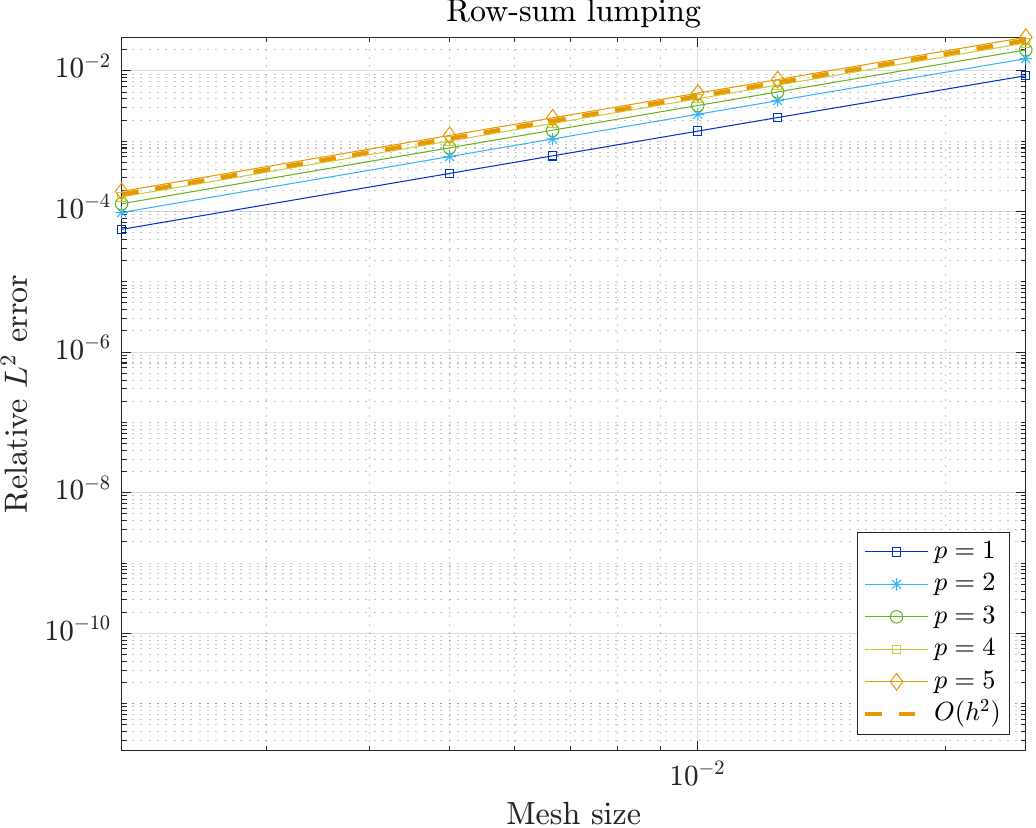}
    \caption{Error for $u_2(x,t)$}
    \label{fig: 1D_Laplace_dynamics_ML_Bspline_diff_func}
     \end{subfigure}
     \hfill
    \caption{Lumped mass (B-spline basis)}
    \label{fig: 1D_Laplace_dynamics_ML_Bspline}
\end{figure}

\begin{figure}[H]
     \centering
     \begin{subfigure}[t]{0.48\textwidth}
    \centering
    \includegraphics[width=\textwidth]{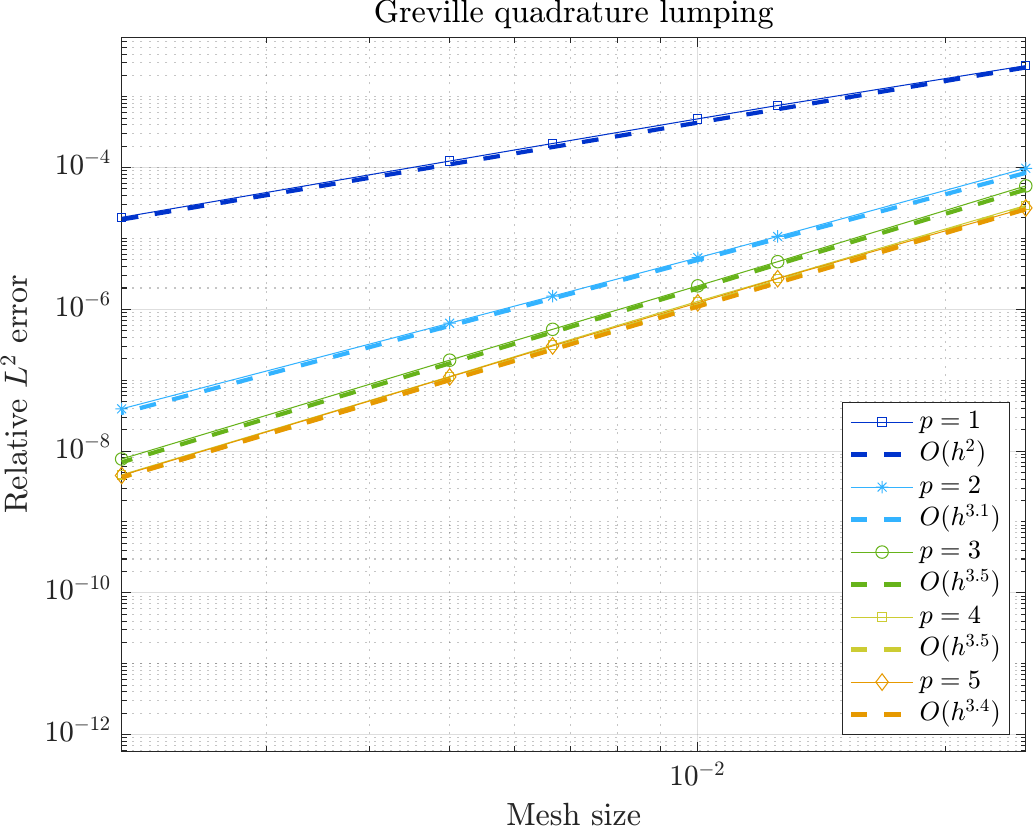}
    \caption{Error for $u_1(x,t)$}
    \label{fig: 1D_Laplace_dynamics_ML_Greville_easy_func}
     \end{subfigure}
     \hfill
     \begin{subfigure}[t]{0.48\textwidth}
    \centering
    \includegraphics[width=\textwidth]{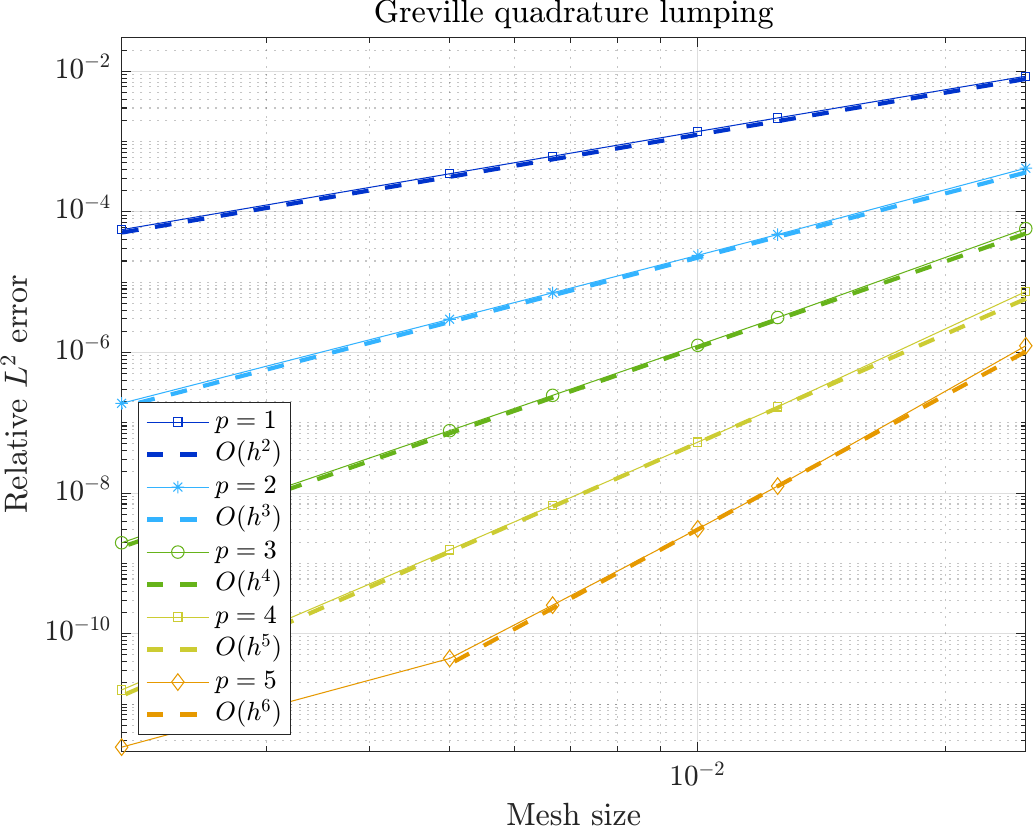}
    \caption{Error for $u_2(x,t)$}
    \label{fig: 1D_Laplace_dynamics_ML_Greville_diff_func}
     \end{subfigure}
     \hfill
    \caption{Lumped mass (Greville Lagrange spline basis)}
    \label{fig: 1D_Laplace_dynamics_ML_Greville}
\end{figure}

\begin{figure}[H]
     \centering
     \begin{subfigure}[t]{0.48\textwidth}
    \centering
    \includegraphics[width=\textwidth]{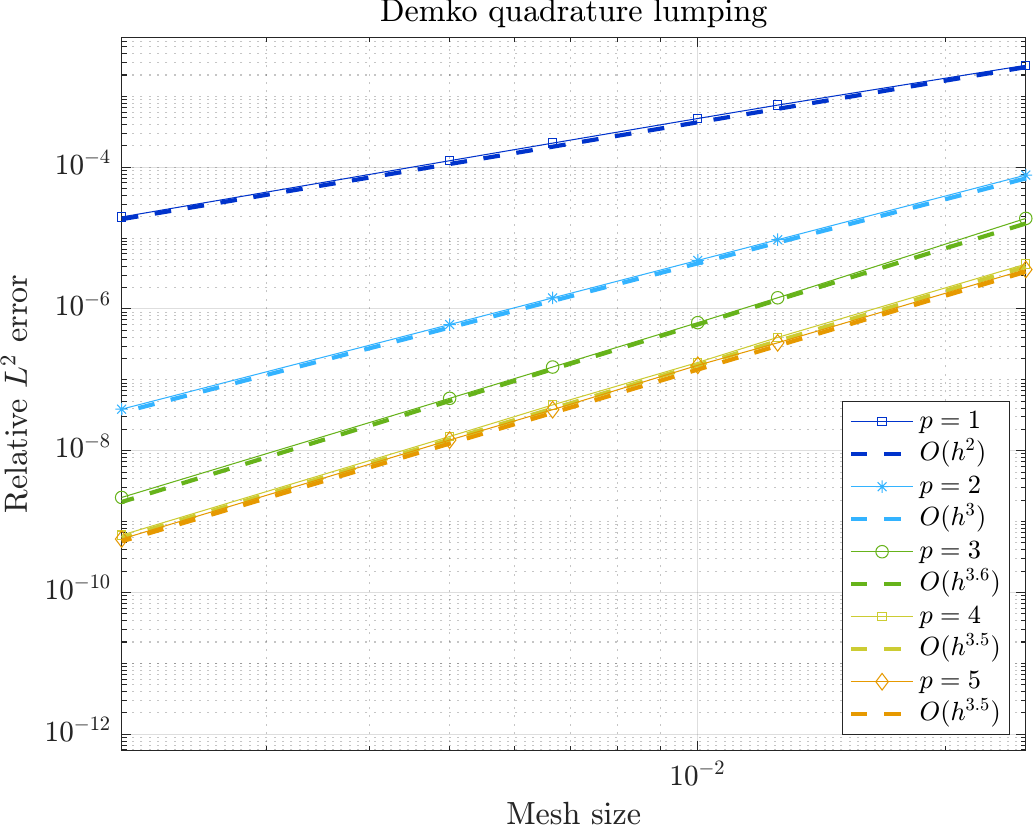}
    \caption{Error for $u_1(x,t)$}
    \label{fig: 1D_Laplace_dynamics_ML_Demko_easy_func}
     \end{subfigure}
     \hfill
     \begin{subfigure}[t]{0.48\textwidth}
    \centering
    \includegraphics[width=\textwidth]{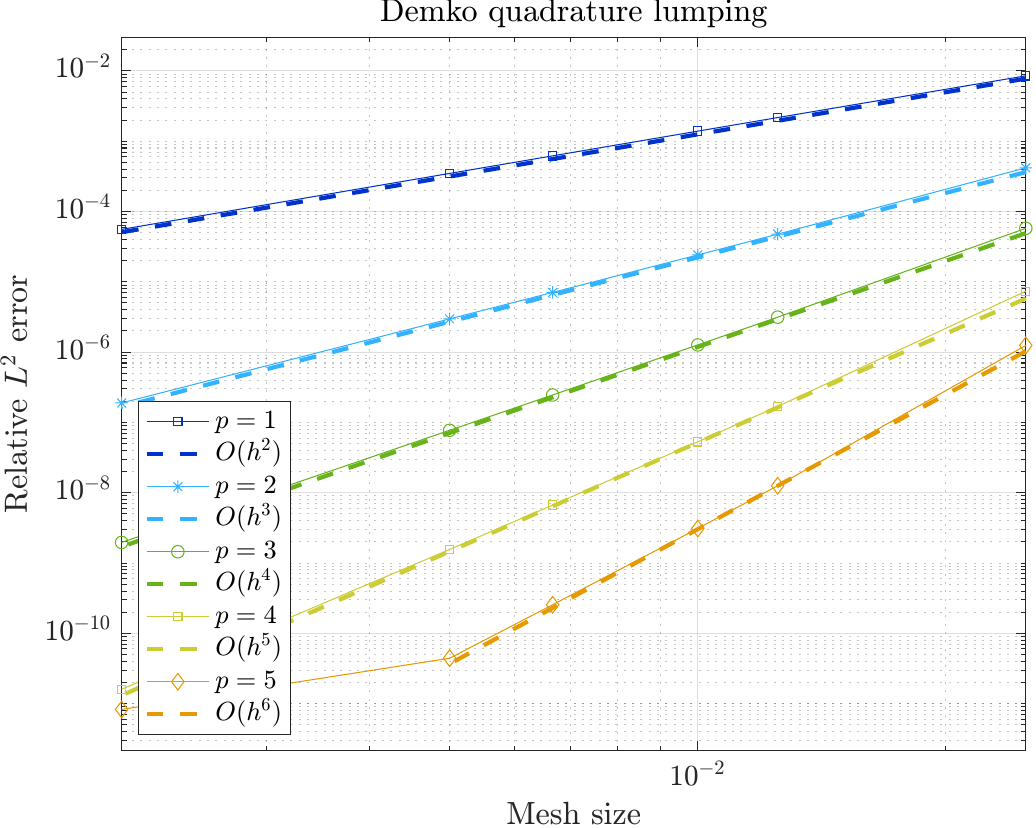}
    \caption{Error for $u_2(x,t)$}
    \label{fig: 1D_Laplace_dynamics_ML_Demko_diff_func}
     \end{subfigure}
     \hfill
    \caption{Lumped mass (Demko Lagrange spline basis)}
    \label{fig: 1D_Laplace_dynamics_ML_Demko}
\end{figure}

\begin{figure}[H]
     \centering
     \begin{subfigure}[t]{0.48\textwidth}
    \centering
    \includegraphics[width=\textwidth]{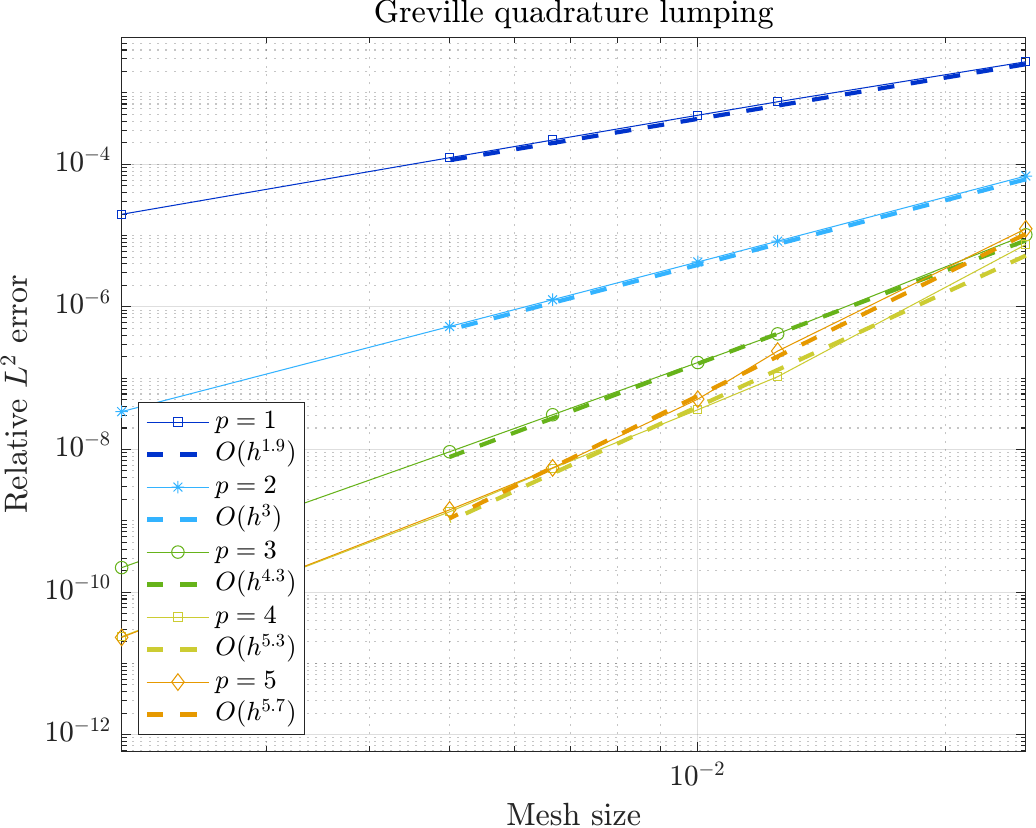}
    \caption{Greville Lagrange spline basis}
    \label{fig: 1D_Laplace_dynamics_ML_Greville_easy_func_sub}
     \end{subfigure}
     \hfill
     \begin{subfigure}[t]{0.48\textwidth}
    \centering
    \includegraphics[width=\textwidth]{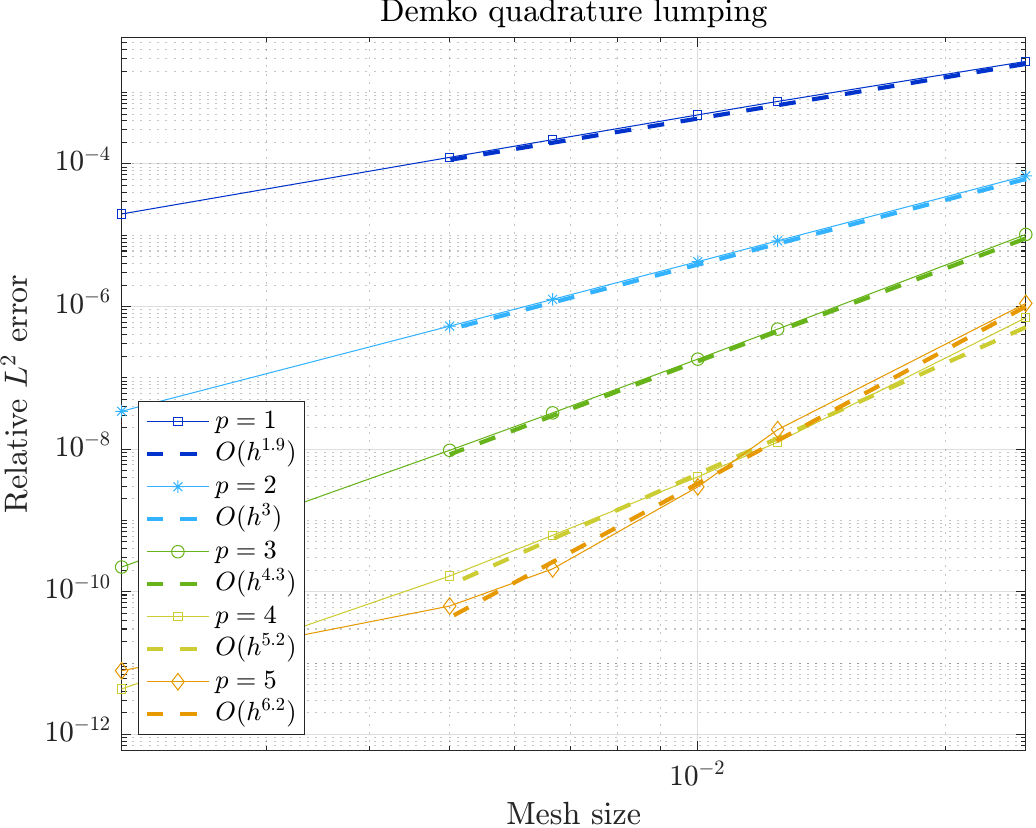}
    \caption{Demko Lagrange spline basis}
    \label{fig: 1D_Laplace_dynamics_ML_Demko_easy_func_sub}
     \end{subfigure}
     \hfill
    \caption{Relative $L^2$ error for $u_1(x,t)$ computed over the sub-interval $[0.1, 0.9]$}
    \label{fig: 1D_Laplace_dynamics_ML_easy_func_sub}
\end{figure}
\end{example}

\begin{example}[1D - Dynamics]
\label{ex: 1D_dynamics_mixed}
To complete the previous example, we consider an initial boundary value problem with mixed boundary conditions instead. The manufactured solution is in this case $u(x,t)=w(x)\sin(\omega t)$ where
\begin{equation}
\label{eq: function_q(x)}
    w(x) = c^{x^a}x\cos\left(\frac{\pi}{x_l-x}\right)
\end{equation}
with parameters $x_l=9/8$, $c=8$, $a=10$ and $\omega=3\pi$. The function $w(x)$ shown in \Cref{fig: 1D_Laplace_dynamics_func_mixed_bc} satisfies homogeneous Dirichlet boundary conditions on the left boundary and we prescribe Neumann boundary conditions on the right boundary. Imposing those boundary conditions does not require any special attention. Thanks to the interpolatory nature of the basis functions, the Dirichlet boundary condition is treated as in standard FEM, SEM and IGA, while the Neumann boundary condition is naturally incorporated in the right-hand side. Similarly to the previous example, we study the convergence of the relative $L^2$ error at the final time $T=1.5$. The results for the consistent mass (\Cref{fig: 1D_Laplace_dynamics_consistent_mixed_bc}) and the lumped mass, whether in the B-spline basis (\Cref{fig: 1D_Laplace_dynamics_ML_Bspline_mixed_bc}) or in the interpolatory spline bases (\Cref{fig: 1D_Laplace_dynamics_ML_Greville_mixed_bc,fig: 1D_Laplace_dynamics_ML_Demko_mixed_bc}) are very similar to those of $u_2$ in \Cref{ex: 1D_dynamics}. In particular, the numerical solutions for the lumped mass in the interpolatory spline bases converge at the optimal rate, except for the Greville Lagrange spline basis with the highest degree ($p=5$).

\begin{figure}[H]
    \centering
    \includegraphics[scale=0.5]{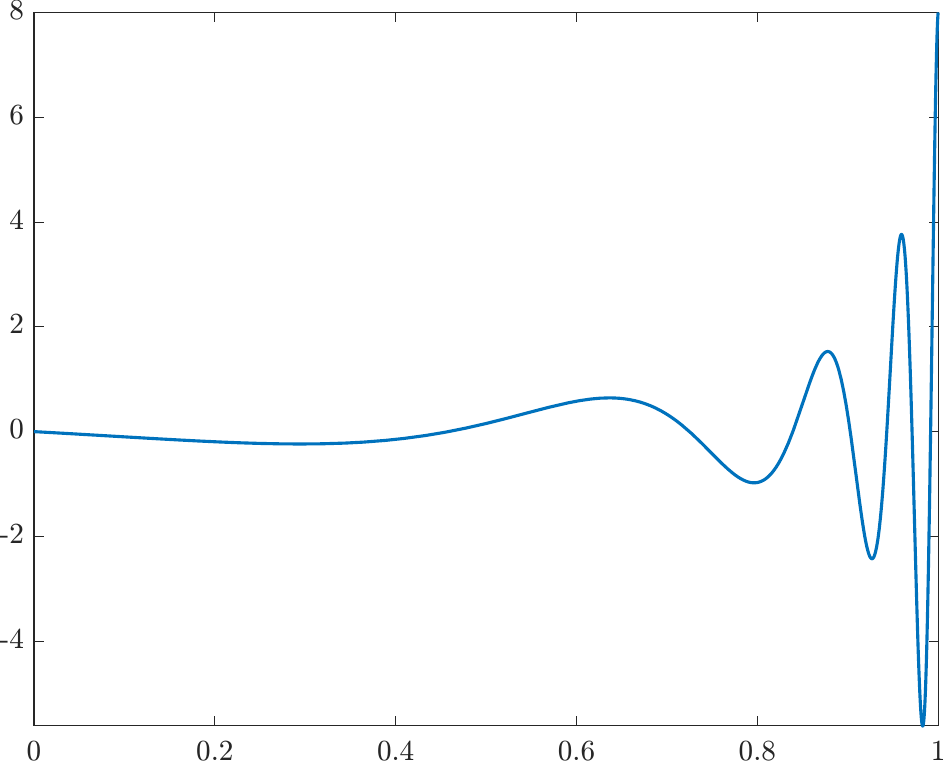}
    \caption{Function $w(x)$ for \Cref{ex: 1D_dynamics_mixed}}
    \label{fig: 1D_Laplace_dynamics_func_mixed_bc}
\end{figure}

\begin{figure}[H]
     \centering
     \begin{subfigure}[t]{0.48\textwidth}
    \centering
    \includegraphics[width=\textwidth]{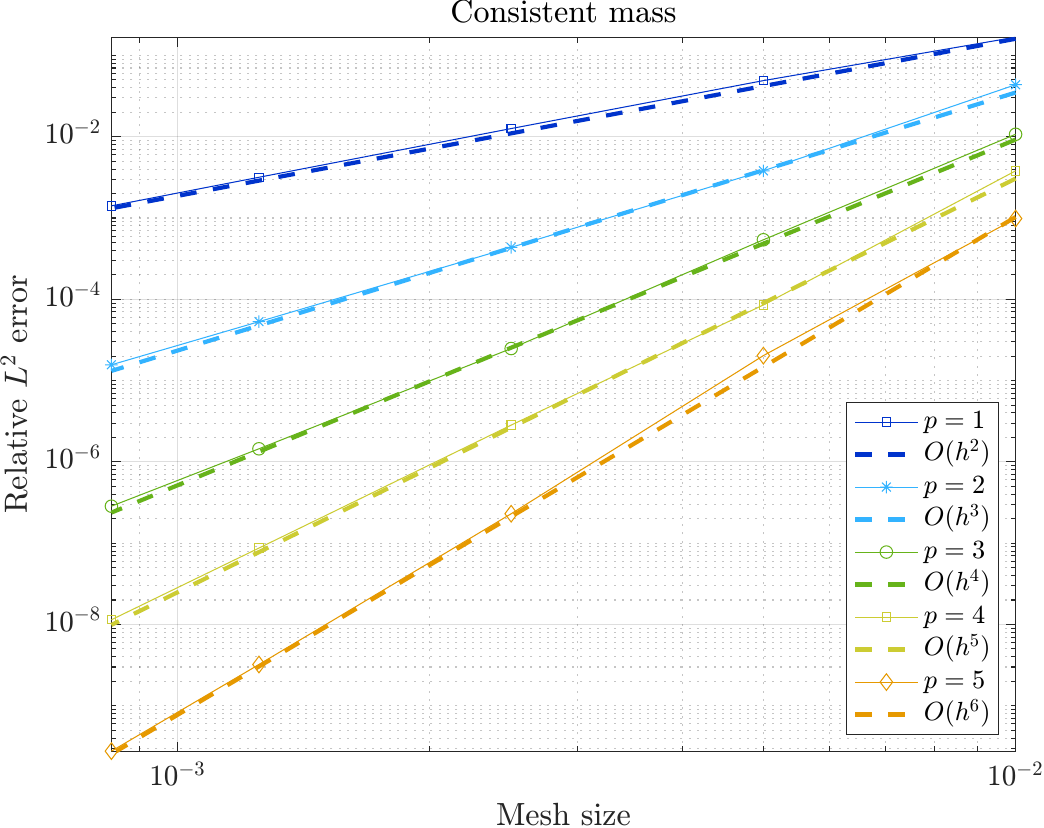}
    \caption{Consistent mass}
    \label{fig: 1D_Laplace_dynamics_consistent_mixed_bc}
     \end{subfigure}
     \hfill
     \begin{subfigure}[t]{0.48\textwidth}
    \centering
    \includegraphics[width=\textwidth]{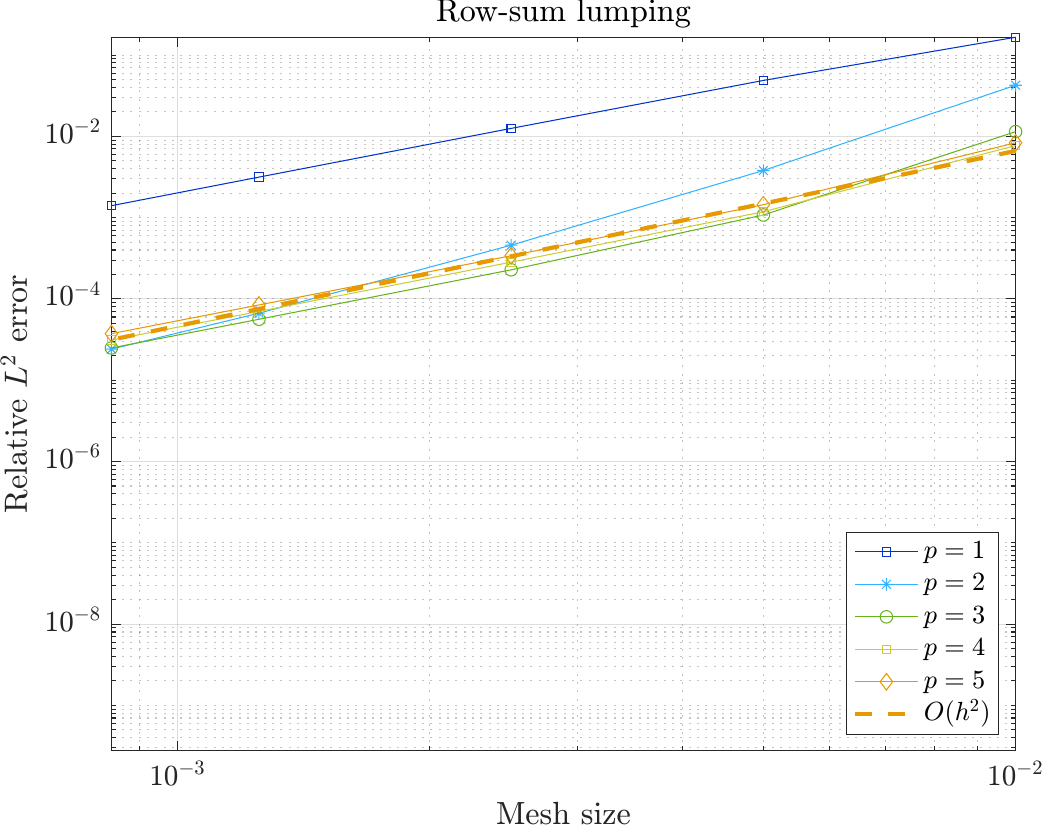}
    \caption{Lumped mass (B-spline basis)}
    \label{fig: 1D_Laplace_dynamics_ML_Bspline_mixed_bc}
     \end{subfigure}
     \hfill
     \vspace{5pt}
    \begin{subfigure}[t]{0.48\textwidth}
    \centering
    \includegraphics[width=\textwidth]{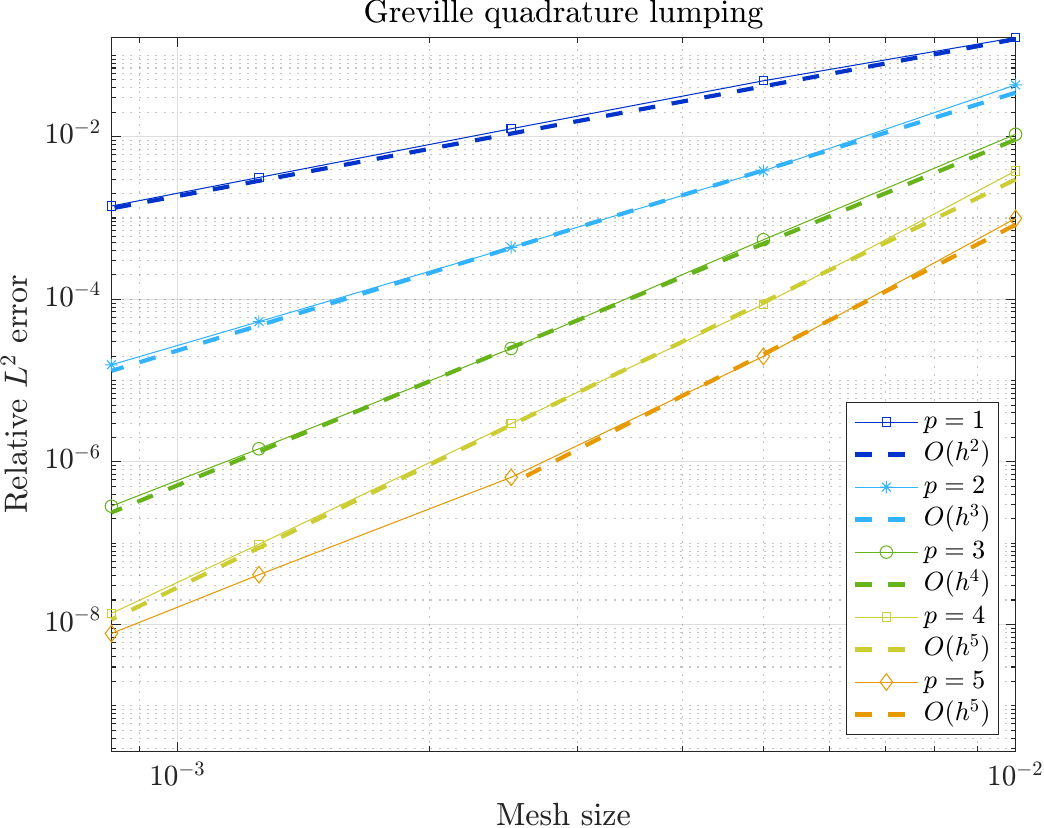}
    \caption{Lumped mass (Greville Lagrange spline basis)}
    \label{fig: 1D_Laplace_dynamics_ML_Greville_mixed_bc}
     \end{subfigure}
     \hfill
    \begin{subfigure}[t]{0.48\textwidth}
    \centering
    \includegraphics[width=\textwidth]{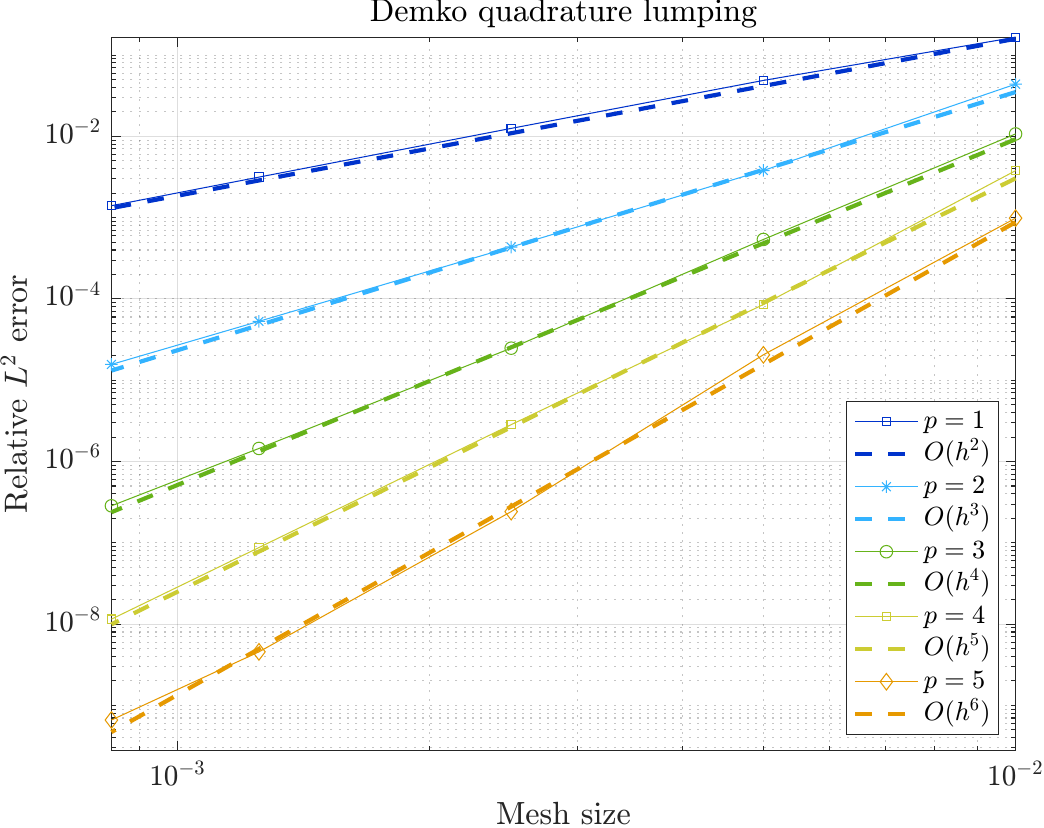}
    \caption{Lumped mass (Demko Lagrange spline basis)}
    \label{fig: 1D_Laplace_dynamics_ML_Demko_mixed_bc}
     \end{subfigure}
     \hfill
    \caption{Relative $L^2$ error for $u(x,t)$ with mixed boundary conditions}
    \label{fig: 1D_Laplace_dynamics_mixed_bc}
\end{figure}
\end{example}

\begin{example}[2D - Dynamics]
\label{ex: 2D_dynamics}
The objective of this example is to demonstrate that the features observed in \Cref{ex: 1D_dynamics} in 1D also hold in higher dimensions, on potentially distorted meshes and while accounting for temporal discretization errors. For this purpose, we choose the upper quarter patch of the $4$-patch square plate with a hole shown in \Cref{fig: 2D_upper_quarter} and remark that this geometry is not a tensor product. Once again, we consider two manufactured solutions $u_i(r,t)=w_i(r)\sin(\omega t)$, where $r=\sqrt{x^2+y^2}$ and
\begin{equation*}
    w_1(r)= (r-1)\cos(\omega r) \quad \text{and} \quad w_2(r)=(r-1)\mathrm{e}^{-\left(\frac{r-r_c}{\sigma}\right)^2}
\end{equation*}
with parameter values $\sigma=0.1$, $r_c=1.5$ and $\omega = 2\pi$. Those functions are shown in \Cref{fig: 2D_Laplace_upper_quarter_dynamics_func}. The functions $u_i(r,t)$ both satisfy homogeneous Dirichlet boundary conditions on the lower curved boundary and we prescribe Neumann boundary conditions on the other three boundaries. The properties of $w_1(r)$ and $w_2(r)$ are rather similar to those in \Cref{ex: 1D_dynamics}. Indeed, while the first derivative of $w_1(r)$ does not vanish near the Dirichlet boundary, $w_2(r)$ has infinitely many (near) vanishing derivatives. 

\begin{figure}[H]
    \centering
    \includegraphics[scale=0.5]{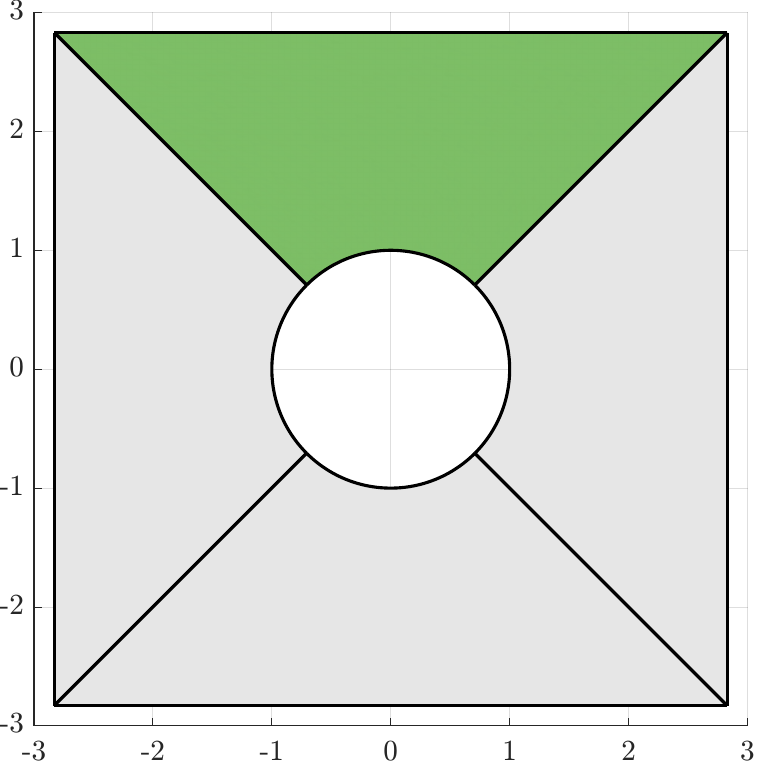}
    \caption{Upper quarter of a $4$-patch square plate with a hole}
    \label{fig: 2D_upper_quarter}
\end{figure}

\begin{figure}[H]
     \centering
     \begin{subfigure}[t]{0.48\textwidth}
    \centering
    \includegraphics[width=\textwidth]{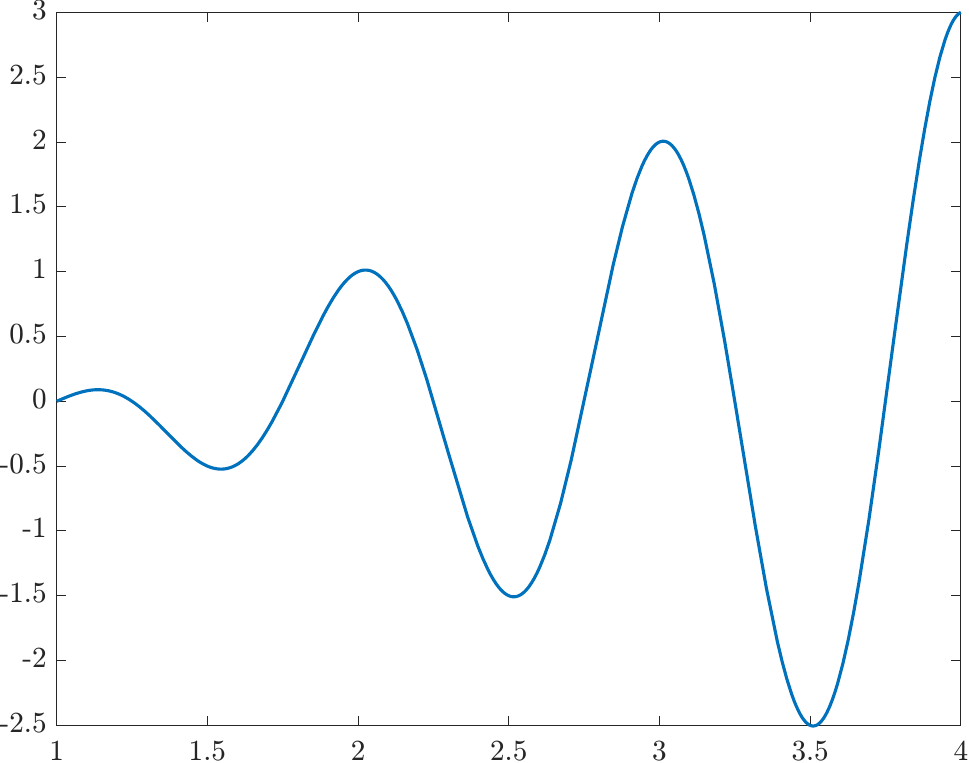}
    \caption{Function $w_1(r)$}
    \label{fig: 2D_Laplace_upper_quarter_dynamics_easy_func}
     \end{subfigure}
     \hfill
     \begin{subfigure}[t]{0.48\textwidth}
    \centering
    \includegraphics[width=\textwidth]{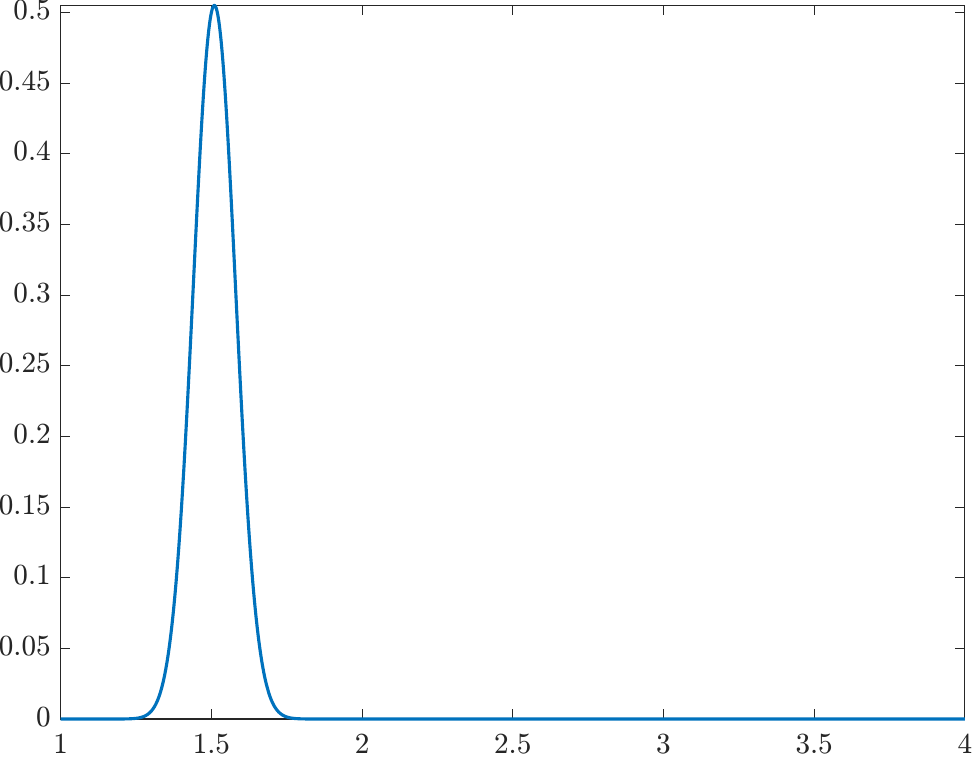}
    \caption{Function $w_2(r)$}
    \label{fig: 2D_Laplace_upper_quarter_dynamics_diff_func}
     \end{subfigure}
     \hfill
    \caption{Spatial part of the manufactured solutions for \Cref{ex: 2D_dynamics}}
    \label{fig: 2D_Laplace_upper_quarter_dynamics_func}
\end{figure}

We now solve the wave equation on the upper quarter of the plate with unit material coefficients. The non-homogeneous Neumann boundary data, the right-hand side and the initial conditions are computed from the manufactured solutions. Spline spaces of degree $1$ to $5$ are then built on increasingly fine meshes. Although the semi-discrete solution also admits a closed form solution for this example, we decided to couple high order spatial discretizations with high order temporal ones to complete the picture. In this specific case, the central difference method is only used for degree $p=1$ while classical explicit Runge Kutta (RK) methods of order $3,4,5,6$ are used for degree $2,3,4,5$ discretizations, respectively, and ensure that the spatial accuracy matches the temporal one. The expression of their critical time step is similar to the central difference method and only differs by a constant in the numerator (see \Cref{se: explicit_rk_methods}), which is then further reduced by a safety factor of $0.85$. However, we must stress that beyond order $4$, the number of stages of explicit RK methods always exceeds the order and these methods are less attractive.

Hereafter, we analyze the convergence of the fully discrete solution $u_h(x,t)$ by computing the relative $L^2$ error at discrete times
\begin{equation*}
    \frac{\|u(t_j)-u_h^j\|_{L^2}}{\|u(t_j)\|_{L^2}},
\end{equation*}
where $t_j = j \Delta t$ and $u_h^j$ denotes the fully discrete solution at step $j$. We study the convergence of the relative error for both manufactured solutions at the final time $T=0.25$, when the discretization error is supposedly the largest. The results for the consistent mass (\Cref{fig: 2D_Laplace_dynamics_consistent}) and the lumped mass approximation, either directly in the B-spline basis (\Cref{fig: 2D_Laplace_dynamics_ML_Bspline}) or in the interpolatory spline basis based on the Greville points (\Cref{fig: 2D_Laplace_dynamics_ML_Greville}) or the Demko points (\Cref{fig: 2D_Laplace_dynamics_ML_Demko}) closely resemble those in \Cref{ex: 1D_dynamics}. Thus, those trends were not limited to 1D and are oblivious to mesh distortion. In particular, for the consistent mass, the error converges at the optimal rate and sometimes even shows signs of super-convergence. For the lumped mass in the B-spline basis, the convergence rate consistently drops to $2$. In fact, increasing the spline degree even deteriorates the constant. In contrast, when lumping the mass matrix in the interpolatory spline basis, the convergence rate is again solution-dependent. It stalls at about $3.5$ for $u_1$ but increases well beyond that for $u_2$. This second function actually required much finer meshes to get passed a pre-asymptotic fast convergence phase. Although the convergence rate for the Greville and Demko points is rather similar, the constants sometimes differ by several orders of magnitude, with the Demko points often yielding greater accuracy. The improved accuracy of high order mass lumping strategies lead to visible differences in the numerical solutions, especially on coarse meshes. The approximations of $u_2$ are exemplarily shown in \Cref{fig: 2D_Laplace_upper_quarter_dynamics_snapshots_diff_func} at the final time. 

Overall, the results of the 1D and 2D experiments perfectly align and certainly warrant further investigations. Similar convergence rates were experienced on different examples and the functions $u_1$ and $u_2$ in our experiments reflect the range of cases encountered.

\begin{figure}[H]
     \centering
     \begin{subfigure}[t]{0.48\textwidth}
    \centering
    \includegraphics[width=\textwidth]{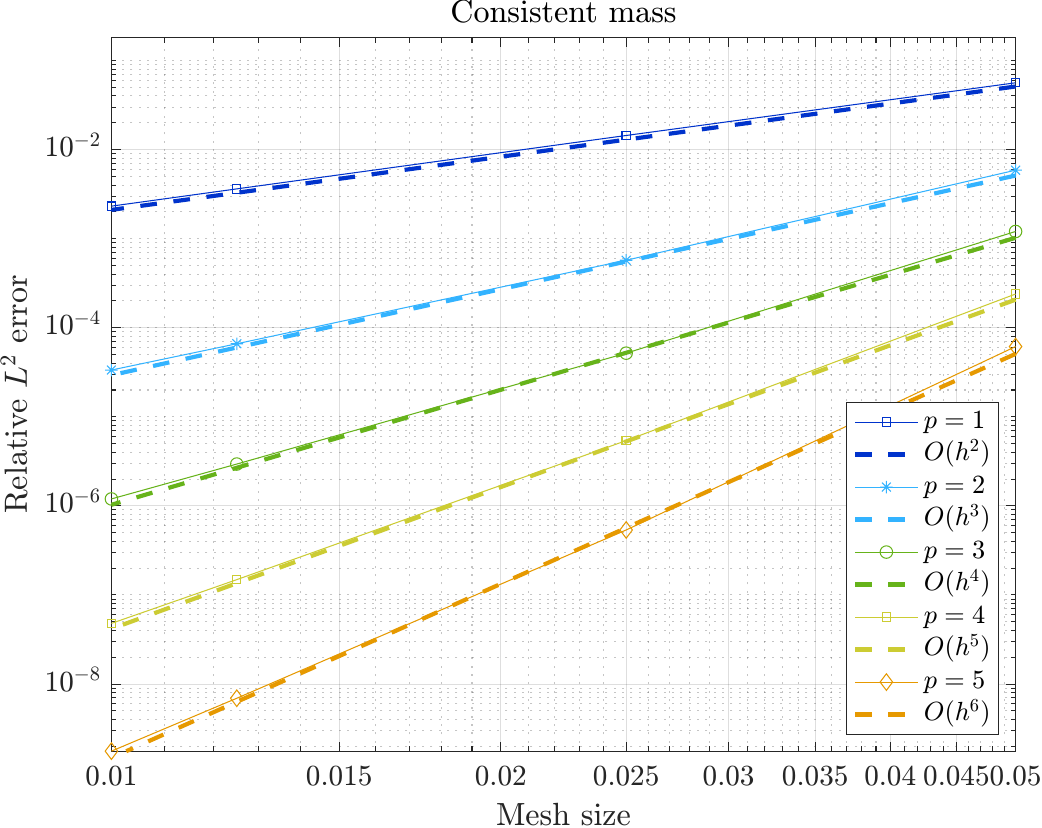}
    \caption{Error for $u_1(r,t)$}
    \label{fig: 2D_Laplace_upper_quarter_dynamics_consistent_easy_func}
     \end{subfigure}
     \hfill
     \begin{subfigure}[t]{0.48\textwidth}
    \centering
    \includegraphics[width=\textwidth]{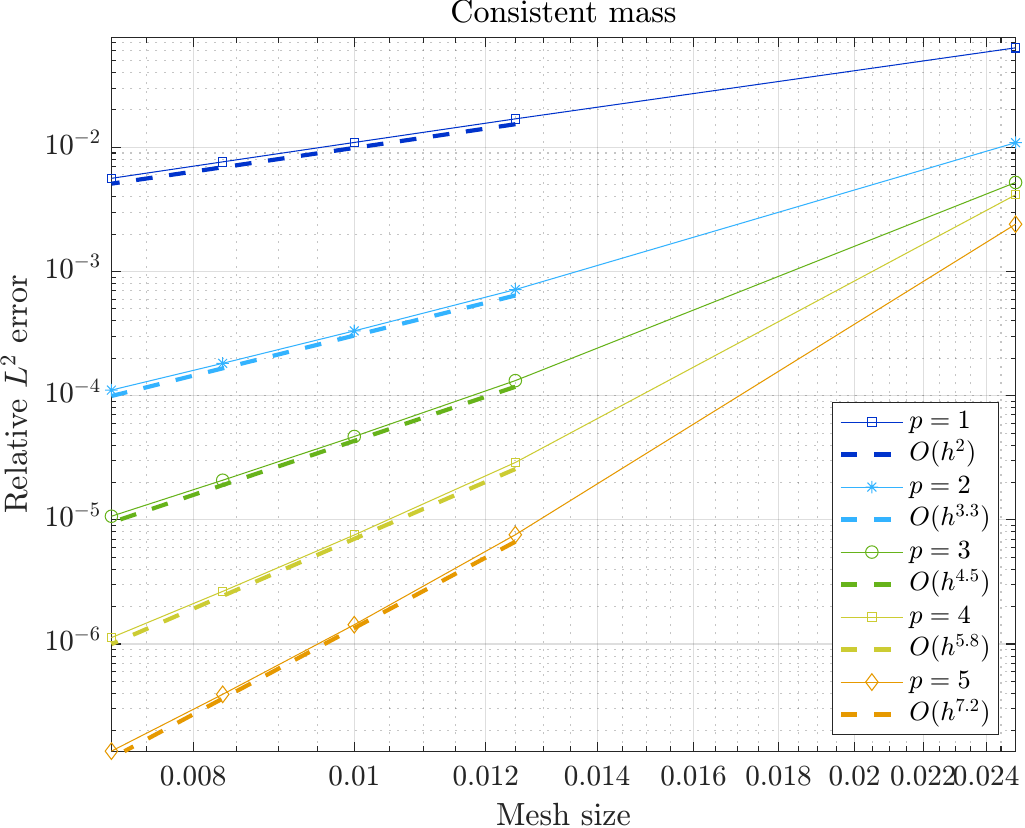}
    \caption{Error for $u_2(r,t)$}
    \label{fig: 2D_Laplace_upper_quarter_dynamics_consistent_diff_func}
     \end{subfigure}
     \hfill
    \caption{Consistent mass}
    \label{fig: 2D_Laplace_dynamics_consistent}
\end{figure}

\begin{figure}[H]
     \centering
     \begin{subfigure}[t]{0.48\textwidth}
    \centering
    \includegraphics[width=\textwidth]{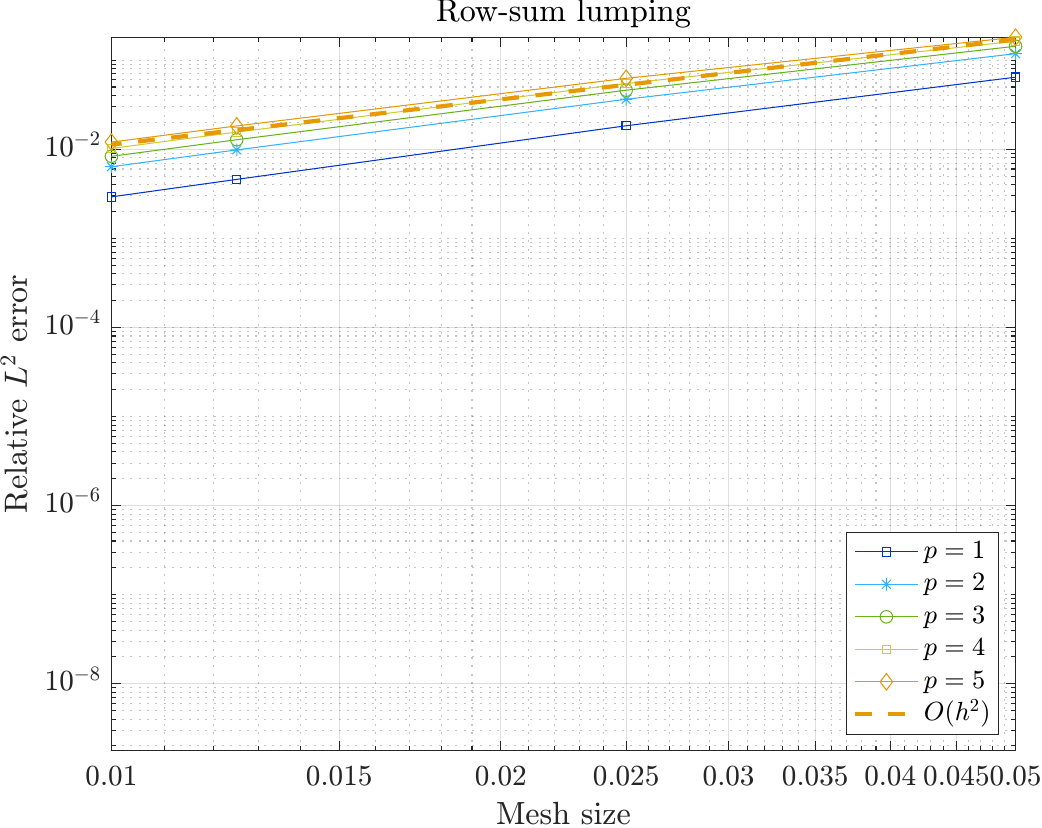}
    \caption{Error for $u_1(r,t)$}
    \label{fig: 2D_Laplace_upper_quarter_dynamics_ML_Bspline_easy_func}
     \end{subfigure}
     \hfill
     \begin{subfigure}[t]{0.48\textwidth}
    \centering
    \includegraphics[width=\textwidth]{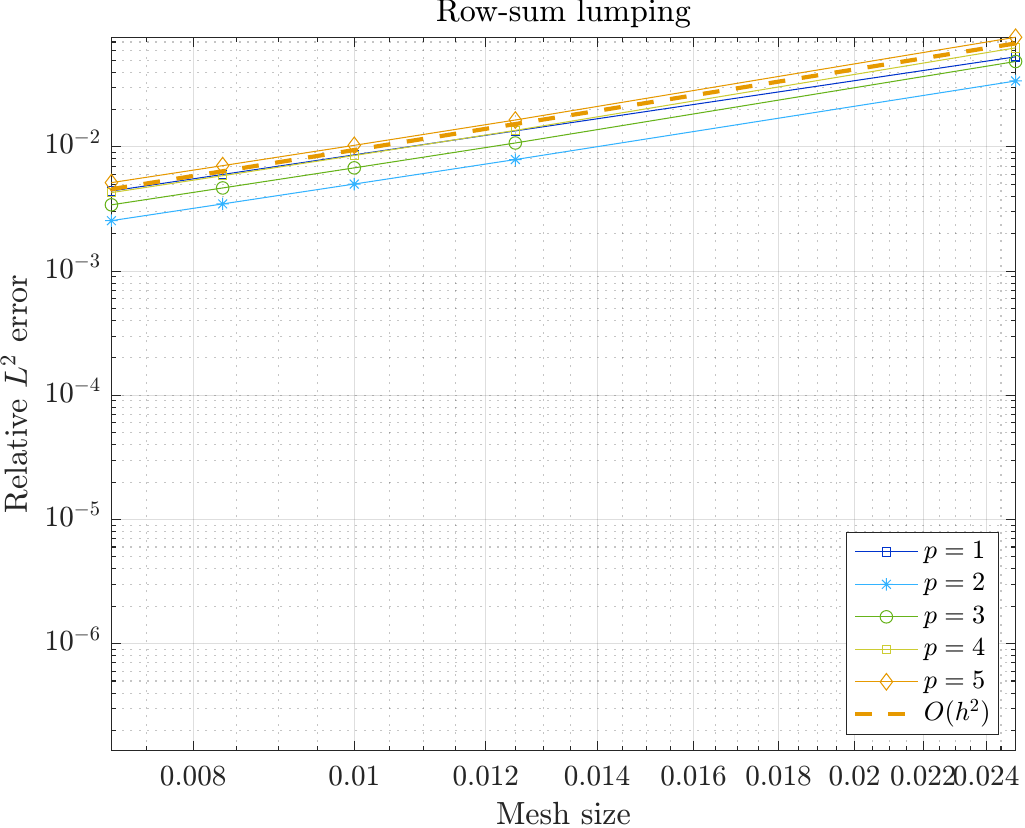}
    \caption{Error for $u_2(r,t)$}
    \label{fig: 2D_Laplace_upper_quarter_dynamics_ML_Bspline_diff_func}
     \end{subfigure}
     \hfill
    \caption{Lumped mass (B-spline basis)}
    \label{fig: 2D_Laplace_dynamics_ML_Bspline}
\end{figure}

\begin{figure}[H]
     \centering
     \begin{subfigure}[t]{0.48\textwidth}
    \centering
    \includegraphics[width=\textwidth]{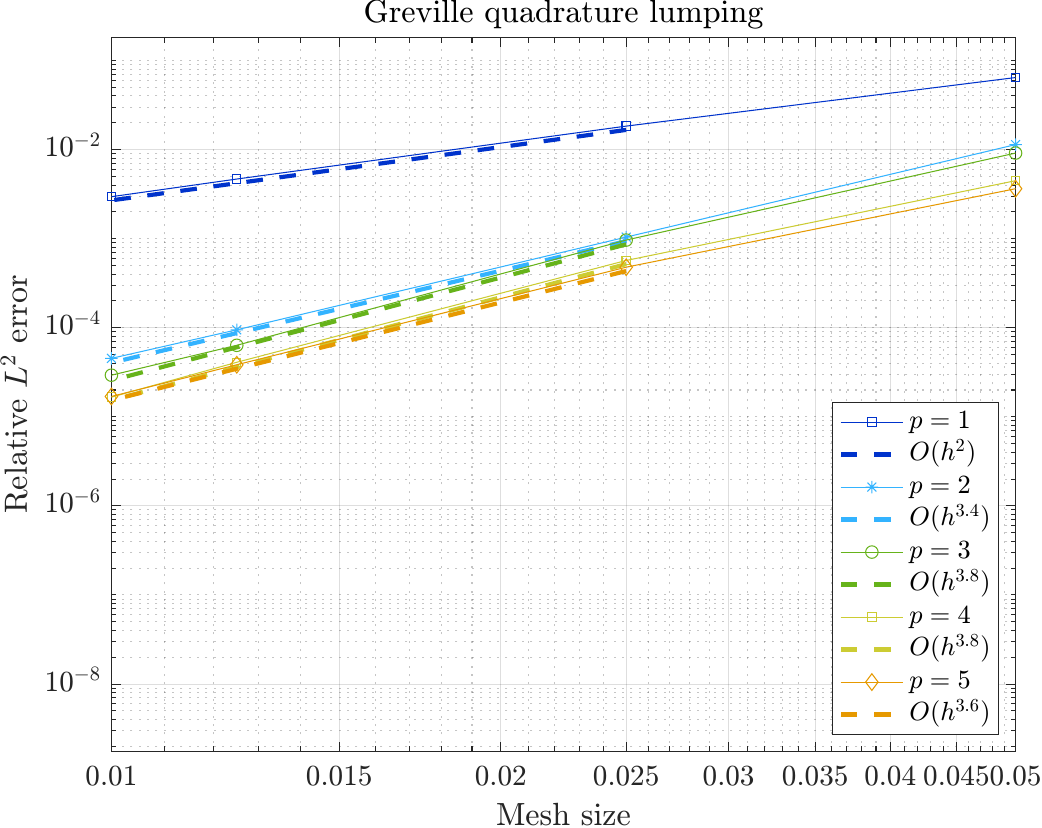}
    \caption{Error for $u_1(r,t)$}
    \label{fig: 2D_Laplace_upper_quarter_dynamics_ML_Greville_easy_func}
     \end{subfigure}
     \hfill
     \begin{subfigure}[t]{0.48\textwidth}
    \centering
    \includegraphics[width=\textwidth]{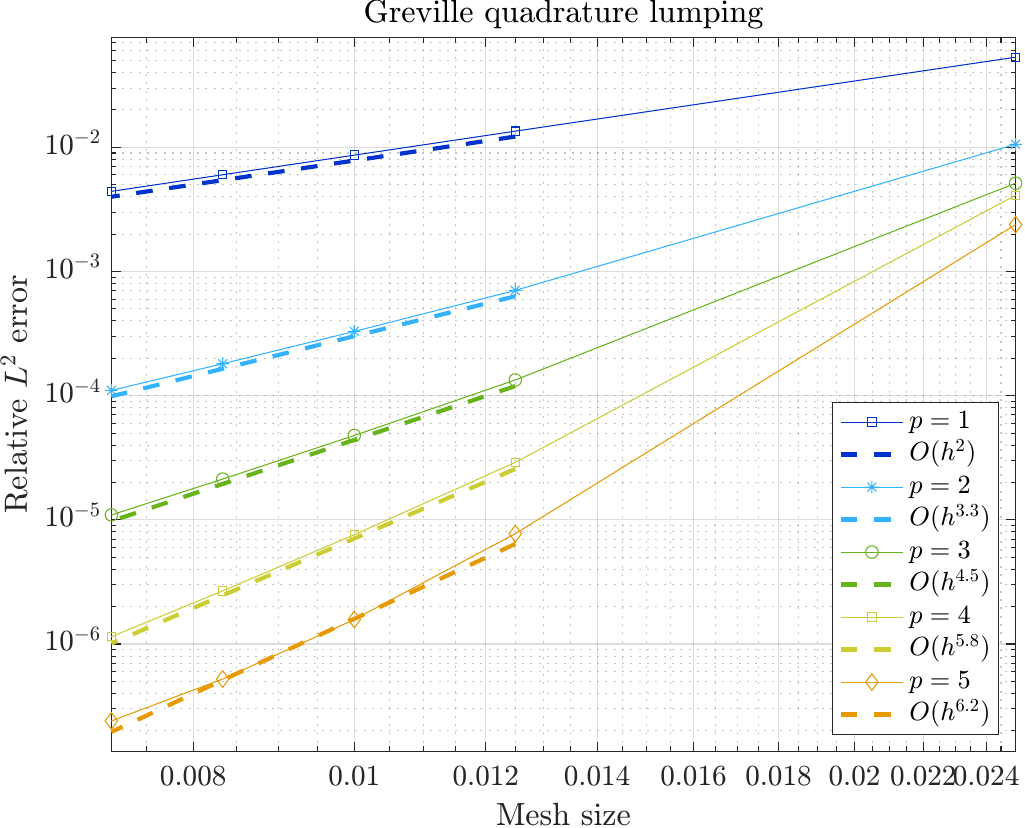}
    \caption{Error for $u_2(r,t)$}
    \label{fig: 2D_Laplace_upper_quarter_dynamics_ML_Greville_diff_func}
     \end{subfigure}
     \hfill
    \caption{Lumped mass (Greville Lagrange spline basis)}
    \label{fig: 2D_Laplace_dynamics_ML_Greville}
\end{figure}

\begin{figure}[H]
     \centering
     \begin{subfigure}[t]{0.48\textwidth}
    \centering
    \includegraphics[width=\textwidth]{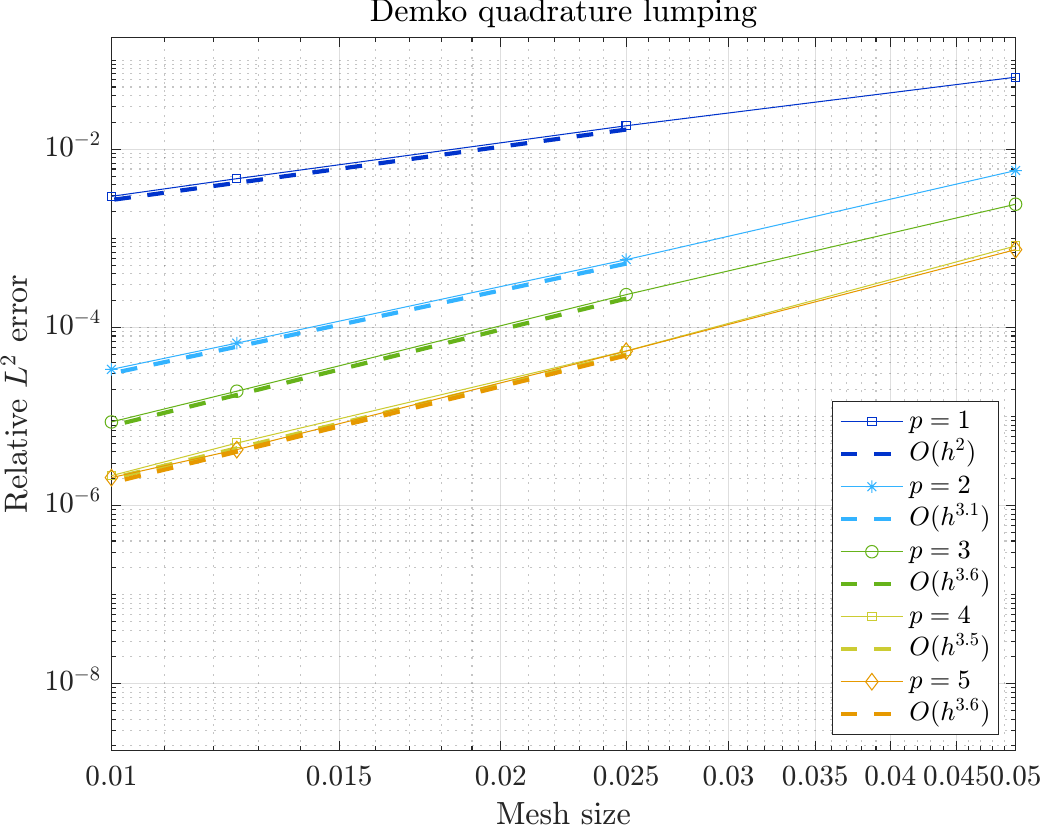}
    \caption{Error for $u_1(r,t)$}
    \label{fig: 2D_Laplace_upper_quarter_dynamics_ML_Demko_easy_func}
     \end{subfigure}
     \hfill
     \begin{subfigure}[t]{0.48\textwidth}
    \centering
    \includegraphics[width=\textwidth]{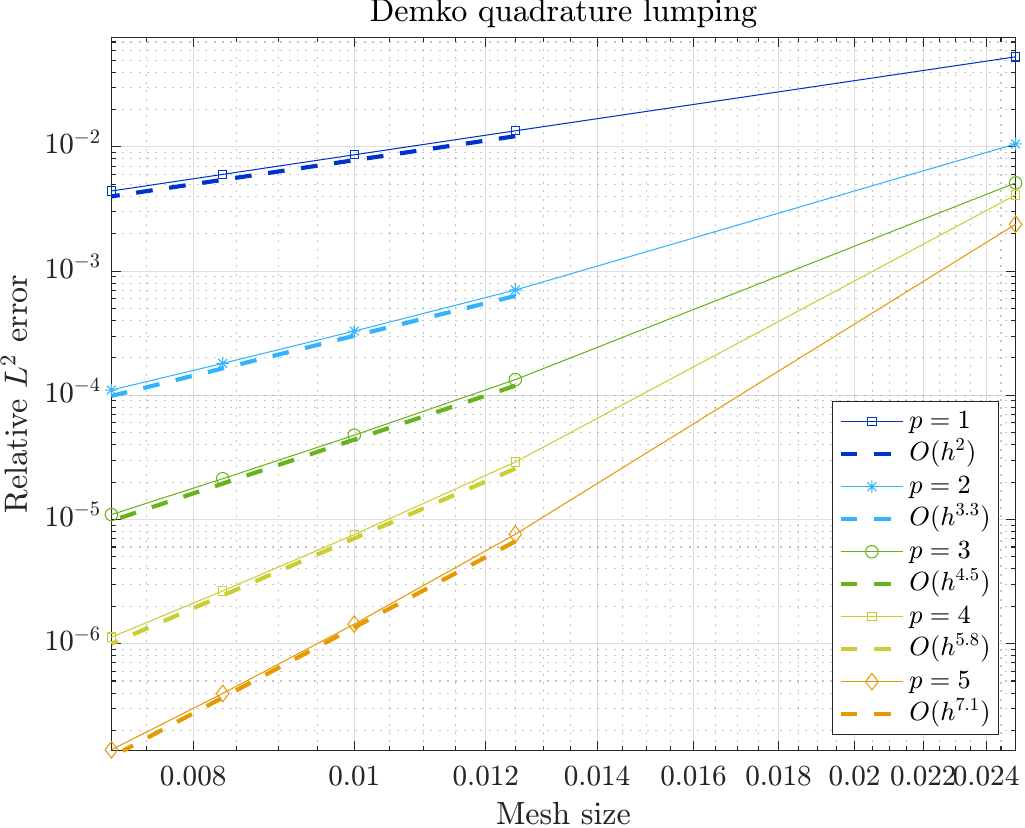}
    \caption{Error for $u_2(r,t)$}
    \label{fig: 2D_Laplace_upper_quarter_dynamics_ML_Demko_diff_func}
     \end{subfigure}
     \hfill
    \caption{Lumped mass (Demko Lagrange spline basis)}
    \label{fig: 2D_Laplace_dynamics_ML_Demko}
\end{figure}    

\begin{figure}[H]
    \centering
    \includegraphics[scale=0.45]{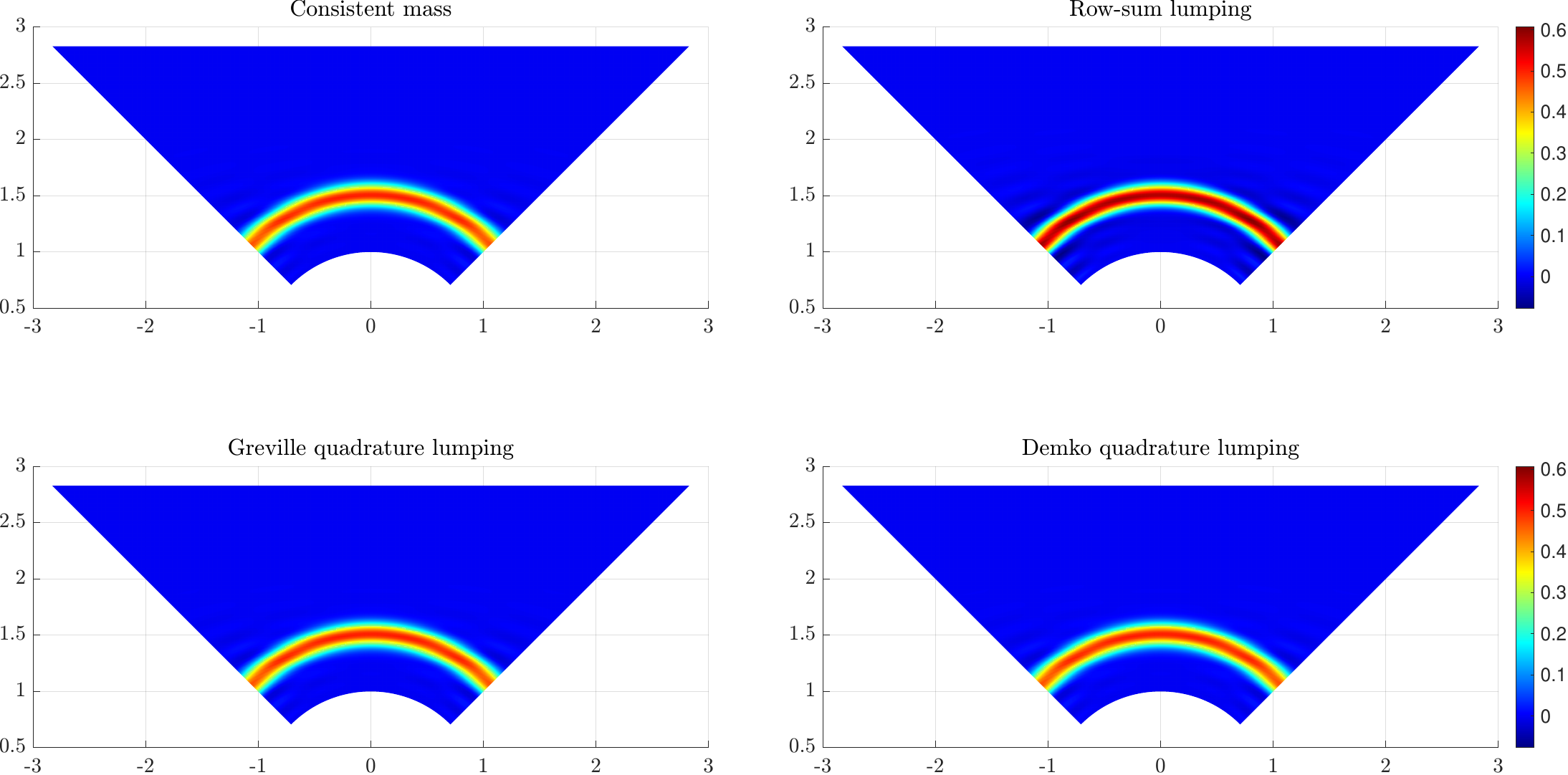}
    \caption{Numerical approximations of $u_2$ at the final time for a quintic spline discretization with $25$ subdivisions in each direction.}
    \label{fig: 2D_Laplace_upper_quarter_dynamics_snapshots_diff_func}
\end{figure}
\end{example}

\begin{example}[Distorted square]
\label{ex: mesh_distortion}
We end this series of experiments with an example investigating the influence of mesh distortion. For this purpose, we consider the same example from \cite{hiemstra2025higher}, where the gridlines of the unit square are distorted by a non-linear map, as shown in \Cref{fig: 2D_Laplace_distorted_square}. A parameter $\alpha$ tunes the severity of the distortion and we refer to the original article for the technical description. The manufactured solution in this case is given by
\begin{equation*}
    u(x,y,t) = \sin(4\pi x)\sin(4 \pi y)\cos(4 \pi t)
\end{equation*}
and we study the convergence of the relative error for the fully discrete solution at the final time $T=1$. Spline spaces of degree $1$ to $5$ are built on meshes with $16,32,64$ and $128$ subdivisions in each direction. Similarly to \Cref{ex: 2D_dynamics}, high order spatial discretizations are coupled to high order temporal ones, described in \Cref{se: explicit_rk_methods}. The results of the convergence study are shown in \Cref{fig: 2D_Laplace_square_dynamics_consistent} for the consistent mass and in \Cref{fig: 2D_Laplace_square_dynamics_ML_Bspline,fig: 2D_Laplace_square_dynamics_ML_Greville,fig: 2D_Laplace_square_dynamics_ML_Demko} for lumped mass approximations. The convergence rates for the consistent mass are completely insensitive to the distortion. Moreover, the accuracy of the row-sum for the B-spline basis, which is already quite poor for regular meshes, thankfully does not get much worse. For what concerns our high order lumping strategies, we only observe a mild dependency on the distortion. Note that the sub-optimal convergence rate already experienced for regular meshes is completely expected from the previous examples in 1D and 2D: the manufactured solution vanishes on the boundary but not its first partial derivatives and therefore closely resembles the function $u_1$ chosen in \Cref{ex: 1D_dynamics,ex: 2D_dynamics}. Overall, the results of this experiment align with our previous findings.

\begin{figure}[H]
     \centering
     \begin{subfigure}[t]{0.48\textwidth}
    \centering
    \includegraphics[width=0.8\textwidth]{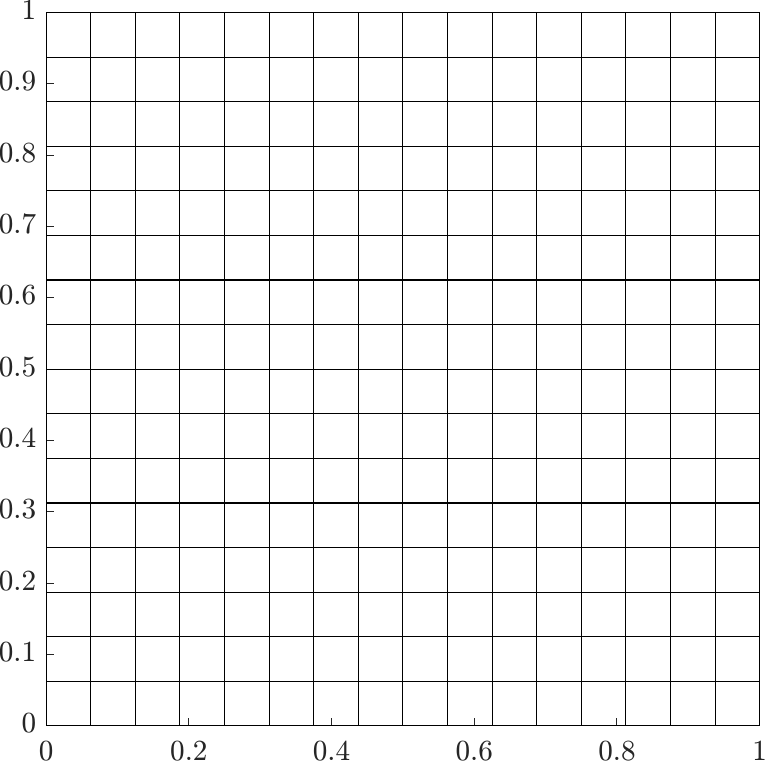}
    \caption{$\alpha=0$}
    \label{fig: 2D_Laplace_unit_square_mesh}
     \end{subfigure}
     \hfill
     \begin{subfigure}[t]{0.48\textwidth}
    \centering
    \includegraphics[width=0.8\textwidth]{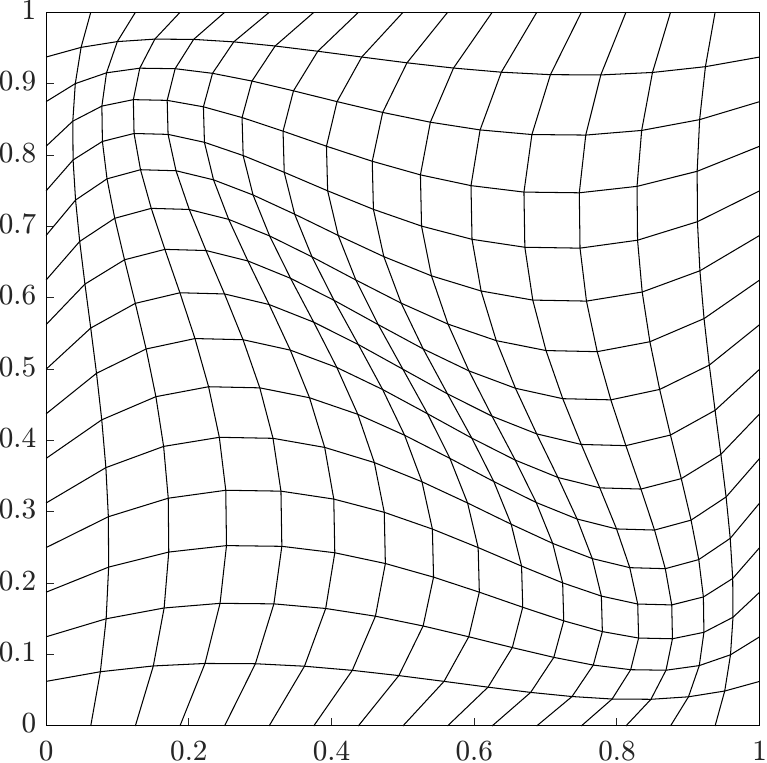}
    \caption{$\alpha=0.5$}
    \label{fig: 2D_Laplace_distorted_square_mesh}
     \end{subfigure}
     \hfill
    \caption{Unit square with (distorted) gridlines}
    \label{fig: 2D_Laplace_distorted_square}
\end{figure}

\begin{figure}[H]
     \centering
     \begin{subfigure}[t]{0.48\textwidth}
    \centering
    \includegraphics[width=\textwidth]{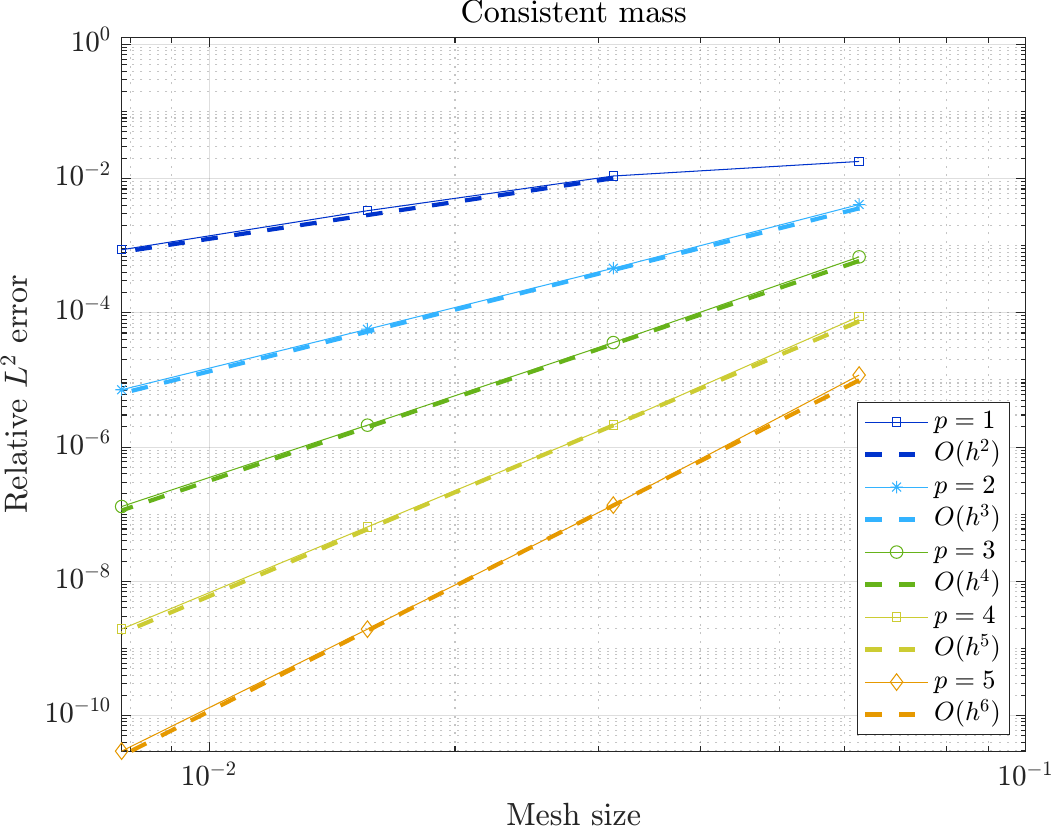}
    \caption{$\alpha=0$}
    \label{fig: 2D_Laplace_unit_square_dynamics_consistent}
     \end{subfigure}
     \hfill
     \begin{subfigure}[t]{0.48\textwidth}
    \centering
    \includegraphics[width=\textwidth]{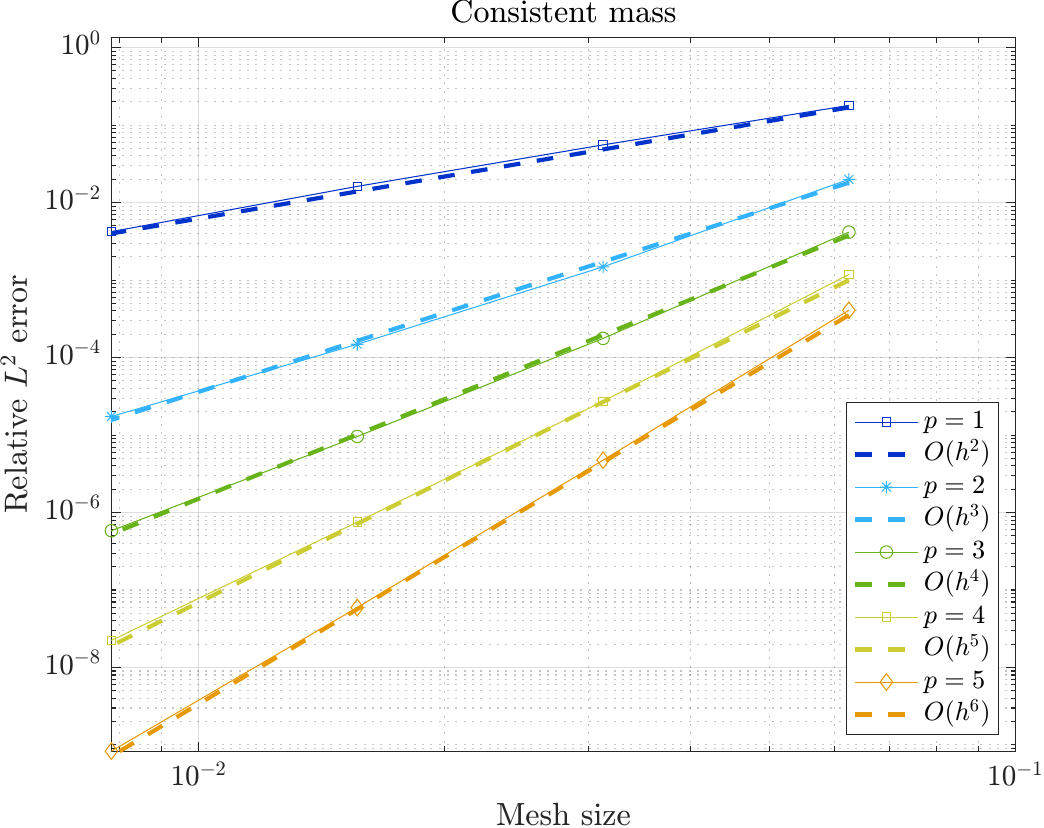}
    \caption{$\alpha=0.5$}
    \label{fig: 2D_Laplace_distorted_square_dynamics_consistent}
     \end{subfigure}
     \hfill
    \caption{Consistent mass}
    \label{fig: 2D_Laplace_square_dynamics_consistent}
\end{figure}

\begin{figure}[H]
     \centering
     \begin{subfigure}[t]{0.48\textwidth}
    \centering
    \includegraphics[width=\textwidth]{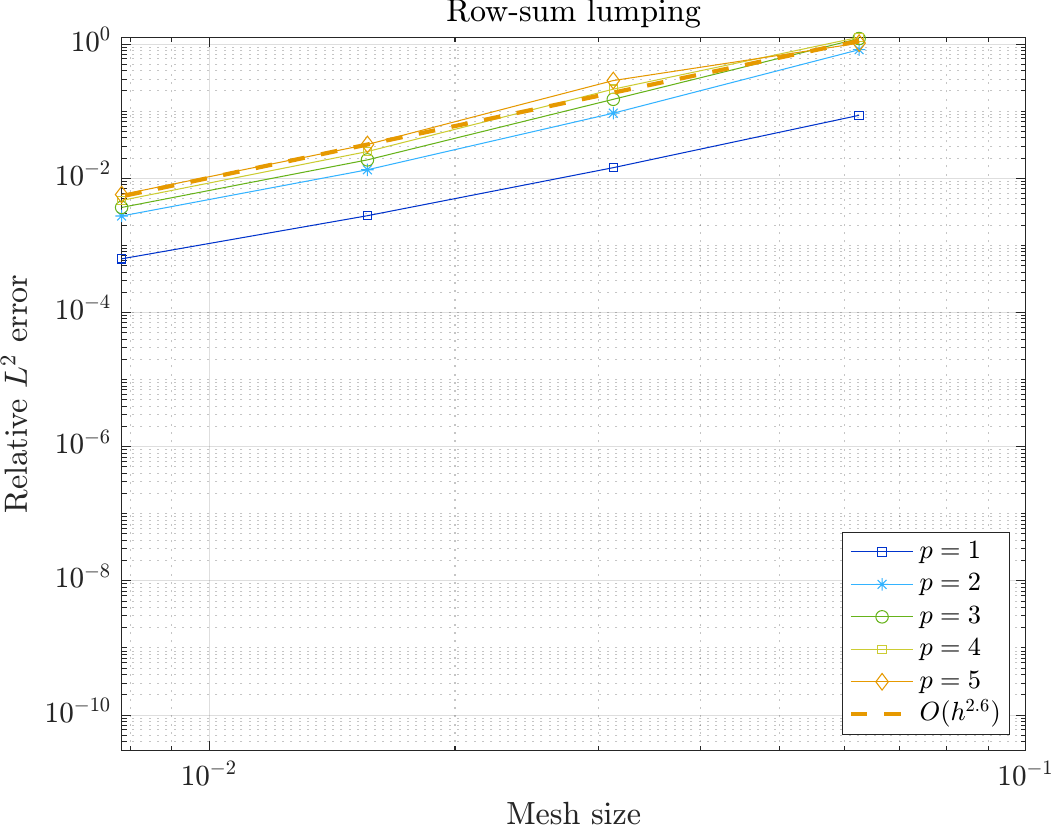}
    \caption{$\alpha=0$}
    \label{fig: 2D_Laplace_unit_square_dynamics_ML_Bspline}
     \end{subfigure}
     \hfill
     \begin{subfigure}[t]{0.48\textwidth}
    \centering
    \includegraphics[width=\textwidth]{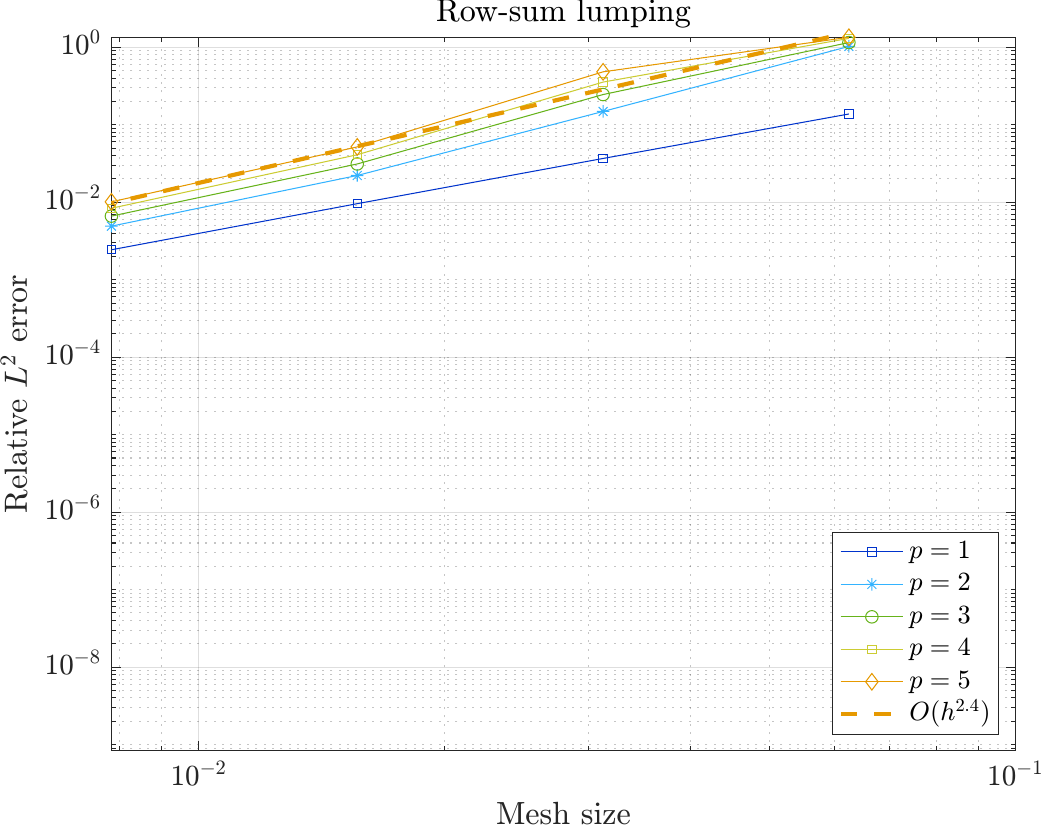}
    \caption{$\alpha=0.5$}
    \label{fig: 2D_Laplace_distorted_square_dynamics_ML_Bspline}
     \end{subfigure}
     \hfill
    \caption{Lumped mass (B-spline basis)}
    \label{fig: 2D_Laplace_square_dynamics_ML_Bspline}
\end{figure}

\begin{figure}[H]
     \centering
     \begin{subfigure}[t]{0.48\textwidth}
    \centering
    \includegraphics[width=\textwidth]{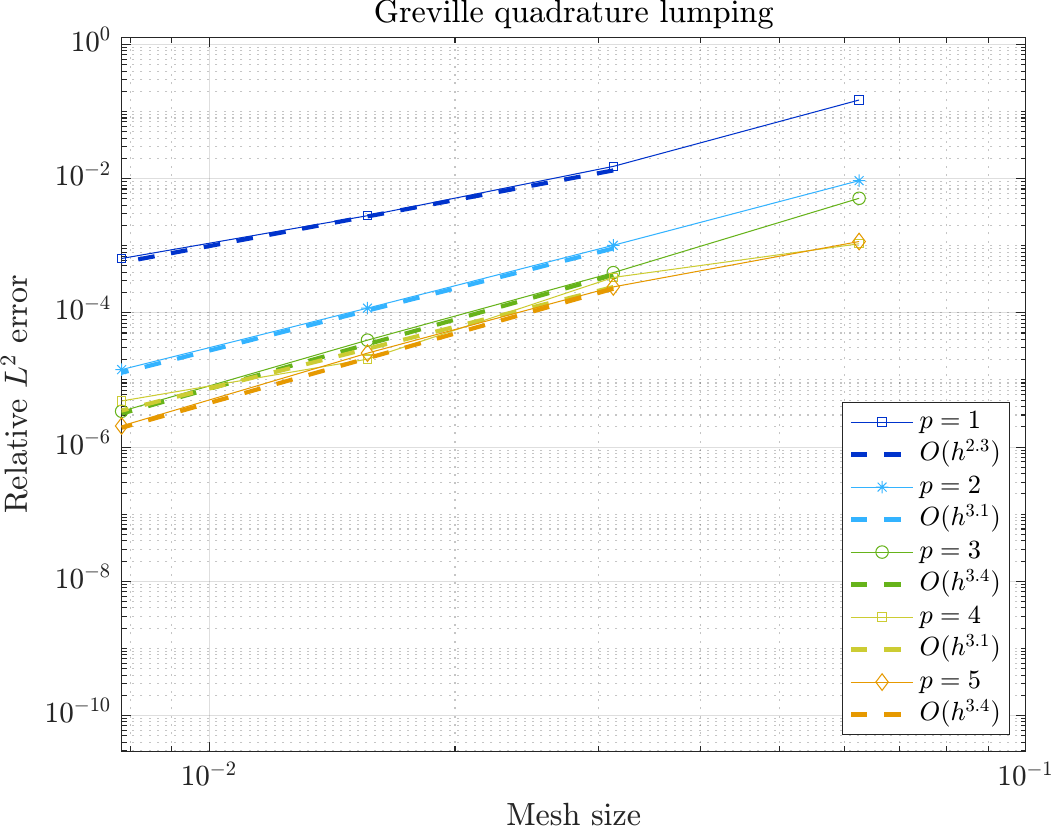}
    \caption{$\alpha=0$}
    \label{fig: 2D_Laplace_unit_square_dynamics_ML_Greville}
     \end{subfigure}
     \hfill
     \begin{subfigure}[t]{0.48\textwidth}
    \centering
    \includegraphics[width=\textwidth]{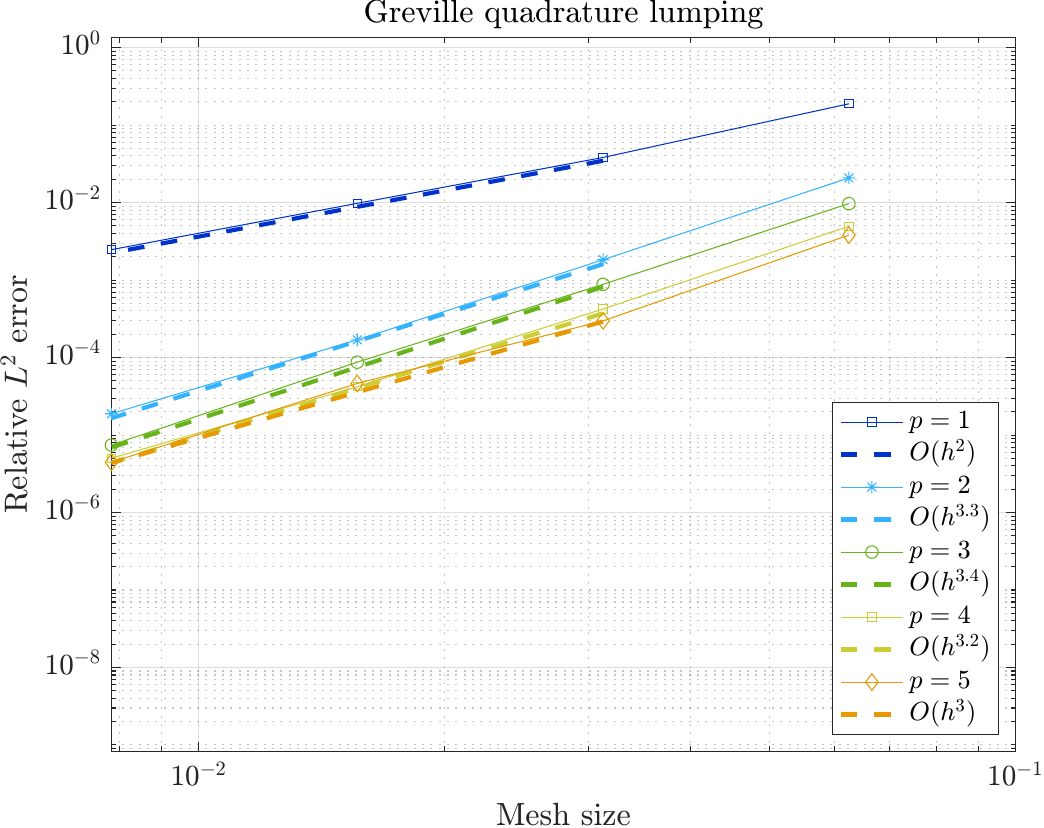}
    \caption{$\alpha=0.5$}
    \label{fig: 2D_Laplace_distorted_square_dynamics_ML_Greville}
     \end{subfigure}
     \hfill
    \caption{Lumped mass (Greville Lagrange spline basis)}
    \label{fig: 2D_Laplace_square_dynamics_ML_Greville}
\end{figure}

\begin{figure}[H]
     \centering
     \begin{subfigure}[t]{0.48\textwidth}
    \centering
    \includegraphics[width=\textwidth]{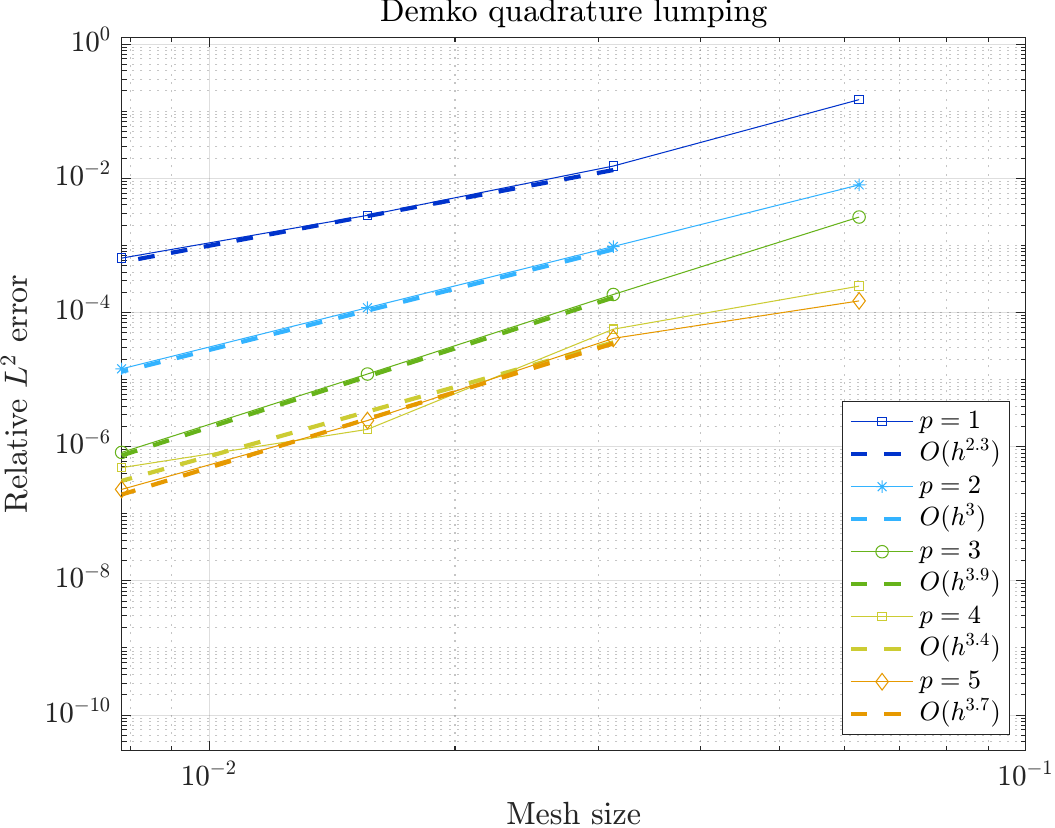}
    \caption{$\alpha=0$}
    \label{fig: 2D_Laplace_unit_square_dynamics_ML_Demko}
     \end{subfigure}
     \hfill
     \begin{subfigure}[t]{0.48\textwidth}
    \centering
    \includegraphics[width=\textwidth]{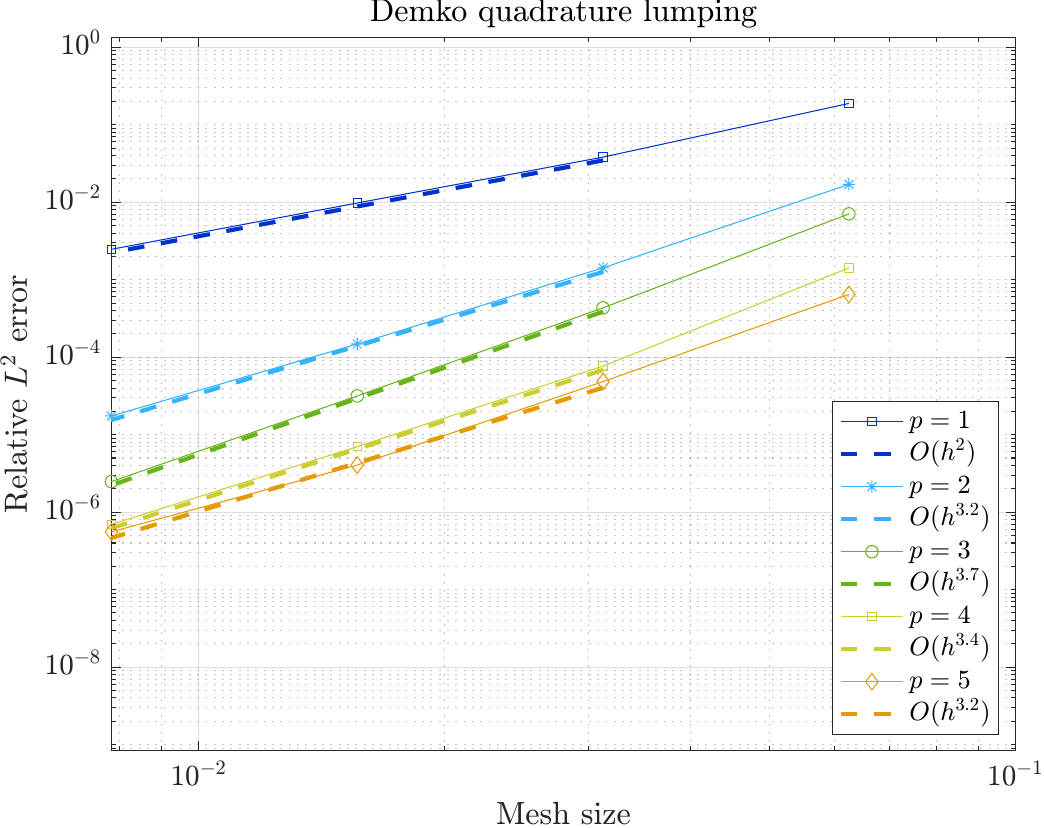}
    \caption{$\alpha=0.5$}
    \label{fig: 2D_Laplace_distorted_square_dynamics_ML_Demko}
     \end{subfigure}
     \hfill
    \caption{Lumped mass (Demko Lagrange spline basis)}
    \label{fig: 2D_Laplace_square_dynamics_ML_Demko}
\end{figure} 
\end{example}

\section{Conclusion}
\label{se: conclusion}
In this article, we have critically assessed the possibility of employing interpolatory spline bases for the purpose of mass lumping in isogeometric discretizations. Although non-interpolatory bases such as the B-spline basis are traditionally favored in IGA, restoring interpolation helps recover some of the critical properties that forged the success of SEM. However, extending those properties to spline functions has proven difficult and a full theory is still lacking. Most notably,
\begin{itemize}[noitemsep]
    \item The convergence rate observed for the smallest eigenfrequencies and the solution of the initial boundary value problem is unclear and does not necessarily correlate with the rate of the underlying spline quadrature rule. Unfortunately, the arguments for proving the convergence of SEM, based on the Strang lemma combined with the elementwise Bramble-Hilbert lemma, do not immediately extend to spline functions. Indeed, classical FEA revolves around elements whereas spline spaces are built from knot vectors and follow a slightly different paradigm. Substituting elements with knot spans is not always that simple. The same problem arises for the analysis of spline quadrature rules.
    \item Lumping the mass matrix for interpolatory bases may produce negative eigenvalues and unstable, potentially diverging, solutions. The Demko points appear to be an exception (see Conjecture \ref{conj: quadrature_condition}) but we were unable to prove it. Nevertheless, the theory for classical polynomial quadrature suggests that stability and positivity of the weights go hand in hand since a counter-example for one is typically also a counter-example for the other. Moreover, sufficient accuracy of the quadrature rule allows proving positivity of its weights. However, for the Demko points and spline quadrature, it remains an open problem.
    \item Similarly to classical FEM on uniformly spaced nodes, lumping the mass matrix for interpolatory spline bases may deteriorate the critical time step. However, just like classical FEM, this may be attributed to a poor choice of interpolation points. Although we have found a sufficient condition on the quadrature rule ensuring both positive weights and an increase of the critical time step, we are currently not aware of any explicit rule satisfying this condition. However, highly accurate quadrature rules might naturally resolve the issue, if one looks at how the proof of \Cref{lem: spectral_fem} could be extended to splines. Whether such rules are cheaply computable is yet another issue.
\end{itemize}
In summary, despite several appealing properties, interpolatory spline bases are still nowhere near mimicking their polynomial counterpart in SEM. While this article certainly lays a steady basis, a strong theory that parallels the nodal quadrature method is still sought for IGA. Nevertheless, higher convergence rates are encouraging and a better choice of interpolation points might further improve the accuracy but this is left as future work.

With interpolatory spline bases, IGA moves closer toward FEM. Extending the method to multi-patch geometries is also foreseeable since they are treated in the same way as elements for classical FEM. However, just like SEM, trimmed geometries might again become problematic, in particular regarding the choice of interpolation points. Beyond IGA, accurate mass lumping strategies on trimmed or immersed geometries are one of the most pressing challenges in computational mechanics \cite{radtke2024analysis,eisentrager2024eigenvalue,burchner2025efficiency,voet2025stabilization,guarino2025stabilization}.

\section*{Acknowledgement}
The second author was supported by the European Research Council under grant agreement 101141807.

\begin{appendices}
\section{Proof of \Cref{lem: spectral_fem}}
\label{se: proof_CFL_SEM}
The proof of \Cref{lem: spectral_fem} is mainly based on the observation of Teukolsky \cite{teukolsky2015short} that the element consistent and lumped mass matrices for SEM differ by a rank-$1$ update. \Cref{lem: spectral_fem} is recalled below and is substantiated with a detailed proof.

\begin{lemma}
For elementwise constant density and affine tensor product spectral elements, $\widehat{M} \succeq M$ and
\begin{equation*}
    \lambda_k(K,\widehat{M}) \leq \lambda_k(K,M).
\end{equation*}
\end{lemma}
\begin{proof}
The first part of the proof mostly follows the arguments in \cite{teukolsky2015short}, although from a different perspective. We first prove the result for the 1D case and later extend it to multiple dimensions on tensor product elements. Given a positive weight $\eta_e>0$, we define the continuous inner product
\begin{equation}
\label{eq: inner_product}
    b_e(u,v) = \int_{-1}^1 \eta_e u(x)v(x) \, dx.
\end{equation}
Note that $b_e$ indeed takes this form when the integral is mapped to the reference element with $\eta_e=\rho |\det(J_e)|$. For the inner product in \eqref{eq: inner_product}, there exists a set of orthogonal polynomials $\dutchcal{P} = \{p_k\}_{k=0}^q$ such that $p_k \in \mathbb{P}_k$ and $b_e(p_k,p_j) = g_k \delta_{kj}$ for some constant $g_k >0$. Since $\eta_e$ is constant, those polynomials are simply rescaled Legendre polynomials and without loss of generality, we set $\eta_e=1$ such that the constants $g_k$ reduce to (see e.g. \cite{quarteroni2010numerical})
\begin{equation*}
    g_k = \frac{2}{2k+1}.
\end{equation*}
Associated to the Legrendre polynomials is the Gauss-Lobatto quadrature rule, denoted $\{x_k,w_k\}_{k=0}^q$, and defining the discrete inner product
\begin{equation*}
    \widehat{b}_e(u,v) = \sum_{k=0}^q w_k u(x_k)v(x_k).
\end{equation*}
To simplify the notation, we omit the hats over the quadrature nodes and weights (although they are still defined over the reference element). Since the $(q+1)$ point Gauss-Lobatto rule has degree of exactness $2q-1$, we deduce that $\widehat{b}_e(u,v)$ is exact if $uv \in \mathbb{P}_{2q-1}$. Consequently, for the Gauss-Lobatto rule,
\begin{equation*}
    \gamma_k := \widehat{b}_e(p_k, p_k) = 
    \begin{cases}
        g_k & k=0,\dots,q-1, \\
        \frac{2}{q} & k=q.
    \end{cases}
\end{equation*}
Nevertheless, since $p_q p_j \in \mathbb{P}_{2q-1}$ for $j \neq q$, it is still integrated exactly and $\{p_k\}_{k=0}^q$ remains orthogonal in the discrete inner product such that
\begin{equation}
\label{eq: orthogonality}
    \widehat{b}_e(p_k, p_k) = \gamma_k \delta_{kj}.
\end{equation}
Now we define the Lagrange interpolating polynomials at the Gauss-Lobatto points
\begin{equation*}
    \varphi_j(x) = \prod_{\substack{i=0 \\ i \neq j}}^q \frac{(x-x_i)}{(x_j-x_i)}.
\end{equation*}
The sets $\dutchcal{P} = \{p_k\}_{k=0}^q$ and $\Phi=\{\varphi_k\}_{k=0}^q$ are merely different bases for the same space $\mathbb{P}_q$. For relating the mass matrices in the Legendre and Lagrange bases, we need to find the basis transformation. For a polynomial $u \in \mathbb{P}_q$,
\begin{equation*}
    u(x) = \sum_{j=0}^q u_j \varphi_j(x) = \sum_{j=0}^q z_k p_k(x)
\end{equation*}
are the expansions of $u$ in the Lagrange and Legendre bases, respectively. Since the Lagrange basis is interpolatory, we immediately deduce that
\begin{equation}
\label{eq: u_to_b_GL}
    u_i = u(x_i) = \sum_{k=0}^q z_k p_k(x_i).
\end{equation}
Moreover, the orthogonality of the discrete inner product \eqref{eq: orthogonality} implies that
\begin{align*}
    \widehat{b}_e(u, p_k) = z_k \gamma_k &= \sum_{j=0}^q u_j \widehat{b}_e(\varphi_j, p_k), \\
    &= \sum_{j=0}^q \sum_{l=0}^q u_j w_l \varphi_j(x_l)p_k(x_l), \\
    &= \sum_{j=0}^q u_j w_j p_k(x_j),
\end{align*}
from which we deduce that
\begin{equation}
\label{eq: b_to_u_GL}
    z_k = \frac{1}{\gamma_k} \widehat{b}_e(u, p_k) = \frac{1}{\gamma_k} \sum_{j=0}^q u_j w_j p_k(x_j).
\end{equation}
Thus, for the expansion of the Lagrange basis functions in the Legendre basis $\varphi_j(x) = \sum_{k=0}^q a_{kj}p_k(x)$, we deduce from \eqref{eq: b_to_u_GL} that
\begin{equation*}
    a_{kj} = \frac{1}{\gamma_k} \sum_{l=0}^q \delta_{jl} w_l p_k(x_l) = \frac{w_j}{\gamma_k}p_k(x_j)
\end{equation*}
and consequently,
\begin{equation}
\label{eq: Lagrange_to_Legendre}
    \varphi_j(x) = \sum_{k=0}^q \frac{w_j}{\gamma_k} p_k(x_j)p_k(x).
\end{equation}
Now, for the consistent mass matrix, we obtain
\begin{align*}
    (M_e)_{ij} &= b_e(\varphi_i, \varphi_j) \\
    &=\sum_{k=0}^q \sum_{l=0}^q w_i w_j \frac{1}{\gamma_k \gamma_l} p_k(x_i)p_l(x_j) b_e( p_k, p_l) \\
    &=\sum_{k=0}^q \sum_{l=0}^q w_i w_j \frac{1}{\gamma_k \gamma_l} p_k(x_i)p_l(x_j) \delta_{kl} g_k \\
    &=\sum_{k=0}^q w_i w_j \frac{g_k}{\gamma_k^2} p_k(x_i)p_k(x_j) \\
    &= \sum_{k=0}^q w_i w_j \frac{1}{\gamma_k} p_k(x_i)p_k(x_j) + \left(\frac{g_q}{\gamma_q^2} - \frac{1}{\gamma_q}\right) w_iw_j p_q(x_i)p_q(x_j) \\
    &= w_i \varphi_j(x_i) + \left(\frac{g_q}{\gamma_q^2} - \frac{1}{\gamma_q}\right) w_iw_j p_q(x_i)p_q(x_j)
\end{align*}
where the last equation follows from \eqref{eq: Lagrange_to_Legendre}. Noticing that $(\widehat{M}_e)_{ij} = w_i \varphi_j(x_i) = w_i \delta_{ij}$, we obtain
\begin{equation*}
    (M_e)_{ij} = (\widehat{M}_e)_{ij} + \left(\frac{g_q}{\gamma_q^2} - \frac{1}{\gamma_q}\right) w_iw_j p_q(x_i)p_q(x_j).
\end{equation*}
This last expression may also be written as
\begin{equation*}
    M_e = \widehat{M}_e - \alpha \bm{v}\bm{v}^T,
\end{equation*}
with $\alpha=(\gamma_q-g_q)/\gamma_q^2$ and $v_i = w_i p_q(x_i)$, indicating that $M_e$ and $\widehat{M}_e$ differ by a rank-$1$ update. Moreover, since
\begin{equation*}
    g_q = \frac{2}{2q+1} < \frac{2}{q} = \gamma_q,
\end{equation*}
we deduce that $\alpha > 0$ and consequently
\begin{equation*}
    \widehat{M}_e = M_e + \alpha \bm{v}\bm{v}^T \succeq M_e.
\end{equation*}
Finally, since this holds for all elements in the mesh $\widehat{M} \succeq M$ also holds globally (see e.g. \cite[Lemma 3.22]{voet2025mass}) and consequently \cite[Corollary 3.6]{voet2023mathematical} implies that
\begin{equation*}
    \lambda_k(K,\widehat{M}) \leq \lambda_k(K,M) \quad k=1,\dots,n.
\end{equation*}
For the $d$-dimensional case, from the tensor product nature of the basis functions and the fact that $\eta_e$ is constant, we find that $M_e = M_{e,1} \otimes \dots \otimes M_{e,d}$, where $M_{e,j}=\widehat{M}_{e,j}-\alpha_j \bm{v}_j \bm{v}_j^T$ and $\widehat{M}_{e,j} = \diag(w_{0,j},\dots,w_{q_j,j})$ contains the weights along the $j$th parametric direction. Following the arguments in \cite[Theorem 3.26]{voet2023mathematical}, we conclude that $0 < \lambda_k(M_e,\widehat{M}_e) \leq 1$ for all $k=1,\dots,m$ with $m=\prod_{j=1}^d(q_j+1)$ and thanks to \cite[Corollary 3.6]{voet2023mathematical}, the result still holds.
\end{proof}

\begin{remark}
Extending the proof to a weighted inner product with a non-constant weight $\eta_e(x)$ does not appear straightforward since the explicit expressions of $g_q$ and $\gamma_q$ are required for deducing the sign of $\alpha$ and concluding that $\widehat{M}_e$ is a positive semidefinite low-rank update of $M_e$. Thus, we cannot conclude that the result holds for varying density functions or isoparametric elements.
\end{remark}

\section{Positivity of Demko weights on non-uniform meshes}
\label{se: non_uniform_meshes}
As we have seen in \Cref{tab: test_conditions}, negative Greville quadrature weights arise for spline spaces of intermediate smoothness on uniform meshes. This appendix collects a few counter-examples for non-uniform meshes. Although negative Greville weights were encountered in all cases, the Demko weights always remained positive, even for the most absurd knot vectors.
\begin{enumerate}[noitemsep]
    \item We first construct the knot vector for a uniform quadratic $C^1$ discretization of the unit line with $32$ subdivisions. The $10$th and $11$th knot values are then assigned to
    \begin{align*}
        \xi_{10} &= \xi_{9}+10^{-6}, \\
        \xi_{11} &= \xi_{9}+10^{-3},
    \end{align*}
    while keeping the other knot values unchanged. The resulting Greville weights are all positive, except $w_8$.
    \item We consider the same initial setup but with only $16$ subdivisions. The first knot span is uniformly refined with $100$ subdivisions. In this case, all weights are positive, except $w_{100}$.
    \item Finally, we consider a maximally smooth quartic discretization constructed from the non-uniform knot vector
    \begin{equation*}
        \Xi = 1/26(0,0,0,0,0,1,11,16,21,26,26,26,26,26).
    \end{equation*}
    Here, $w_1<0$ while all other weights are positive.
\end{enumerate}

The norm of the quadrature operator, its conjectured upper bound from \Cref{conj: quadrature_condition} and its provable upper bound from \cref{eq: quadrature_condition} are reported in \Cref{tab: norm_quad_greville} for up to four digits of accuracy.

\begin{table}[H]
\centering
\begin{tabular}{llll}
    \toprule
    Example & $\|Q\|$ & $\|c\|I(1)$ & $\|A^{-1}\|I(1)$ \\
    \midrule
    1 & 1.0015 & 1.6301 & 3.0898 \\
    2 & 1.0017 & 1.3854 & 2.5742 \\
    3 & 1.0201 & 1.3370 & 8.6653 \\
    \bottomrule
\end{tabular}
\caption{Norm of the quadrature operator $\|Q\|$, conjectured upper bound $\|c\|I(1)$ and provable upper bound $\|A^{-1}\|I(1)$ for the Greville quadrature and Examples 1-3.}
\label{tab: norm_quad_greville}
\end{table}

In contrast, the Demko quadrature weights were positive in all cases.

\section{High order explicit Runge Kutta methods}
\label{se: explicit_rk_methods}
This appendix briefly explains how to obtain the CFL constants for some of the less common high order time integration methods used in this work. Indeed, structural dynamics is less accustomed to high order techniques than it is to low order ones, commonly presented in standard textbooks \cite{hughes2012finite,bathe2006finite}. Although research for high order techniques in structural dynamics has lately gained momentum (see e.g. \cite{evans2018explicit,behnoudfar2022explicit,song2022high}), we will adopt explicit high order Runge Kutta (RK) methods for simplicity and ease of implementation. The material in this section is fairly standard and is only included for the sake of deriving the CFL conditions. For a broader overview of RK methods, interested readers may refer to \cite{quarteroni2010numerical,hairer1993solving,wanner1996solving} among many other references. Our point of departure is the semi-discrete problem \eqref{eq: semi_discrete_pb}, rewritten as an extended system of first order ODEs. By introducing the unknown $\bm{v}(t)=\dot{\bm{u}}(t)$, the semi-discrete problem becomes a coupled system of first order ODEs
\begin{align*}
\begin{split}
\dot{\bm{u}}(t) - \bm{v}(t) &= \bm{0}, \\
M\dot{\bm{v}}(t) + K\bm{u}(t) &= \bm{f}(t), \\
\bm{u}(0) &= \bm{u}_0,\\
\bm{v}(0) &= \bm{v}_0,
\end{split}
\end{align*}
which can be rewritten as a standard first order system
\begin{equation}
\label{eq: linearized_pb}
    \dot{\mathbf{u}}(t) = \mathcal{A}\mathbf{u}(t) + \mathbf{f}(t),
\end{equation}
where
\begin{equation*}
    \mathcal{A}=
    \begin{pmatrix}
        0 & I \\
        -M^{-1}K & 0
    \end{pmatrix},
    \quad \mathbf{u}(t)=
    \begin{pmatrix}
        \bm{u}(t) \\
        \bm{v}(t)
    \end{pmatrix},
    \quad \text{and} \quad \mathbf{f}(t)=
    \begin{pmatrix}
        \bm{0} \\
        M^{-1}\bm{f}(t)
    \end{pmatrix}.
\end{equation*}
The solution of \cref{eq: linearized_pb} can then be approximated with standard RK methods. These methods are stable if the eigenvalues of $\Delta t \mathcal{A}$ are within the (absolute) stability region
\begin{equation}
\label{eq: stability_region}
    \mathcal{S} = \{z \in \mathbb{C} \colon |R(z)| < 1\},
\end{equation}
where
\begin{equation*}
    R(z) = 1 + z\mathbf{b}(I-zA)^{-1}\mathbf{1}
\end{equation*}
is the stability function and 
\begin{equation*}
\begin{array}
{c|c}
\mathbf{c} & A \\
\hline
& \mathbf{b}
\end{array}
=
\begin{array}
{c|cccc}
c_1 & a_{11} & \hdots & a_{1s} \\
\vdots & \vdots & \ddots & \vdots \\
c_s & a_{s1} & \hdots & a_{ss} \\
\hline
& b_1 & \hdots & b_s
\end{array}
\end{equation*}
is the so-called Butcher Tableau containing the coefficients defining the RK method; see e.g. \cite{quarteroni2010numerical,hairer1993solving,wanner1996solving}. Note that $\mathbf{b}$ is a row vector whereas $\mathbf{1}$ is a column vector of all ones. For explicit RK methods, $A$ is strictly lower triangular and the stability function reduces to $R(z) = \det(I-z(A-\mathbf{1}\mathbf{b}))$, a polynomial function in $z$ \cite{quarteroni2010numerical}. In the sequel, RK$(s,p)$ denotes an explicit RK method with $s$ stages and order $p$ (with $s \geq p$). For $p \leq 4$, explicit RK methods with $s=p$ stages are known and their stability function is given by
\begin{equation*}
    R(z) = \sum_{k=0}^s \frac{z^k}{k!}.
\end{equation*}
Therefore, for $p \leq 4$, two different methods with the same order yield the same stability region. However, for $p \geq 5$, $s>p$ \cite[Theorem 5.1]{hairer1993solving} and different methods generally have different stability functions and therefore different stability regions. For choosing a suitable method, we must know where the eigenvalues of $\mathcal{A}$ lie. The result is given in the next lemma.

\begin{lemma}
If $M$ is positive definite,
\begin{equation*}
    \Lambda(\mathcal{A})=\{ \pm i \sqrt{\lambda} \colon \lambda \in \Lambda(K,M)\}.
\end{equation*}
\end{lemma}
\begin{proof}
If $M$ is positive definite, there exists an invertible matrix of $M$-orthonormal eigenvectors $U$ and a diagonal matrix of eigenvalues $D$ such that $U^TMU=I$ and $U^TKU=D=\diag(\lambda_1,\dots,\lambda_n)$ \cite[Theorem VI.1.15]{stewart1990matrix}. Equivalently, $M^{-1}K=UDU^{-1}$ is the spectral decomposition of $M^{-1}K$. From the factorization
\begin{equation*}
\mathcal{A}=
    \begin{pmatrix}
        0 & I \\
        -M^{-1}K & 0
    \end{pmatrix}
    =
    \begin{pmatrix}
        U & 0 \\
        0 & U
    \end{pmatrix}
    \begin{pmatrix}
        0 & I \\
        -D & 0
    \end{pmatrix}
    \begin{pmatrix}
        U^{-1} & 0 \\
        0 & U^{-1}
    \end{pmatrix},
\end{equation*}
we deduce that $\mathcal{A}$ is similar to 
\begin{equation*}
    \begin{pmatrix}
        0 & I \\
        -D & 0
    \end{pmatrix},
\end{equation*}
which is itself (unitarily) similar to the block-diagonal matrix $\mathcal{D}=\diag(D_1,\dots,D_n)$ \cite{golub2003solving}, where
\begin{equation*}
    D_j=
    \begin{pmatrix}
        0 & 1 \\
        -\lambda_j & 0
    \end{pmatrix}.
\end{equation*}
Since $\Lambda(D_j)=\{-i\sqrt{\lambda_j},i\sqrt{\lambda_j}\}$, we finally deduce that
\begin{equation*}
    \Lambda(\mathcal{A})=\{ \pm i \sqrt{\lambda} \colon \lambda \in \Lambda(K,M)\}.
\end{equation*}
\end{proof}
Since the eigenvalues of $\mathcal{A}$ are purely imaginary, we must select methods that maximize the stability region along the imaginary axis. Stabilized RK methods have extensively been studied in the literature \cite{wanner1996solving,van1996development,doehring2024many} but enlarging the stability along the imaginary axis usually comes at the pricing of sacrificing some accuracy. Thus, we will merely select classical high order methods that have satisfactory stability along the imaginary axis. For $p \leq 4$, this does not leave any choice, and in fact the stability region of classical explicit RK$(2,2)$ schemes does not include any of the imaginary axis at all, but fortunately classical RK$(3,3)$ and RK$(4,4)$ include a reasonable portion of the imaginary axis (see \Cref{fig: stability_regions}). However, for $p \geq 5$, choosing the right method becomes important. Among high order explicit methods is a well-known RK$(6,5)$ method, originally proposed by Kutta and corrected by Nyström \cite{butcher1996history}, which includes a relatively large portion of the imaginary axis. Unfortunately, the same cannot be said for some of Butcher's RK$(7,6)$ methods \cite{butcher1964runge}. The stability region of those methods barely intersects the imaginary axis and we instead chose Verner's RK$(8,6)$ method \cite{hairer1993solving}, thereby trading some efficiency for greater stability. The Butcher tableaux of the high order RK methods used in this work are found among the references cited and their stability regions are shown in \Cref{fig: stability_regions}. From those regions, we deduce the step size restriction
\begin{equation*}
    \Delta t \leq \frac{C}{\sqrt{\lambda_n(K,M)}}
\end{equation*}
where the \emph{linear stability imaginary axis inclusion} \cite{macdonald2003constructing}
\begin{equation*}
    C = \sup \{y \colon iy \in \mathcal{S}\}
\end{equation*}
is computed numerically. There certainly exist better choices of explicit high order integration techniques in structural dynamics and we think this problem deserves greater attention. 

\begin{figure}[H]
    \centering
    \includegraphics[scale=0.5]{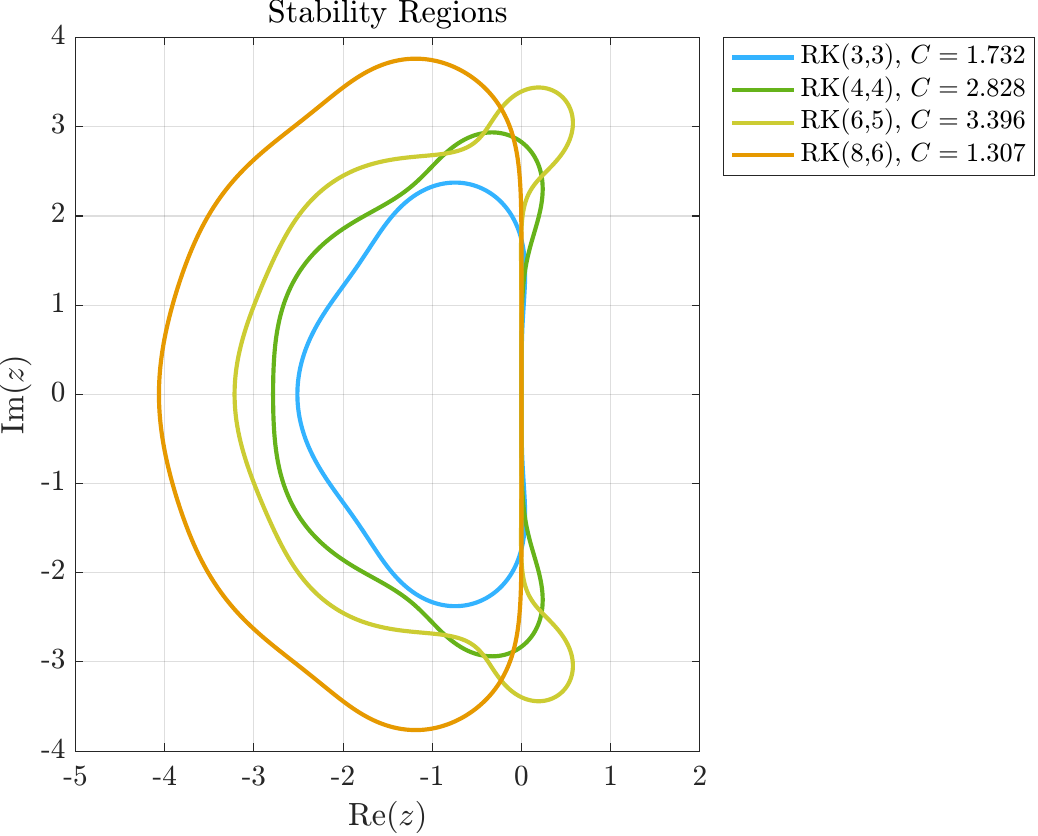}
    \caption{Stability regions of some explicit RK methods}
    \label{fig: stability_regions}
\end{figure}

\end{appendices}

\end{document}